%% file: main.tex
\documentclass[a4paper, 11pt]{article}

\include{Basics}

\definecolor{darkgreen}{rgb}{0.0, 0.5, 0.13}

\newcommand{\prob}[1]{{\mathbb P}(#1)}

\title{
Tight tail probability bounds for distribution-free decision making}

\author{Ernst Roos\thanks{Corresponding author: e.j.roos@tilburguniversity.edu}, Ruud Brekelmans, Wouter van Eekelen \\
\small Department of Econometrics and Operations Research, Tilburg University  \\ 
Dick den Hertog \\
\small University of Amsterdam \\
Johan van Leeuwaarden \\
\small Department of Econometrics and Operations Research, Tilburg University }

\linedabstractkw{
Chebyshev's inequality provides an upper bound on the tail probability of a random variable based on its mean and variance. While tight, the inequality has been criticized for only being attained by pathological distributions that abuse the unboundedness of the underlying support and are not considered realistic in many applications. 
We provide alternative tight lower and upper bounds on the tail probability given a bounded support, mean and mean absolute deviation of the random variable. We obtain these bounds as exact solutions to semi-infinite linear programs. We leverage the bounds for distribution-free analysis of the newsvendor model, monopolistic pricing, and stop-loss reinsurance. We also exploit the bounds for safe approximations of sums of correlated random variables, and to find convex reformulations of single and joint ambiguous chance constraints that are ubiquitous in distributionally robust optimization. 
}{Chebyshev inequality, probability bounds, distributionally robust optimization, chance constraints}

\Template{Roos et al.}{Tight tail probability bounds for distribution-free decision making}{}

\begin{document}
\maketitle


\section{Introduction} \label{sec: intro}

 Chebyshev's inequality  provides an upper bound on the tail probability of a random variable using only its first two moments~\citep{bienayme1853considerations,chebyshev1867valeurs}. 
 Due to this distribution-free nature, Chebyshev's inequality is widely applicable. 
 Let the ambiguity set $\mathcal{P}_{(\mu,\sigma)}$ contain all distributions with a given mean $\mu$ and variance $\sigma^2$, and let the random variable $X$ follows some distribution $\mathbb{P} \in \mathcal{P}_{(\mu,\sigma)}$. Chebyshev's inequality (the one-sided version also known as Cantelli's inequality) then follows from the worst-case distribution that solves the optimization problem
\begin{equation}\label{}
\bP \left( X \geq t \right)\leq \sup\limits_{X\sim\mathbb{P} \in \mathcal{P}_{(\mu,\sigma)}}\bE_{\mathbb{P}} \left[\1\{ X\geq t \}\right]  = \frac{\sigma^2}{\sigma^2 + \left(t - \mu\right)^2}.
\end{equation}
This inequality is tight, meaning it cannot be improved in general. However, Chebyshev's inequality can be criticized for only being attained by pathological distributions that abuse the unboundedness of the underlying support. Indeed, the worst-case distribution takes only values on the points $\mu-\sigma^2/(t-\mu)$ and $t$ (with probabilities $(t-\mu)^2/(\sigma^2+(t-\mu)^2)$ and $\sigma^2/(\sigma^2+(t-\mu)^2)$, resp.), which can be regarded unrealistic \citep{vanparys2016generalized}. In many practical applications some information on the minimum and maximum of uncertain parameters is known. This is particularly true for OR applications that consider uncertain parameters that are known to be nonnegative, such as inventory management, service operations, appointment scheduling and pricing mechanisms. We remark that a tight tail probability bound under knowledge of the mean, variance and a bounded support was derived by \citet{schepper1995general}. 
Next to restricting the support, a second potential improvement of Chebychev's inequality concerns robustness for outliers. Whereas the (sample) variance is greatly influenced by outliers, the mean absolute deviation (MAD) is less sensitive for large deviations from the mean, and hence a potentially more robust measure of statistical dispersion in data. We therefore propose to replace variance with mean absolute deviation (MAD). Using the MAD comes with additional advantages. We show that the set of extremal distributions for which the derived tail bounds are tight is more varied than a single pathological distribution: it consists of an infinite number of mixed distributions instead. Second, because the MAD is a linear function, it allows for elegant closed-form bounds, a feature we shall leverage when applying the bounds to domain-specific OR questions. 

In obtaining the robust tail bounds, we need to solve \begin{equation}\label{optp2}
\sup\limits_{X\sim \mathbb{P} \in \mathcal{P}_{(\mu , b , d)}}
\bE_{\mathbb{P}} \left[\1\{ X\geq t \}\right],
\end{equation}
with $\mathcal{P}_{(\mu , b , d)}$  the ambiguity set that contains all distributions with a given mean $\mu$, support $[0,b]$ and mean absolute deviation $d$. Optimization problem \eqref{optp2} is a semi-infinite linear optimization problem (LP) that is reminiscent of those arising in moment problems, and typically does not allow for an analytic (closed-form) solution. 
Using the MAD based ambiguity set $\mathcal{P}_{(\mu , b , d)}$, the dual program to \eqref{optp2} can be solved explicitly. 
While comparable dual programs are often solvable as semidefinite or second-order conic programs (see, e.g., \citet{xin2013time,natarajan2007mean,perakis2008regret,semivariance,das2018heavy}), analytic solutions as in our case are typically hard to attain. 

The solution of \eqref{optp2} gives 
a generic tight upper bound on the tail probability of all random variables with a given bounded support, mean and MAD. This new robust bound is of a similar simplicity and generality as the original Chebyshev inequality, and can therefore be used widely in various applications. The worst-case distribution that solves \eqref{optp2}, is however more complicated than the two-point distribution of the Chebyshev inequality, and is a mixed distribution with up to three discrete parts and one continuous part. We also derive three more tail probability bounds: the tight lower bound under $\mathcal{P}_{(\mu , b , d)}$ ambiguity and the tight upper and lower bounds under $\mathcal{P}_{(\mu , b , d,\beta)}$ ambiguity, where we also condition on the skewness $\bP \left(X\geq \mu\right)=\beta$. 

Recent advances in Distributionally Robust Optimization (DRO) also exploit ambiguity sets in terms of bounded support, mean and MAD to obtain closed-form expressions for stochastic quantities such as the minimum and maximum expectation of a convex function~\cite{postek2018robust,ghosal2020distributionally}. 
These closed-form expressions are then used to solve minmax and maxmin optimization problems that arise naturally in decision making under uncertainty.  \cite{postek2018robust} specifically use results from \cite{ben1972more} on tight upper and lower bounds on the expectation of {\it convex} function of a random variable. 
This paper presents the first closed-form solution for the combination of  $\mathcal{P}_{(\mu , b , d)}$ (or $\mathcal{P}_{(\mu , b , d,\beta)}$) constraints and a {\it non-convex} objective function. This proof method is not restricted to the indicator function, and could potentially work for a much larger class of  (measurable) functions. 




The first part of this paper revolves around the tail probability bounds. After the primal-dual proofs we conclude that part with extensive numerical demonstrations and a comparison with other classical bounds. We expect the bounds to be useful in many domains. 
The second part deals specifically with the use of these new bounds in the OR domain. Although numerical aspects are important, the focus is on utilizing the structural properties of the closed-form inequalities. For optimization problems that have tail probabilities as input, this can lead to closed-form or tractable solutions, which would remain out of reach without the derived tight inequalities. 

We first apply the robust bounds for distribution-free analysis of three classical models that can be subjected to minmax or maxmin optimization. We start with the {\it newsvendor model}, the basic single-period inventory model that searches for the optimal order quantity in view of overage and underage costs. Under full information, the optimal order quantity corresponds to a specific quantile of the demand. \cite{Scarf1958} studied the situation when only the mean and variance of demand are known, and derived a robust order quantity as the solution of a {minmax} optimization problem, where the decision maker takes the best decision under the worst possible circumstances in light of mean-variance ambiguity. Scarf's distribution-free analysis is one of the first forms of DRO, and has been a source of inspiration for many OR studies. 
Technically, it requires computing upper bounds via a linear program on the expected value of a convex function $\mathbb{E}[h(X)]$ for a random variable $X$ with mean $\mu$ and variance $\sigma^2$.  We shall apply our tail probability bounds for a comparable analysis, with $\mathcal{P}_{(\mu , b , d,\beta)}$ ambiguity. 

We then turn to the {\it monopolistic pricing problem}, where a seller seeks to maximize profit when selling a single object to a buyer who is willing to pay some unknown value $X$. Traditionally, it is assumed that $X$ is drawn from some distribution that is known to the seller, so that the seller can set the optimal price. When there is a single buyer, the optimal strategy is to post a fixed price $p$ that maximizes the expected profit $p\prob{X> p}$; see \cite{riley1983optimal} and \cite{myerson1981optimal}. 
We apply the derived tail probability bounds to the robust variant of the monopolistic pricing problem, where instead of knowing the distribution, the seller only has partial information contained in $\mathcal{P}_{(\mu , b , d)}$. The seller then becomes a {maxmin} decision maker who chooses the price that maximizes the worst-case expected profit. 

The third classical model occurs in {\it stop-loss reinsurance}. An insurance company faces a claim of size $X$, which it pays up to a predefined level $z$, while the reinsurance company covers the remainder (up to a predefined maximum $m$). We study this problem from both the insurer's and reinsurer's perspective, the latter of which requires an extension of our tail probability bound. Specifically, we derive an upper bound for the expected payment of the reinsurer, which is neither a convex or indicator type function. 
The three classical models come from different corners of Operations Research and Management Science, but have in common that important characteristics can be expressed in distribution functions, which facilitates a direct application of the tail probability bounds for distribution-free analysis. The models also show that the bounds lend themselves to both minmax and maxmin decision problems. 
We emphasize that the models have been chosen somewhat arbitrarily, and  
there are many other OR questions where tail probability bounds under mean-MAD constraints can prove useful. 

As alluded to above, our bounds are part of a 
larger research effort that deals with exploiting the tractability of mean-MAD constraints for enhancing the state-of-affairs in DRO. 
To highlight the connection of our work with DRO, we extend our tail probability bound to sums of random variables. We demonstrate the multivariate tail bound with a risk management example that considers portfolios with multiple risky assets. The bound allows for arbitrary dependence structures and hence is particularly suitable for, e.g., credit risk and insurance problems.
Additionally, we apply our tail probability bounds to find convex reformulations of several types of ambiguous chance constraints containing a single random variable. Specifically, we consider a general convex constraint with right-hand side uncertainty, and a constraint that is bilinear in the decision variable and the uncertain parameter. This allows optimization problems with such ambiguous chance constraints to be solved through conventional solution methods from continuous optimization under mean-MAD ambiguity.
These theoretical results on ambiguous chance constraints are illustrated through an example from radiotherapy optimization, in which the dose of radiation delivered to the tumor is to be maximized, under a probabilistic constraint on the dose of radiation delivered to the surrounding healthy tissue.

{\it Outline and contributions.}~This introduction largely revolves around bounds of tail probabilities through mean-MAD ambiguity 
and their applications in probability theory, stochastic OR and optimization. For each of these applications, we will discuss more details and related studies in the appropriate sections. 
 In Section~\ref{sec: probs} we present, prove and illustrate the tail probability inequalities. We provide tight upper and lower bounds for the probability that a random variable exceeds a specified threshold under a known support, mean and mean absolute deviation. 
In Section~\ref{sec: applications} we use these bounds for distribution-free analysis of classic OR problems. Specifically, we study the newsvendor model in Section \ref{sec: newsvendor}, the monopoly pricing model in Section \ref{sec: pricing}, and the stop-loss reinsurance model in Section \ref{sec: reinsurance}. 
To further demonstrate the large scope of our bounds, and to highlight the connection with DRO, we present in Section \ref{sec: applications2} two more perspectives.
We extend our tail probability bound to sums of random variables in Section \ref{sec: sums} and illustrate this extension in Section~\ref{sec: insuranceexample} with an insurance problem that considers a portfolio with multiple risks. Section \ref{sec: acc} discusses the application of tail probability bounds to reformulate ambiguous chance constraints in distributionally robust optimization. This application is illustrated through an example from radiotherapy optimization in Section \ref{sec: radiotherapy}.

\section{Novel tail probability bounds} \label{sec: probs}

In this section we derive novel bounds for the probability $\mathbb{P}(X\geq t)$ that a random variable $X$ with given support, mean and MAD exceeds $t$. 
We obtain the bounds by solving the semi-infinite linear program
\begin{equation}\label{eq: mean mad ambiguity LP}
    \sup_{\mathbb{P}\in \mathcal{P}_{(\mu, b, d)}}  \int_x \1_{\{x\geq t\}}{\rm d} \mathbb{P}(x),
\end{equation}
where we maximize over a set of probability measures with the stated characteristics, i.e., 
\begin{equation} \label{eq: mean mad ambiguity}
    \cP_{(\mu, b, d)} = \{\bP:\mathcal{B}\to[0,1]\mid \bP(X\in[0,b])=1,\, \bE_{\bP}[X]=\mu,\, \bE_{\bP}[|X-\mu|]=d\}
\end{equation}
with $\mathcal{B}$ the Borel $\sigma$-algebra of the closed set $[0,b]$, and $\mu, b, d\in\bR_+$ are parameters that describe all known properties of the distribution. 
We solve the linear programs and present the novel bounds in Section~\ref{sec:novelbounds}. We then compare the novel bounds with some existing bounds in Section~\ref{sec:compare}, and briefly discuss the existing literature on generalized versions of Chebyshev's inequality. 

\subsection{Tight lower and upper bounds}\label{sec:novelbounds}
 Since $\bP$ is a probability measure it should satisfy the constraint $\int_{x\in[0,b]}{\rm d}\bP(x)=1$. Moreover, this probability measure should satisfy the mean and MAD constraints $\int_{x\in[0,b]}x{\rm d}\bP(x)=\mu$ and $\int_{x\in[0,b]}|x-\mu|{\rm d}\bP(x)=d$. Under these constraints, we solve 
the semi-infinite linear program \eqref{eq: mean mad ambiguity LP}, which gives our first main result. 
\begin{theorem}\label{theorem:main}
Consider a random variable $X$ with a distribution $\mathbb{P}$ in $\mathcal{P}_{(\mu, b, d)}$. Then, 
\begin{equation} \label{eq: worst-case prob}
    \sup\limits_{\mathbb{P}\in\mathcal{P}_{(\mu, b, d)}}\mathbb{P}(X\geq t)=\sup\limits_{\mathbb{P}\in\mathcal{P}_{(\mu, b, d)}}\mathbb{P}(X> t)=
    \begin{cases}
     1, \quad &  t\in[0,\tau_1], \\
     \frac{\mu}{t}-\frac{d(b-t)}{2t(b-\mu)}, \quad &  t\in[\tau_1,\mu], \\
     1-\frac{d}{2\mu}, &  t\in[\mu,\tau_2], \\
     \frac{d}{2(t-\mu)}, &  t\in[\tau_2,b],
    \end{cases}
\end{equation}
with $\tau_1$ and $\tau_2$ given by
\begin{equation*}
    \tau_1=\mu-\frac{d(b-\mu)}{2(b-\mu)-d},\quad \tau_2=\mu+\frac{d\mu}{2\mu-d}.
\end{equation*}
\end{theorem}
\begin{proof}
Let $\mathcal{M}^+$ be the set of non-negative measures defined on the measurable space $([0,b],\mathcal{B})$. We need to solve
\begin{equation}\label{eq:primal1}
\begin{aligned}
&\! \sup_{\mathbb{P}\in \mathcal{M}^+} &  &\int_x \1_{\{x\geq t\}}{\rm d} \mathbb{P}(x)\\
&\text{s.t.} &      &  \int_x {\rm d}\mathbb{P}(x)=1, \ \int_x x{\rm d}\mathbb{P}(x)=\mu,\ \int_x |x-\mu|{\rm d}\mathbb{P}(x)=d.
\end{aligned}
\end{equation}
A useful fact is that the semi-infinite LP \eqref{eq:primal1} can be reduced to an equivalent finite LP that yields the same optimal value. In particular, when certain Slater conditions hold for the moment constraints (i.e., the moment vector should lie in the interior of the set of feasible moments) then solving the primal semi-infinite LP is equivalent to solving its finite dual counterpart; see, e.g., \citet{Isii1962} or \citet{popescu2005semidefinite}.
Moreover, the Richter-Rogosinski Theorem (see, e.g.,~\citet{rogosinski1958moments,shapiro2009lectures}, or \citet{han2015convex}) states that there exists an extremal distribution for problem  \eqref{eq:primal1} with at most three support points. While finding these points in closed form is typically not possible for general semi-infinite problems, we next show that this is possible for the problem at hand by resorting to the dual problem and exploiting the specific shape of the dual constraints that is imposed by the MAD constraint $\int_x |x-\mu|{\rm d}\mathbb{P}(x)=d$.

Consider the dual of \eqref{eq:primal1}, 
\begin{equation}\label{eq:dual1}
\begin{aligned}
&\inf_{\lambda_0,\lambda_1,\lambda_2} &  &\lambda_0 + \lambda_1 \mu+\lambda_2 d\\
&\text{s.t.} &      & \1_{\{x\geq t\}}\leq \lambda_0  +\lambda_1 x+\lambda_2 |x-\mu| \eqqcolon F(x), \ \forall x\in[0,b].
\end{aligned}
\end{equation}
The constraint of the dual problem requires $F(x)$ to majorize $\1\{x\geq t\}$. Note that $F(x)$ has a 'kink' at $x=\mu$, that is, $F(x)$ is piecewise linear and can only change direction in $x=\mu$. Solving \eqref{eq:dual1} boils down to finding the tightest majorant. We have four candidates for the solution, which are depicted in Figure~\ref{fig:majors1}. When $t\in[0,\mu]$, $F(x)$ touches $\1\{x\geq t\}$ in $\{0,t,b\}$ (scenario 1a), or $F(x)=1$ and touches $\1\{x\geq t\}$ in $[t, b]$ (scenario 1b). When $t\in[\mu,b]$, $F(x)$ touches in $[0, \mu] \cup \{t\}$ (scenario 2a) or in $\{0\} \cup [t, b]$ (scenario 2b).

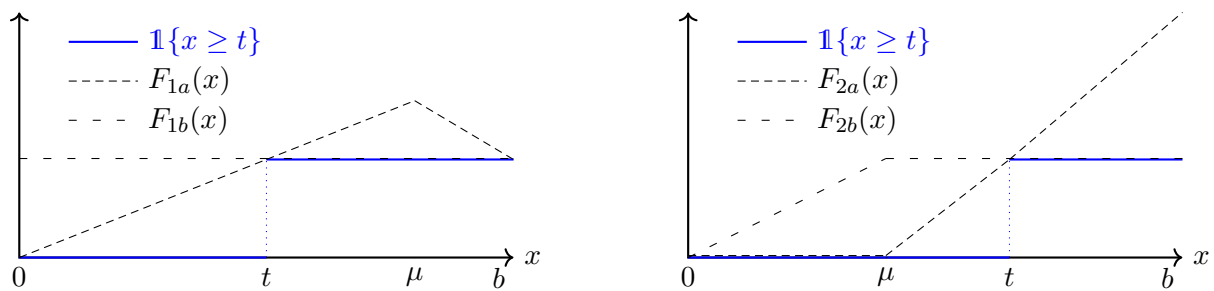
\begin{figure}[h!]
\begin{center}
\begin{minipage}[b]{0.45\linewidth}
\begin{tikzpicture}[scale=1.3]
\draw [<->,thick] (0,2.5) node (yaxis) [above] {}
        |- (5,0) node (xaxis) [right] {$x$};
        \draw[blue,thick] (0.5,2.2) -- (1.2,2.2) node[right] {$\1\{x\geq t\}$};
        \draw[densely dashed] (0.5,1.8) -- (1.2,1.8) node[right] {$F_{1a}(x)$};
        \draw[loosely dashed] (0.5,1.4) -- (1.2,1.4) node[right] {$F_{1b}(x)$};
        \coordinate (mu+) at (4,1);
        \coordinate (mu-) at (4,0);
        \coordinate (mu--) at (4,0.01);
        \coordinate (z) at (0, 0);
        \coordinate (z--) at (0.01,0.01);
        \coordinate (b) at (5,1);
        \coordinate (b+) at (5,2.5);
        \coordinate (t+) at (2.5,1);
        \coordinate (t-) at (2.5,0);
        \draw[blue, thick] (z) -- (t-);
        \draw[blue, dotted] (t-) -- (t+);
        \draw[blue, thick] (t+) -- (b);
        \draw[densely dashed] (z) -- (t+);
        \draw[densely dashed] (t+) -- (4,1.6);
        \draw[densely dashed] (4,1.6) -- (5,1);
        \draw[loosely dashed] (0,1.01) -- (5,1.01);
        \draw[dotted] (mu-) -- (mu-) node[below] {$\mu$};
        \draw[dotted] (t-) -- (t-) node[below] {$t$};
        \draw[dotted] (0,0) -- (0,0) node[below] {$0$};
        \draw[dotted] (4.85,0) --(4.85,0) node[below] {$b$};

\end{tikzpicture}
\end{minipage}
\hfill
\begin{minipage}[b]{0.45\linewidth}
\begin{tikzpicture}[scale=1.3]
\draw [<->,thick] (0,2.5) node (yaxis) [above] {}
        |- (5,0) node (xaxis) [right] {$x$};
        \draw[blue,thick] (0.5,2.2) -- (1.2,2.2) node[right] {$\1\{x\geq t\}$};
        \draw[densely dashed] (0.5,1.8) -- (1.2,1.8) node[right] {$F_{2a}(x)$};
        \draw[loosely dashed] (0.5,1.4) -- (1.2,1.4) node[right] {$F_{2b}(x)$};
        \coordinate (mu+) at (2,1);
        \coordinate (mu-) at (2,0);
        \coordinate (mu--) at (2,0.02);
        \coordinate (z) at (0, 0);
        \coordinate (z--) at (0,0.02);
        \coordinate (b) at (5,1);
        \coordinate (b+) at (5,2.5);
        \coordinate (t+) at (3.25,1);
        \coordinate (t-) at (3.25,0);
        \draw[blue, thick] (z) -- (t-);
        \draw[blue, dotted] (t-) -- (t+);
        \draw[blue, thick] (t+) -- (b);
        \draw[densely dashed] (z--) -- (mu--);
        \draw[densely dashed] (mu--) -- (t+);
        \draw[densely dashed] (t+) -- (b+);
        \draw[loosely dashed] (0,0) -- (2,1.01);
        \draw[loosely dashed] (2,1.01) -- (5,1.01);
        \draw[dotted] (mu-) -- (mu-) node[below] {$\mu$};
        \draw[dotted] (t-) -- (t-) node[below] {$t$};
        \draw[dotted] (0,0) -- (0,0) node[below] {$0$};
        \draw[dotted] (4.85,0) --(4.85,0) node[below] {$b$};
        
\end{tikzpicture}
\end{minipage}
\caption{Scenario 1 and the majorizing functions $F_{1a}(x)$ and $F_{1b}(x)$ under scenarios 1a and 1b, respectively. Scenario 2 and the majorizing functions $F_{2a}(x)$ and $F_{2b}(x)$ under scenarios 2a and 2b, respectively.}\label{fig:majors1}
\end{center}
\end{figure}

Scenario 1a implies $F(0)=0, \, F(t)=F(b)=1,$ which gives dual solution
\begin{equation}
    \lambda_2=-\frac{b-t}{2t(b-\mu)}, \quad \lambda_1=\frac{1}{t}+\lambda_2, \quad \lambda_0=-\lambda_2\mu,
\end{equation}
and objective value
\begin{equation}
    \lambda_0+\lambda_1\mu + \lambda_2 d = \frac{\mu}{t}-\frac{d(b-t)}{2t(b-\mu)}.
\end{equation}
The next step of our proof is to find a feasible solution for the primal problem which yields the same objective value as the solution to the dual problem. By weak duality of semi-infinite linear programming, we know that a feasible solution to the dual problem provides us with a valid upper bound for the optimal primal solution value. Now finding a feasible primal solution with an objective value equal to this upper bound results in strong duality. Next, we will provide a constructive approach for finding such a primal solution. Assume that we have strong duality. The primal maximizer $\bP^*$ and the dual minimizer $(\lambda_0^*,\lambda_1^*,\lambda_2^*)$ are then related as
\begin{equation}\label{eq: compslack}
    \int _x\1\{x\geq t\}{\rm d}\mathbb{P}^*(x) = \int _x(\lambda^*_0+\lambda^*_1x + \lambda^*_2 |x-\mu|){\rm d}\mathbb{P}^*(x).
\end{equation}
Moreover, due to dual feasibility we must have that $\lambda^*_0+\lambda^*_1\mu + \lambda^*_2 d - \1\{x\geq t\}\geq0$ pointwise for each $x\in[0,b]$. This inequality combined with equation \eqref{eq: compslack} is also known as the complementary slackness relation in (semi-infinite) linear programming. An immediate consequence of complementary slackness is that the worst-case probability distribution should be supported on the points where the dual solution function $F^*(x)=\lambda^*_0+\lambda^*_1x + \lambda^*_2 |x-\mu|$ coincides with the indicator function $\1\{x\geq t\}$. For scenario 1a we have one (unique) option, that is, a discrete probability distribution with probability masses on the elements of the set $\{0,t,b\}$. The corresponding optimal probabilities of \eqref{eq:primal1} follow from solving 
\begin{equation}
    p_0+p_t+p_b=1, \quad p_t t+p_b b=\mu, \quad p_0 \mu + p_t(\mu-t)+p_b(b-\mu)=d.
\end{equation}
This gives
\begin{equation}
    p_b=\frac{d}{2(b-\mu)}, \quad p_t=\frac{\mu}{t}-\frac{bd}{2t(b-\mu)},
\end{equation}
and hence
\begin{equation}
    \int_x\1\{x\geq t\}{\rm d}\mathbb{P}(x)=p_t+p_b=\frac{\mu}{t}-\frac{d(b-t)}{2t(b-\mu)}.
\end{equation}
Since, by weak duality of semi-infinite linear programming, we have strong duality as both the primal and dual objective value are the same, these are the optimal solutions.

Scenario 1b implies $F(0)=F(t)=F(b)=1$ and hence $\lambda_0=1,\,\lambda_1=\lambda_2=0$ with objective value 1. One feasible primal solution is $p_b=\frac{\mu-t}{b-t},\,p_t=1-p_b$, with objective 1. Note that this primal solution is not a unique optimum, as the dual solution function $F_{1b}^*(x)$ coincides with $\1\{x\geq t\}$ on the entire interval $[t,b]$. Therefore, one could construct an arbitrary (discrete, continuous or mixed) probability distribution with support on the interval $[t,b]$, which then serves as the worst-case distribution, as long as the mean and MAD conditions are satisfied.

Scenario 2a implies $F(0)=F(\mu)=0, \, F(t)=1$, which gives
\begin{equation}
    \lambda_1=\lambda_2=\frac{1}{2(t-\mu)}, \quad \lambda_0=-\frac{\mu}{2(t-\mu)},
\end{equation}
and objective value
\begin{equation}
    \lambda_0+\lambda_1\mu + \lambda_2 d=\frac{d}{2(t-\mu)}.
\end{equation}
Solving the optimal probabilities of \eqref{eq:primal1}, where we take $\{0,\mu,t\}$ for the support of the worst-case distribution, indeed confirms that $p_t=\frac{d}{2(t-\mu)}$.

Scenario 2b gives $F(0)=0, \, F(\mu)=F(b)=1$, which results in
\begin{equation}
    \lambda_0=\frac12, \quad \lambda_1=\frac{1}{2\mu}, \quad \lambda_2=-\frac{1}{2\mu}
\end{equation}
and dual objective value
\begin{equation}
    \lambda_0+\lambda_1\mu+\lambda_2d=1-\frac{d}{2\mu}.
\end{equation}
Solving \eqref{eq:primal1} with support $\{0,t,b\}$ confirms that $p_0=\frac{d}{2\mu}$.

The proof is then completed by looking which scenario prevails on a specific interval, and these intervals can be determined by simply equating the minimum objective values and thereafter solving the resulting equations with respect to $t$ to find $\tau_1$ and $\tau_2$ for, respectively, scenario 1 and 2. 

We remark that the proof is identical for the strict inequality. Because the majorant is a continuous function, it is irrelevant whether the indicator function that is majorized is lower or upper semi-continuous.
\end{proof}

We mention some noteworthy characteristics of the bound in Theorem~\ref{theorem:main}. The bound is continuous in $t=\mu$.
If the support is symmetric around $\mu$, then the worst-case probability is at least $1/2$ for $t\in[0,\mu]$. 
The upper bound for $t\in[\mu,b]$ is increasing for $d\leq{2\mu(t-\mu)}/{t}$ and decreasing for larger values of $d$.
This last observation in particular is interesting as one might anticipate the bound to increase with MAD. This also implies that when MAD is unknown, the worst-case probability based on only the support and mean is given by the result of Theorem~\ref{theorem:main} for $d={2\mu(t-\mu)}/{t}$. This indeed returns Markov's inequality. We also mention that the support information $[0,b]$ can easily be extended to $[a,b]$ with $a\in\mathbb{R}$ by shifting the distribution accordingly. The tail bounds for the second and third interval then change into
\begin{equation}
    \frac{\mu-a}{t-a} - \frac{d(b-t)}{2(t-a)(b-\mu)} \quad {\rm and} \quad  1-\frac{d}{2(\mu-a)},
\end{equation}
respectively.

For a tight lower bound on $\bP(X>t)$, we can use the results and the remark above on a slightly altered version of the input. The idea is formalized in the following theorem:

\begin{theorem}\label{theorem:inf}
Consider a random variable $X$ with a distribution $\mathbb{P}$ in $\mathcal{P}_{(\mu, b, d)}$. Then,
\begin{equation}\label{eq:thm2}
    \inf\limits_{\mathbb{P}\in\mathcal{P}_{(\mu, b, d)}}\mathbb{P}(X\geq t)=\inf\limits_{\mathbb{P}\in\mathcal{P}_{(\mu, b, d)}}\mathbb{P}(X> t)=
    \begin{cases}
     1-\frac{d}{2(\mu-t)}, \quad &  t\in[0,\tau_1], \\
     \frac{d}{2(b-\mu)}, \quad &  t\in[\tau_1,\mu], \\
     \frac{\mu-t}{b-t}+\frac{dt}{2\mu(b-t)}, &  t\in[\mu,\tau_2], \\
     0, &  t\in[\tau_2,b]
    \end{cases}
\end{equation}
with $\tau_1$ and $\tau_2$ given by
\begin{equation*}
    \tau_1=\mu-\frac{d(b-\mu)}{2(b-\mu)-d},\quad \tau_2=\mu+\frac{d\mu}{2\mu-d}.
\end{equation*}
\end{theorem}
\begin{proof}
We reformulate the infimum as follows:
\begin{align*}
    \inf\limits_{\mathbb{P}\in\mathcal{P}_{(\mu, b, d)}}\mathbb{P}(X> t)& = 1 - \sup\limits_{\mathbb{P}\in\mathcal{P}_{(\mu, b, d)}}\mathbb{P}(X\leq t) \\
    &= 1 - \sup\limits_{\mathbb{P}\in\tilde{\mathcal{P}}_{(\mu, \tilde{b}, \tilde{d})}}\mathbb{P}(X\geq \tilde{t}),
\end{align*}
where
\begin{equation*}
    \tilde{\mathcal{P}}_{(\mu, \tilde{b}, \tilde{d})}=\{\mathbb{P}:\mathcal{B}\to[0,1] \mid  \mathbb{P}(X\in[\tilde{a},\tilde{b}])=1,\:\mathbb{E}_{\mathbb{P}}[X]=\mu,\: \mathbb{E}_{\mathbb{P}}[|X-\mu|]=d\},
\end{equation*}
and $\tilde{a}=2\mu-b$, $\tilde{b}=2\mu$ and $\tilde{t}=2\mu-t$. Plugging in the results from Theorem~\ref{theorem:main} for $t\in(a,b]$ then yields \eqref{eq:thm2}. Similarly, the result for $\inf_{\mathbb{P}\in\mathcal{P}}\mathbb{P}(X\geq t)$ can be obtained.
\end{proof}
We now describe in more detail the worst-case distributions found that are revealed in the proof of Theorem~\ref{theorem:main}. 

\begin{proposition}\label{prop: worst-cases}
Consider the set of worst-case distributions
$\cP^* = \arg\sup_{\bP \in \cP_{(\mu,b,d)}} \bE_{\bP}[\1\{X \geq t\}]$. Then, 
\begin{itemize}
    \item[{\rm (i)}]  If \(t \in [0, \tau_1]\), 
    $\cP^* = \{\bP \in \cP_{(\mu, b, d)} \mid \bP\left(X \in [t, b]\right) = 1 \}$, 
    all  distributions in \(\cP_{(\mu,b,d)}\) that are supported on the interval \([t, b]\). 
    
    \item[{\rm (ii)}] If \(t \in [\tau_1, \mu]\), 
    $\cP^* = \{\bP :  \bP\left(X = 0\right) = 1-\frac{\mu}{t}+\frac{d(b-t)}{2t(b-\mu)}, \, \bP\left(X = t\right) = \frac{\mu}{t} - \frac{b d}{2t(b-\mu)}, \, \bP\left(X = b\right) = \frac{d}{2(b-\mu)} \}$,
    the three-point distribution as derived in scenario 1a in the proof of Theorem~\ref{theorem:main}.
    
    \item[{\rm (iii)}] If \(t \in [\mu, \tau_2]\), 
    $\cP^* = \{\bP \in \cP_{(\mu, b, d)} \bigm| \bP\left(X = 0\right) = \frac{d}{2\mu}, \, \bP\left(X \in [t, b]\right) = 1 - \frac{d}{2\mu} \}$,
    all discrete/mixed distributions with probability mass \(\frac{d}{2\mu}\) on 0 and the remainder of its probability mass supported on \([t, b]\).
    
    \item[{\rm (iv)}] If \(t \in [\tau_2, b]\),
    $\cP^* = \{\bP \in \cP_{(\mu, b, d)} \bigm| \bP\left(X = t\right) = \frac{d}{2(t - \mu)}, \, \bP(X \in [0, \mu]) = 1 - \frac{d}{2(t - \mu)} \}$,
    all discrete/mixed distributions with probability mass \(\frac{d}{2(t - \mu)}\) on \(t\) and the remainder of its probability mass supported on \([0, \mu]\).
\end{itemize}
\end{proposition}
\begin{proof}
The proof follows almost directly from the complementary slackness relation explained in the proof of Theorem~\ref{theorem:main}. For $t\in[0,\tau_1]$ the dual solution function coincides with $\1\{x\geq t\}$ on the interval $[t,b]$. Hence, all distributions that are supported on this interval and obey the mean and MAD requirements are possible candidates for the worst-case distribution. Next, one can apply a similar reasoning for $t\in[\mu,\tau_2]$ and $t\in[\tau_2,b]$. The worst-case distribution can exist on the range where the dual solution function $F^*(x)$ and the indicator function coincide. To attain the same optimal value, the probability mass on the singletons is chosen accordingly. Finally, note that the second case is already shown in the proof of Theorem~\ref{theorem:main}.
\end{proof}

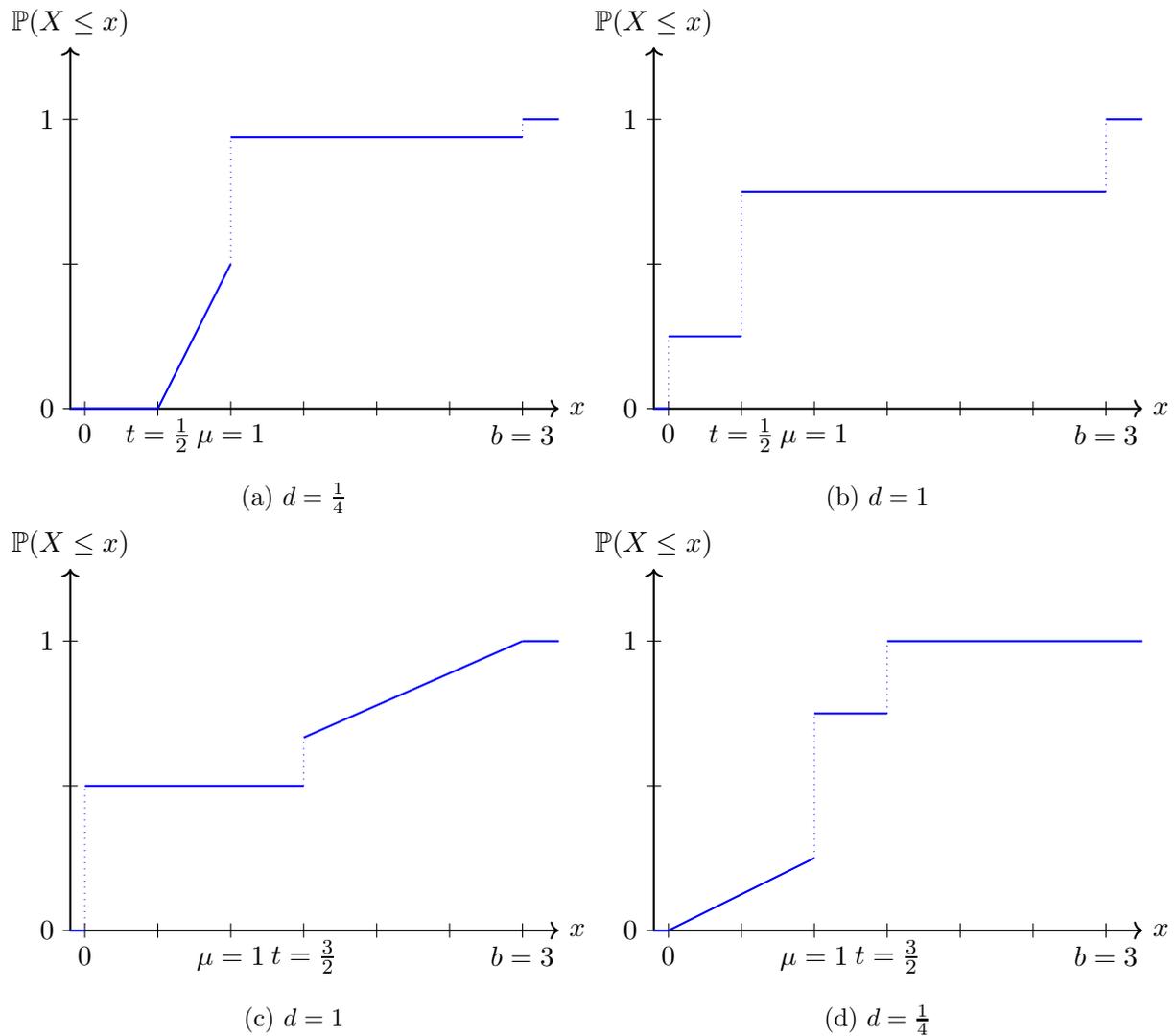
\begin{figure}[t]
	\begin{center}
	\begin{subfigure}{.5\textwidth}
	\begin{center}
    \begin{tikzpicture}[scale=2]
        \draw[<->,thick] (-0.1,2.5) node[above] (yaxis) {\(\bP(X \leq x)\)} |- (3.25,0) node[right] (xaxis) {$x$};
        
        \foreach \i in {0, ..., 6}
        {
            \draw (1/2*\i, -0.05) -- (1/2*\i, 0.05);
        }
        
        \foreach \i in {0, ..., 2}
        {
            \draw (-0.15, \i) -- (-0.05, \i);
        }
        
        \coordinate (origin) at (-0.1, 0);
        \node[left, xshift=-2.5] at (origin) {$0$};
        \node[left, xshift=-2.5] at (-0.1, 2) {$1$};
        \coordinate (destination) at (3.25, 2*1);
        
        \coordinate (00) at (0, 0);
        \node[below, yshift=-2.5] at (00) {$0$};

		\coordinate (mu0) at (1, 0);
		\node[below, yshift=-2.5] at (mu0) {$\mu = 1$};
        \coordinate (mu+) at (1, 2*0.5);
        \coordinate (mu++) at (1, 2*0.9375);
    
    	\coordinate (t0) at (0.5, 0);
    	\node[below] at (t0) {$t = \frac{1}{2}$};
        
        \coordinate (b0) at (3, 0);
        \node[below, yshift=-2.5] at (b0) {$b = 3$};
        \coordinate (b+) at (3, 2*0.9375);
        \coordinate (b++) at (3, 2*1);
        
        \draw[thick, blue] (origin) -- (00);
        \draw[thick, blue] (00) -- (t0);
        \draw[thick, blue] (t0) -- (mu+);
        \draw[dotted, blue] (mu+) -- (mu++);
        \draw[thick, blue] (mu++) -- (b+);
        \draw[dotted, blue] (b+) -- (b++);
        \draw[thick, blue] (b++) -- (destination);
        
    \end{tikzpicture}
    \caption{$ d = \frac{1}{4} $}
    \end{center}
\end{subfigure}%
\begin{subfigure}{.5\textwidth}
    \begin{center}
    \begin{tikzpicture}[scale=2]
        \draw[<->,thick] (-0.1,2*1.25) node[above] (yaxis) {\(\bP(X \leq x)\)} |- (3.25,0) node[right] (xaxis) {$x$};
        
        \foreach \i in {0, ..., 6}
        {
            \draw (1/2*\i, -0.05) -- (1/2*\i, 0.05);
        }
        
        \foreach \i in {0, ..., 2}
        {
            \draw (-0.15, \i) -- (-0.05, \i);
        }
        
        \coordinate (origin) at (-0.1, 0);
        \node[left, xshift=-2.5] at (origin) {$0$};
        \node[left, xshift=-2.5] at (-0.1, 2*1) {$1$};
        \coordinate (destination) at (3.25, 2*1);
        
        \coordinate (00) at (0, 0);
        \node[below, yshift=-2.5] at (00) {$0$};

        \coordinate (0+) at (0, 2*0.25);

		\coordinate (mu0) at (1, 0);
		\node[below, yshift=-2.5] at (mu0) {$\mu = 1$};
    
    	\coordinate (t0) at (0.5, 0);
    	\node[below] at (t0) {$t = \frac{1}{2}$};
    	\coordinate (t+) at (0.5, 2*0.25);
    	\coordinate (t++) at (0.5, 2*0.75);
        
        \coordinate (b0) at (3, 0);
        \node[below, yshift=-2.5] at (b0) {$b = 3$};
        \coordinate (b+) at (3, 2*0.75);
        \coordinate (b++) at (3, 2*1);
        
        \draw[thick, blue] (origin) -- (00);
        \draw[dotted, blue] (00) -- (0+);
        \draw[thick, blue] (0+) -- (t+);
        \draw[dotted, blue] (t+) -- (t++);
        \draw[thick, blue] (t++) -- (b+);
        \draw[dotted, blue] (b+) -- (b++);
        \draw[thick, blue] (b++) -- (destination);
        
    \end{tikzpicture}
    \caption{$ d = 1 $}
\end{center}
\end{subfigure}

\begin{subfigure}{.5\textwidth}
	\begin{center}
    \begin{tikzpicture}[scale=2]
        \draw[<->,thick] (-0.1,2*1.25) node[above] (yaxis) {\(\bP(X \leq x)\)} |- (3.25,0) node[right] (xaxis) {$x$};
        
        \foreach \i in {0, ..., 6}
        {
            \draw (1/2*\i, -0.05) -- (1/2*\i, 0.05);
        }
        
        \foreach \i in {0, ..., 2}
        {
            \draw (-0.15, \i) -- (-0.05, \i);
        }
        
        \coordinate (origin) at (-0.1, 0);
        \node[left, xshift=-2.5] at (origin) {$0$};
        \node[left, xshift=-2.5] at (-0.1, 2*1) {$1$};
        \coordinate (destination) at (3.25, 2*1);
        
        \coordinate (00) at (0, 0);
        \node[below, yshift=-2.5] at (00) {$0$};

        \coordinate (0+) at (0, 2*1/2);

		\coordinate (mu0) at (1, 0);
		\node[below, yshift=-2.5] at (mu0) {$\mu = 1$};
    
    	\coordinate (t0) at (1.5, 0);
    	\node[below] at (t0) {$t = \frac{3}{2}$};
    	\coordinate (t+) at (1.5, 2*1/2);
    	\coordinate (t++) at (1.5, 2*4/6);
        
        \coordinate (b0) at (3, 0);
        \node[below, yshift=-2.5] at (b0) {$b = 3$};
        \coordinate (b++) at (3, 2*1);
        
        \draw[thick, blue] (origin) -- (00);
        \draw[dotted, blue] (00) -- (0+);
        \draw[thick, blue] (0+) -- (t+);
        \draw[dotted, blue] (t+) -- (t++);
        \draw[thick, blue] (t++) -- (b++);
        \draw[thick, blue] (b++) -- (destination);
        
    \end{tikzpicture}
    \caption{$d = 1$}
    \end{center}
\end{subfigure}%
\begin{subfigure}{.5\textwidth}
    \begin{center}
    \begin{tikzpicture}[scale=2]
        \draw[<->,thick] (-0.1,2*1.25) node[above] (yaxis) {\(\bP(X \leq x)\)} |- (3.25,0) node[right] (xaxis) {$x$};
        
        \foreach \i in {0, ..., 6}
        {
            \draw (1/2*\i, -0.05) -- (1/2*\i, 0.05);
        }
        
        \foreach \i in {0, ..., 2}
        {
            \draw (-0.15, \i) -- (-0.05, \i);
        }
        
        \coordinate (origin) at (-0.1, 0);
        \node[left, xshift=-2.5] at (origin) {$0$};
        \node[left, xshift=-2.5] at (-0.1, 2*1) {$1$};
        \coordinate (destination) at (3.25, 2*1);
        
        \coordinate (00) at (0, 0);
        \node[below, yshift=-2.5] at (00) {$0$};

		\coordinate (mu0) at (1, 0);
		\node[below, yshift=-2.5] at (mu0) {$\mu = 1$};
		\coordinate(mu+) at (1, 2*2/8);
		\coordinate(mu++) at (1, 2*6/8);
    
    	\coordinate (t0) at (1.5, 0);
    	\node[below] at (t0) {$t = \frac{3}{2}$};
    	\coordinate (t+) at (1.5, 2*6/8);
    	\coordinate (t++) at (1.5, 2*1);
        
        \coordinate (b0) at (3, 0);
        \node[below, yshift=-2.5] at (b0) {$b = 3$};
        
        \draw[thick, blue] (origin) -- (00);
        \draw[thick, blue] (00) -- (mu+);
        \draw[dotted, blue] (mu+) -- (mu++);
        \draw[thick, blue] (mu++) -- (t+);
        \draw[dotted, blue] (t+) -- (t++);
        \draw[thick, blue] (t++) -- (destination);
        
    \end{tikzpicture}
    \caption{$d = \frac{1}{4} $}
    \end{center}
\end{subfigure}
\caption{Examples of the extremal distributions that attain the tail probability bound as described in Proposition \ref{prop: worst-cases}.  \label{fig: example distr}}
\end{center}
\end{figure}

Observe that when \(t\) equals \(\tau_1\), \(\mu\), or \(\tau_2\), there is only a single discrete extremal distribution. Figure~\ref{fig: example distr} provides examples of the worst-case distributions for several different parameter settings and values of $t$. Proposition~\ref{prop: worst-cases} shows that the ambiguity set $\cP_{(\mu,b,d)}$ results in a non-trivial collection of worst-case distributions; that is, the mean-MAD approach results in a set that does not solely include discrete distributions with a small number of atoms for $t\notin[\tau_1,\mu]\cup\{\tau_2\}$.


We next consider the tail bounds when also $\beta=\bP(X\geq\mu)$ is known. We therefore consider the extended ambiguity set
\begin{equation}
    \mathcal{P}_{(\mu,b, d,\beta)}=\{\bP : \bP\in\cP_{(\mu, b, d)},\,\bP(X\geq\mu)=\beta\}.
\end{equation}
Using this ambiguity set results in new tight bounds. These results are stated in the following two theorems for which the primal-dual proofs are given in Appendix~\ref{EC:betaproofs}.

\begin{theorem}\label{theorem:betaUB}
Consider a random variable $X$ with a distribution $\mathbb{P}$ in $\mathcal{P}_{(\mu, b, d,\beta)}$. Then,
\begin{equation}
    \sup\limits_{\mathbb{P}\in\mathcal{P}_{(\mu, b, d,\beta)}}\mathbb{P}(X\geq t)=
    \begin{cases}
    1, \quad &  t\in[0,\tau_1], \\
    \frac{(1-\beta)\mu+\beta t}{t}-\frac{d}{2t}, \quad &  t\in[\tau_1,\mu), \\
    \beta, &  t\in[\mu,\tau_2], \\
    \frac{d}{2(t-\mu)}, &  t\in[\tau_2,b], 
    \end{cases}
\end{equation}
with $\tau_1$ and $\tau_2$ given by
\begin{equation*}
    \tau_1 = \mu - \frac{d}{2(1-\beta)},\quad \tau_2 = \mu + \frac{d}{2\beta}.
\end{equation*}
\end{theorem}

\begin{theorem}\label{theorem:betaLB}
Consider a random variable $X$ with a distribution $\mathbb{P}$ in $\mathcal{P}_{(\mu, b, d,\beta)}$. Then,
\begin{equation}
    \inf\limits_{\mathbb{P}\in\mathcal{P}_{(\mu, b, d,\beta)}}\mathbb{P}(X > t)=
    \begin{cases}
    1-\frac{d}{2(\mu-t)}, \quad &  t\in[0,\tau_1], \\ 
    \beta, \quad &  t\in[\tau_1,\mu), \\
    \frac{\beta(\mu-t)}{(b-t)} + \frac{d}{2(b-t)}, &  t\in [\mu,\tau_2], \\
    0, &  t\in [\tau_2,b]
    \end{cases}
\end{equation}
with $\tau_1$ and $\tau_2$  given by
\begin{equation*}
    \tau_1 = \mu - \frac{d}{2(1-\beta)},\quad \tau_2 = \mu + \frac{d}{2\beta}.
\end{equation*}
\end{theorem}
Note that for these bounds equality between $\bP(X\geq t)$ and $\bP(X > t)$ does not hold. In particular, the bounds admit a jump discontinuity at $\mu$ for all distributions with $\beta\neq 1-\frac{d}{2\mu}$. In Figure~\ref{fig:illustrationbounds} the upper and lower bounds are depicted for the ambiguity set that considers all distributions with $\mu=0.5$, $d=0.1875$, $\beta=0.5$, $a=0$, and $b=1$. As a point of reference, the $\text{Beta}(2,2)$ tail distribution, which is a member of the ambiguity set, is also plotted.


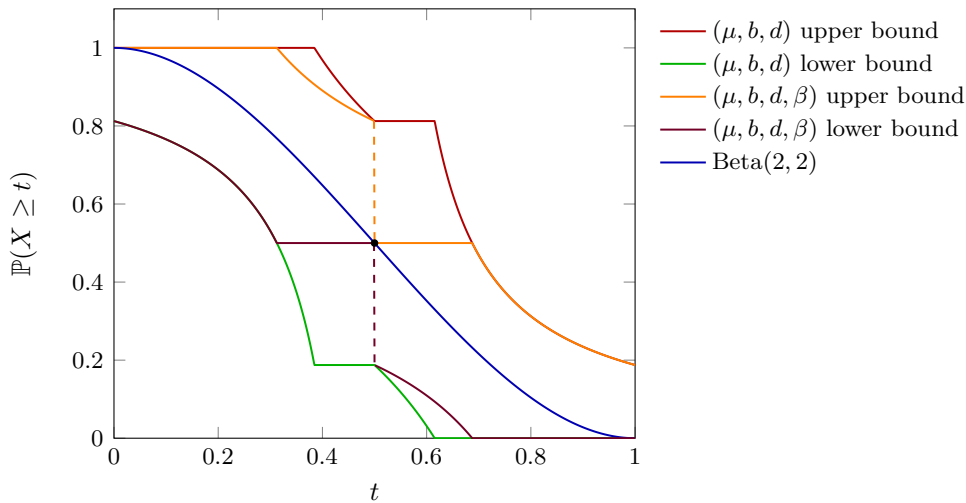
\begin{figure}[h!]
\begin{center}
\begin{tikzpicture}
\begin{axis}[
	xlabel={$t$},
	xmin=0, xmax=1,
	ymin=0, ymax=1.1,
    ylabel={$\bP(X \geq t)$},
    ylabel style={yshift=0.2cm},
    xticklabel style={
    	/pgf/number format/precision=4,
    	/pgf/number format/fixed,
	},
	yticklabel style={/pgf/number format/precision=4, /pgf/number format/fixed},
    label style={font=\small},
    ticklabel style={font=\footnotesize},
    legend style={draw=none, font=\footnotesize},
    legend cell align={left},
    legend pos = outer north east
]
\addplot[red!70!black, line width=0.75pt] table {Figure_Data/Comparison_MADUB.dat};

\addlegendentry{${(\mu,b,d)}$ upper bound};
\addplot[green!70!black, line width=0.75pt] table {Figure_Data/Comparison_MADLB.dat};

\addlegendentry{$(\mu,b,d)$ lower bound};
\addplot[orange, line width=0.75pt] table {Figure_Data/Comparison_betaUBp1.dat};
\addlegendentry{$(\mu,b,d,\beta)$ upper bound};
\addplot[purple!60!black, line width=0.75pt] table {Figure_Data/Comparison_betaLBp1.dat};
\addlegendentry{$(\mu,b,d,\beta)$ lower bound};
\addplot[blue!70!black, line width=0.75pt] table {Figure_Data/Comparison_exactUniform.dat};
\addlegendentry{$\text{Beta}(2,2)$};
\addplot[orange, line width=0.75pt, dashed] table {Figure_Data/Comparison_betaUBp2.dat};
\addplot[orange, line width=0.75pt] table {Figure_Data/Comparison_betaUBp3.dat};
\addplot[purple!60!black, dashed, line width=0.75pt] table {Figure_Data/Comparison_betaLBp2.dat};
\addplot[purple!60!black, line width=0.75pt] table {Figure_Data/Comparison_betaLBp3.dat};
\node at (0.5,0.5)[circle,fill,inner sep=1pt]{};
\end{axis}
\end{tikzpicture}
\end{center}
    \caption{An illustration of the mean-MAD-$\beta$ bounds for the tail probability where the ambiguity set consists of all distributions with $\mu=0.5$, $d=0.1875$, $\beta=0.5$, $a=0$, and $b=1$. The blue line represents the tail distribution of a $\text{Beta}(2,2)$ distribution.}
    \label{fig:illustrationbounds}
\end{figure}

\subsection{Comparison with other bounds}\label{sec:compare}

Closely related to our results is the discussion in section 4.1 of \cite{ghosal2020distributionally} . In particular, they consider, among others, an ambiguity set given by
\[\tilde{\cP}_{(\mu, b, d)} = \left\{\bP : \mathcal{B} \to \left[0, 1\right] \relmiddle| \bP\left[ X \in [0, b] \right] = 1, \enskip \bE_\bP\left[ X \right] = \mu, \enskip \bE_\bP\left[ | X - \mu | \right] \leq d \right\}.\]
The only difference with the ambiguity set we use is the inclusion of all distributions with a lower mean absolute deviation. This has major implications for the maximum and minimum probability to exceed \(t\), however. First of all, it should be noted that the distribution with all its probability mass on \(\mu\) is an element of \(\tilde{\cP}_{(\mu, b, d)}\) for any value of \(d\). This means that for any \(t \leq \mu\) it holds that
\[\sup_{\bP \in \tilde{\cP}} \bP \left(X \geq t\right) = 1.\]
Moreover, for any \(t > \mu\) and \(d > \frac{2 \mu \left(t - \mu\right)}{t}\), the maximum probability of \(X\) exceeding \(t\) is attained by a distribution with a mean absolute deviation equal to \(\frac{2 \mu \left(t - \mu\right)}{t}\), which is explained by the observation that the bound we obtain is decreasing in \(d\) for \(d > \frac{2 \mu \left(t - \mu\right)}{t}\). 

Clearly, because of the above observations, the theoretical maximum of \(\bP\left(X > t\right)\) has a much simpler closed-form solution than \eqref{eq: worst-case prob} for the ambiguity set \(\tilde{\cP}_{(\mu, b, d)}\). A big downside is that many of the extra distributions contained in \(\tilde{\cP}_{(\mu, b, d)}\) but not in \(\cP_{(\mu, b, d)}\) might be unrealistic. Especially when the mean absolute deviation is known or can be accurately estimated, there is little reason to consider distributions with a different (in this case lower) mean absolute deviation. For large values of \(d\) relative to \(t\) in particular, using \(\tilde{\cP}_{(\mu, b, d)}\) can lead to an overestimation of the maximum value of \(\bP \left( X > t\right)\). The observation that the maximum value of \(\bP \left( X > t\right)\) is decreasing in \(d\) for large values of \(d\) also means that considering distributions with a lower mean absolute deviation can lead to a higher bound on \(\bP \left( X > t\right)\).

Comparing the result of Theorem~\ref{theorem:main} to Cantelli's inequality \citep{chebyshev1867valeurs} is harder, since we assume the mean absolute deviation to be known, but not the variance. Hence, some relation between these two quantities is needed to be able to make a comparison. In particular, we will use that
\begin{equation} \label{eq: sigma bounds} 
\frac{d^2}{4\beta\left(1-\beta\right)} \leq \sigma^2 \leq \frac{d b}{2},
\end{equation}
where \(\beta = \bP\left(X > \mu \right)\)~\citep{ben1985approximation}. We note that this also implies \(d \leq \sigma\). Throughout the comparison below we assume that \(d\) is given and compare the bound obtained in Theorem~\ref{theorem:main} with Cantelli's bound for different values of \(\sigma\) satisfying \eqref{eq: sigma bounds}. Figure~\ref{fig: ub comparison} illustrates this comparison for a simple numerical example with the following parameters: \(a = -1\), \(\mu = 0\), \(b = 1\), \(d = \frac{1}{4}\). We consider three values for \(\sigma\): \(\sigma = d = \frac{1}{4}\), \(\sigma = \frac{1}{3}\) and \(\sigma = \sqrt{\frac{d b}{2}} = \frac{1}{2}\). 
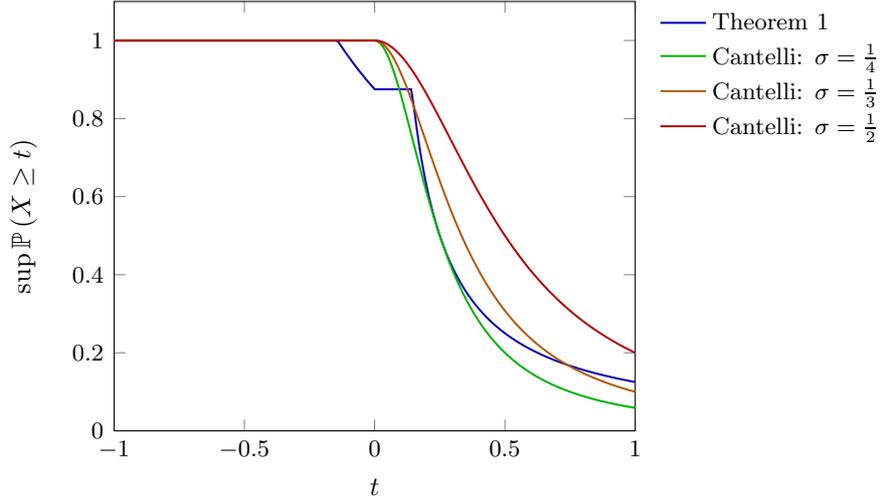
\begin{figure}
\begin{center}
\begin{tikzpicture}
\begin{axis}[
	xlabel={\(t\)},
	xmin=-1, xmax=1,
	ymin=0, ymax=1.1,
    ylabel={\(\sup \bP\left(X \geq t\right)\)},
    ylabel style={yshift=0.2cm},
    xticklabel style={
    	/pgf/number format/precision=4,
    	/pgf/number format/fixed,
	},
	yticklabel style={/pgf/number format/precision=4, /pgf/number format/fixed},
    label style={font=\small},
    ticklabel style={font=\footnotesize},
    legend style={draw=none, font=\footnotesize},
    legend cell align={left},
    legend pos = outer north east
]
\addplot[blue!70!black, line width=0.75pt] table {Figure_Data/Comparison_mu_d.dat};
\addlegendentry{Theorem \ref{theorem:main}};
\addplot[green!70!black, line width=0.75pt] table {Figure_Data/Comparison_sigma_low.dat};
\addlegendentry{Cantelli: \(\sigma = \frac{1}{4}\)};
\addplot[orange!70!black, line width=0.75pt] table {Figure_Data/Comparison_sigma_med.dat};
\addlegendentry{Cantelli: \(\sigma = \frac{1}{3}\)};
\addplot[red!70!black, line width=0.75pt] table {Figure_Data/Comparison_sigma_high.dat};
\addlegendentry{Cantelli: \(\sigma = \frac{1}{2}\)};
\end{axis}
\end{tikzpicture}
\end{center}
\caption{A comparison of the bound in Theorem~\ref{theorem:main} with Cantelli's bound for three different values of \(\sigma\) with the parameter values \(a = -1\), \(\mu = 0\), \(b = 1\) and \(d = \frac{1}{4}\). \label{fig: ub comparison}}
\end{figure}

Figure~\ref{fig: ub comparison} gives rise to a number of interesting observations. First of all, we note that since Cantelli's bound is 1 for any \(t \leq \mu\), the bound from Theorem~\ref{theorem:main} is at most Cantelli's bound as it includes an interval for which it is not 1. Furthermore, the flat area in the blue line corresponds to the values of \(t\) such that 
\[\min \Big\{\frac{d}{2\left(t - \mu\right)}, \: 1 - \frac{d}{2\mu} \Big\} = 1 - \frac{d}{2 \mu },\]
which corresponds to all \(\mu \leq t \leq \tau^* := \mu + \frac{d \mu}{2 \mu - d}\).
Moreover, we note that for \(\sigma = d\), Cantelli's bound is lower than \eqref{eq: worst-case prob} for all \(\tau^* \leq t \leq b\). This is true for all parameters as:
\begin{align*}
\frac{d^2}{d^2 + \left(t - \mu\right)^2}  \enskip & = \frac{d^2}{d^2 + (t - \mu)^2 -2d(t - \mu) + 2d(t - \mu)} \\
& = \frac{d^2}{\left(d - (t - \mu)\right)^2 + 2d(t-\mu)} \\
& \leq \frac{d^2}{2d(t - \mu)} \\
& = \frac{d}{2\left(t - \mu\right)}.
\end{align*}
In particular, for \(\sigma = d\) Cantelli's bound and \eqref{eq: worst-case prob} always coincide at \(t = \mu + d\), since:
\[\frac{d^2}{d^2 + d^2} = \frac{1}{2} = \frac{d}{2d}.\]
If, on the other hand, we choose \(\sigma = \sqrt{\frac{d b}{2}}\), its highest possible value, Cantelli's bound is higher than \eqref{eq: worst-case prob}. This is true for all parameter values as well, as Cantelli's bound is increasing in \(\sigma\) and must thus be at least \eqref{eq: worst-case prob} for its highest possible value.

For intermediate values of \(\sigma\), we observe behavior similar to the line corresponding to \(\sigma=\frac{1}{3}\) in Figure~\ref{fig: ub comparison}. More specifically, we find that \eqref{eq: worst-case prob} is lower than Cantelli's bound for all \(t\) in the two intervals \([0, \hat{\tau}]\) and\([\underline{\tau}, \overline{\tau}]\), with the three boundaries given by
\begin{align*}
\hat{\tau} &= \mu + \sqrt{\frac{d \sigma^2}{2(\mu-a) -d}},\\
\underline{\tau} &= \frac{\sigma^2}{d} - \sigma \sqrt{\frac{\sigma^2}{d^2} - 1},\\
\overline{\tau} &= \min \Big\{b, \frac{\sigma^2}{d} + \sigma \sqrt{\frac{\sigma^2}{d^2} - 1} \Big\}.
\end{align*}

Note that for some \(\sigma\), such as \(\sigma = \sqrt{\frac{d b}{2}}\) in Figure~\ref{fig: ub comparison}, it holds that \(\hat{\tau} \geq \underline{\tau}\), that is, \eqref{eq: worst-case prob} is lower than Cantelli's bound for all \(t\in [\mu, \overline{\tau}]\). To visually clarify all boundaries discussed above, Figure~\ref{fig: all taus} only shows Cantelli's bound for \(\sigma = 0.27\) and marks \(\tau^*\), \(\hat{\tau}\), \(\underline{\tau}\) and \(\overline{\tau}\).
\begin{figure}
\begin{center}
\begin{tikzpicture}
\begin{axis}[
	xlabel={\(t\)},
	xmin=0, xmax=1,
	ymin=0, ymax=1,
    ylabel={\(\bP\left(X \geq t\right)\)},
    ylabel style={yshift=0.2cm},
    xticklabel style={
    	/pgf/number format/precision=4,
    	/pgf/number format/fixed,
	},
	yticklabel style={/pgf/number format/precision=4, /pgf/number format/fixed},
    label style={font=\small},
    ticklabel style={font=\footnotesize},
    legend style={draw=none, font=\footnotesize},
    legend cell align={left}
]
\addplot[blue!70!black, line width=0.75pt] table {Figure_Data/Comparison_mu_d.dat};
\addlegendentry{Theorem \ref{theorem:main}};
\addplot[orange!70!black, line width=0.75pt] table {Figure_Data/Illustration_tau.dat};
\addlegendentry{Cantelli: \(\sigma = 0.27\)};
\end{axis}


\draw[color=black, style=dashed] (0.7000, 4.9836) node[below, xshift=-5] {\(\hat{\tau}\)} -- (0.7000, 0);

\draw[color=black, style=dashed] (0.9797, 4.9836) node[above, xshift=7] {\(\tau^*\)} -- (0.9797, 5.6955);

\draw[color=black, style=dashed] (1.2444, 3.9236) node[right, xshift=1, yshift=-3] {\(\underline{\tau}\)} -- (1.2444, 0);

\draw[color=black, style=dashed] (2.7541, 1.7719) node[above, xshift=1] {\(\overline{\tau}\)} -- (2.7541, 0);
\end{tikzpicture}
\end{center}
\caption{An illustration of \(\hat{\tau}\), \(\tau^*\), \(\underline{\tau}\) and \(\overline{\tau}\) for the parameter values \(a = -1\), \(\mu = 0\), \(b = 1\) and \(d = 0.25\). \label{fig: all taus}}
\end{figure}
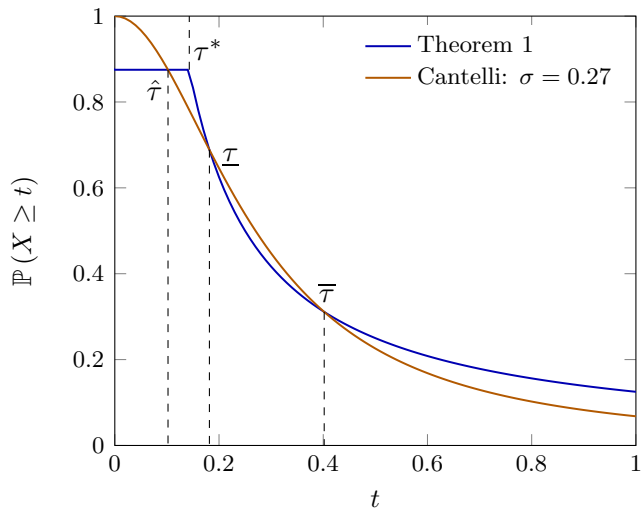

\subsection{Prior work on Chebyshev-type tail bounds}
Multivariate generalizations of Chebyshev's inequality have also been studied. In \cite{bertsimas2005optimal} and \cite{vandenberghe2007generalized} generalizations are studied through formulating a convex optimization problem, given that the prescribed confidence region can be described by polynomial or linear and quadratic inequalities, respectively. In \cite{grechuk2010chebyshev} on the other hand, closed-form variants of Chebyshev's inequality are provided for different dispersion measures than the variance. Generalized versions of Chebyshev's inequality for products of random variables that focus on a one-sided inequality have also received some attention recently~\citep{rujeerapaiboon2018chebyshev}.

While Chebyshev's inequality is tight, it has been criticized for only being attained by pathological distributions that abuse the unboundedness of the underlying support and are not considered realistic in many applications~\citep{vanparys2016generalized}. A variant of the Chebyshev inequality that was already considered in \cite{gauss1821theoria} restricts the distributions it considers to be unimodal. This yields an improvement by a factor \(\frac{4}{9}\) over the classical Chebyshev inequality. This idea of including unimodality has been extended to the multivariate case recently as well~\citep{vanparys2016generalized}.

All the above mentioned inequalities, however, still assume an unbounded support. \cite{schepper1995general} mention tail probability bounds that incorporate the upper bound of the random variable's range. A comparison is provided in Appendix~\ref{app: deschepper}. Unfortunately, these bounds are only attained by rather pathological probability measures, i.e., point distributions with two or three atoms. Using the MAD instead of the variance results in a richer class of worst-case distributions.

\section{Distribution-free analysis of OR models} \label{sec: applications}

We now turn to three classical OR models: the newsvendor problem, monopoly pricing and stop-loss reinsurance. These three models can be subjected to distribution-free analyses that make direct use of the novel Chebyshev bounds. This leads to closed-form solutions of the associated maxmin optimization. The common theme is that with ambiguity described in terms of mean, MAD and restricted support, distribution-free analysis  leads to valuable structural insights, while unrestricted  support often yields degenerate results. 

\subsection{Newsvendor problem} \label{sec: newsvendor}

The newsvendor problem serves to find the order quantity that maximizes the expected profit for a single period given a stochastic demand. Denote by $q$ the order quantity  (number of units) and by $D$ the stochastic demand during a single selling period. Per unit,  $p$ denotes the selling price and $c$ the purchase cost.  Let $p>c$, and assume without loss of generality that unsold units have zero salvage value. 
The expected profit is $\mathbb{E}_\mathbb{P}[\pi(q, D)]$ with $D\sim \mathbb{P}$  and 
$$
\pi(q, D) = p \min\{q, D\}-cq. 
$$
The decision maker then chooses the optimal order quantity $q^*$ that solves $\max_q\mathbb{E}_\mathbb{P}[\pi(q, D)]$. This solution is known to be the $\eta:=(p-c)/p$ quantile (critical quantile) of the distribution of $D$, that is, 
\begin{equation}\label{critqua}
q^*=\min\{q \, : \, \mathbb{P}(D\leq q) \geq \eta\}.
\end{equation} 
In practice, however,  the decision maker might only know partial information on the demand distribution. \cite{Scarf1958} pioneered distribution-free analysis of the newsvendor problem, when only the mean and the variance of the demand are known. Scarf obtained the optimal order quantity for the worst case demand, turning the newsvendor into a maxmin decision maker that solves
\begin{equation}\label{maxminscarf}
\max_q\inf\limits_{\mathbb{P} \in \mathcal{P}_{(\mu , \sigma)}}
\bE_{\mathbb{P}} \left[\pi(q, D) \right] ,
\end{equation}
with $\mathcal{P}_{(\mu ,\sigma)}$  the ambiguity set that contains all distributions with a given mean $\mu$ and variance $\sigma^2$, and 
solution
\begin{equation}\label{optqsca}
q^S=
\left\{\begin{array}{ll}
0,&  {\rm if} \  \eta< \frac{\sigma^2}{\mu^2+\sigma^2},\\
\mu+\frac{\sigma}{2}\Big(\sqrt{\frac{1}{\eta(1-\eta)}} - 2\sqrt{\frac{1}{\eta}-1} \Big),&  {\rm if} \  \eta\geq \frac{\sigma^2}{\mu^2+\sigma^2}.
 \end{array}\right.
\end{equation}

We shall instead consider all demand distributions with given mean $\mu$, MAD $d$ and support $[0,b]$, and consider
\begin{equation}\label{maxminbental}
\max_q\inf\limits_{\mathbb{P} \in \mathcal{P}_{(\mu , d,b)}}
\bE_{\mathbb{P}} \left[\pi(q, D) \right].  
\end{equation}
This is the counterpart of problem \eqref{maxminscarf}. \cite{Scarf1958} solved  \eqref{maxminscarf} directly, computing the lower bound $\inf_{\mathbb{P} \in \mathcal{P}_{(\mu , \sigma)}}\bE_{\mathbb{P}} \left[\pi(q, D) \right]$ via a linear program. 
Instead, we do not solve \eqref{maxminbental} directly, but apply the robust Chebyshev bounds to the first-order condition for $q^*$ in \eqref{critqua}. 
Clearly, tight lower and upper bounds for this quantile follow from 
$\inf_{\mathbb{P} \in \mathcal{P}_{(\mu , d,b)}}\mathbb{P}(D>q)$ and $\sup_{\mathbb{P} \in \mathcal{P}_{(\mu , d,b)}}\mathbb{P}(D>q)$, respectively, providing an interval that contains the optimal order quantity $q^*$.

\begin{theorem}[Order quantity bounds under mean-MAD-range ambiguity]\label{newsmadthm}  Suppose the newsvendor knows the mean $\mu$, the mean absolute deviation $d$ and the upper 
bound $b$ of the demand distribution $\mathbb{P}(D\leq q)$. The optimal order quantity $q^*$ that solves $\max_q\mathbb{E}_\mathbb{P}[\pi(q, D)]$ is then contained in the interval $[q^L,q^U]$ with
\begin{align}\label{optimalqcase1}
[q^L,q^U]&=
    \begin{cases}
        \  \Big[0,\, \frac{2\mu(b-\mu)-bd}{2(b-\mu)(1-\eta)-d}\Big], & {\rm if } \ \eta< \frac{d}{2\mu},\\
    \  \Big[\mu-\frac{d}{2\eta},\,\mu+\frac{d}{2(1-\eta)}\Big],  \quad & {\rm if }\ \frac{d}{2\mu}\leq \eta\leq 1-\frac{d}{2(b-\mu)}, \\
   \       \Big[\frac{\mu-b(1-\eta)}{\eta-d/(2\mu)},\, b\Big],  \quad & {\rm if }\ \eta\geq 1-\frac{d}{2(b-\mu)}. \\
    \end{cases}
    \end{align}
\end{theorem}
The theorem provides various handles for a robust policy that responds to the uncertainty captured in $\mathcal{P}_{(\mu , d,b)}$. The lower bound $q^L$ follows from the worst-case demand distribution. Observe that $q^L$ is larger than $\mu$ when the profit margin $\eta$ exceeds $1-d/(2(b- \mu))$, and smaller than $\mu$ otherwise. This insight can be contrasted with $q^S$ in \eqref{optqsca} that also considers the worst-case scenario, but then in view of $\mathcal{P}_{(\mu ,\sigma)}$ ambiguity. Scarf's $q^S$ is larger than $\mu$ if $\eta>1/2$ and smaller than $\mu$ otherwise. Hence, $q^L$ quantifies the dependency on $b$, where $q^S$ does not. In particular, when the profit margin $\eta$ is fixed, the pessimistic newsvendor that uses $q^L$ will only order above the mean when $b$ does not exceed $\mu+d/(2(1-\eta))$. 

Table~\ref{res:newstable} shows that the support $[0,b]$ also influences the intervals $[q^L,q^U]$, in particular for low and high profit margins. We also recognize the three different regimes in Theorem~\ref{newsmadthm} that correspond to low margins, average margins and high margins.

\begin{table}[h!]
\begin{center}
\begin{tabular}{rrrrr}\hline
{ $\eta$}     & { $b=10$}           & {$b=15$}           & { $b=20$}            & { $b=\infty$}           \\ \hline
0.01 & $[0.00, 4.17]$     & $[0.00, 4.23]$      & $[0.00, 4.26]$      & $[0.00, 4.29]$      \\
0.1  & $[0.00, 4.67]$     & $[0.00, 4.70]$      & $[0.00, 4.71]$      & $[0.00, 4.72]$      \\
0.2  & $[ 1.25, 5.94]$ & $[ 1.25, 5.94]$  & $[ 1.25, 5.94]$  & $[ 1.25, 5.94]$  \\
0.4  & $[ 3.13, 6.25]$ & $[ 3.13, 6.25]$  & $[ 3.13, 6.25]$  & $[ 3.13, 6.25]$  \\
0.5  & $[ 3.50, 6.50]$ & $[ 3.50, 6.50]$  & $[ 3.50, 6.50]$  & $[ 3.50, 6.50]$  \\
0.7  & $[ 3.93, 7.50]$ & $[ 3.93, 7.50]$  & $[ 3.93, 7.50]$  & $[ 3.93, 7.50]$  \\
0.9  & $[ 5.33,10.00]$    & $[ 4.17, 12.50]$ & $[ 4.17, 12.50]$ & $[ 4.17, 12.50]$ \\
0.95 & $[ 5.63,10.00]$    & $[ 5.31,15.00]$     & $[ 5.00,20.00]$     & $[ 4.21, 20.00]$    \\
0.99 & $[ 5.83,10.00]$    & $[ 5.77,15.00]$     & $[ 5.71,20.00]$     & $[ 4.24, 80.00]$    \\ \hline   
\end{tabular}
\end{center}
\caption{The intervals $[q^L,q^U]$ for mean-MAD ambiguity with $\mu=5$, $d=1.5$ and various profit margins $\eta$.}\label{res:newstable}
\end{table}

We mention two further works related to Theorem~\ref{newsmadthm}.  \cite{BenTal1976} use general techniques for stochastic programs with limited information such as \eqref{maxminbental}. For such stochastic programs the available information is often not sufficient to find the optimal solution. \cite{BenTal1976} develop a method to construct the minimal set that should contain the optimum. They also demonstrate this technique for the newsvendor model with given mean and MAD, but unbounded support, and obtain intervals that indeed arise from Theorem~\ref{newsmadthm} for the limit $b\to\infty$: \begin{align}\label{optimalrcase1special}
[q^L,q^U]&=
\left\{\begin{array}{ll}
\left[0,\frac{\mu-{d}/{2}}{1-\eta}\right]&  {\rm when} \  \eta< \frac{d}{2\mu},\\
\left[\mu-\frac{d}{2\eta},\mu+\frac{d}{2(1-\eta)}\right]&  {\rm when} \  \eta\geq \frac{d}{2\mu}.
 \end{array}\right.
\end{align}

\cite{semivariance} introduce semi-variance as an extra piece of information about the skewness of the distribution. Together with the mean and variance, this results in a more restrictive ambiguity set (compared to Scarf), and therefore a less conservative (or sharper) estimation of $q^*$.  Theorem~\ref{newsmadthm} can also be viewed as a way to address conservatism, by taking into account the finite support. We could even restrict the ambiguity further by using the robust Chebyshev bound with the additional $\mathbb{P}(X\geq \mu)=\beta$ constraint, which like  semi-variance measures skewness. We apply these bounds to the newsvendor model in Appendix~\ref{newsskew}. 

Apart from modifying or narrowing the ambiguity set, conservatism can be alleviated by choosing alternate objective functions, for instance replacing the profit function by a regret function (opportunity cost of not making the optimal decision)
\citep{yue2006expected,perakis2008regret}, or by extending the profit function with a utility function $u(\cdot)$ for max-min analysis of 
$\mathbb{E}[u(\pi(q,D))]$  \citep{han2014risk}. See 
\cite{semivariance} for an extensive review of many other studies on distribution-free newsvendor models.  The robust Chebyshev bounds developed in this paper can be used  for distribution-free analysis of more advanced models, including those modeling regret and utility mentioned above, the risk-averse newsvendor with stochastic price-dependent demand \citep{chen2009risk} and multi-product settings \citep{choi2011multiproduct}.

\subsection{Monopoly pricing} \label{sec: pricing}
The monopoly pricing problem maximizes a seller's profit when selling a single object to a buyer who is willing to pay some unknown value $B$. Traditionally, it is assumed that $B$ is drawn from some distribution $\prob{B\leq r}$ that is known to the seller, so that the seller can set the optimal price. When there is a single buyer, the optimal strategy is to post a fixed price $r$ that maximizes the expected profit $r\prob{B> r}$; see \cite{riley1983optimal} and \cite{myerson1981optimal}. 
The seller thus faces the tradeoff between price and sale, because the probability of sale $\prob{B> r}$ decreases with the price $r$. 

We consider a robust variant of this model, where instead of knowing the distribution, the seller only knows that the distribution of $B$ is contained in $\mathcal{P}_{(\mu , d,b)}$
and chooses the price that maximizes the worst-case expected profit: 
\begin{equation}\label{mpp}
\max_r\inf\limits_{\mathbb{P} \in \mathcal{P}_{(\mu , d,b)}}
r\prob{B> r} .
\end{equation}
Refer to the solution to this maxmin optimization problem as the robustly optimal (or maxmin) price $r^L$. The tight lower bound for $r\prob{B> r} $, denoted by $F(r)$, follows from the robust Chebyshev bound for the tail probability  $\prob{B> r}$ in \eqref{eq:thm2}. 
The resulting worst-case profit function turns out not be concave and in fact to have multiple local maxima, as illustrated in Figure~\ref{fig:plotsF}.

\begin{figure}[h]
\begin{center}
\begin{tikzpicture}
\begin{axis}[
	xlabel={$r$},
	xmin=0, xmax=1,
	ymin=0, ymax=0.22,
    ylabel={$F(r)$},
    ylabel style={yshift=0.2cm},
    xticklabel style={
    	/pgf/number format/precision=4,
    	/pgf/number format/fixed,
	},
    scaled x ticks=false,
	yticklabel style={/pgf/number format/precision=4, /pgf/number format/fixed},
	scaled y ticks=false,
    label style={font=\small},
    ticklabel style={font=\footnotesize},
    y label style={at={(current axis.above origin)},rotate=-90,anchor=south,yshift=0.2cm},
    legend style={draw=none, font=\footnotesize},
    legend cell align={left},
    legend pos = outer north east
]
\addplot[blue!70!black, line width=0.75pt] table {Figure_Data/pricing/Fd1.dat};
\addlegendentry{$d=0.20$};
\addplot[yellow!70!black, line width=0.75pt] table {Figure_Data/pricing/Fd2.dat};
\addlegendentry{$d=0.25$};
\addplot[green!70!black, line width=0.75pt] table {Figure_Data/pricing/Fd3.dat};
\addlegendentry{$d=0.30$};
\addplot[red!70!black, line width=0.75pt] table {Figure_Data/pricing/Fd4.dat};
\addlegendentry{$d=0.38$};
\end{axis}
\end{tikzpicture}
\end{center}
\caption{Tight lower bound $F(r)$ for $r\prob{B> r}$ with $\mu=0.5$ and $b=1$, which is evaluated for $d=0.20,\,0.25,\,0.30,\,0.38$. 
}\label{fig:plotsF}
\end{figure}
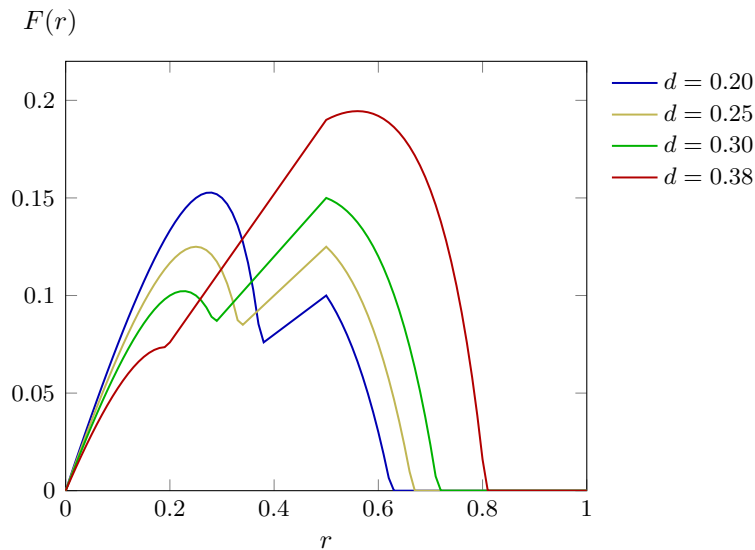

Upon maximizing this worst-case profit function, the next result shows that there exist three ranges of dispersion (measured in MAD), each with a different optimal price that attains the largest local maximum.

\begin{theorem}[Optimal price under mean-MAD-range ambiguity; \cite{jacco}]\label{mainmadthm}  Suppose the seller knows the mean $\mu$, the mean absolute deviation $d$ and the upper 
bound $b$ of the value distribution. For $b\in[\mu,5\mu]$, the solution to the monopoly pricing problem \eqref{mpp} is given by
\begin{align}\label{optimalrcase1}
r^L&=
\left\{\begin{array}{ll}
r_1=\mu-\sqrt{\frac{d\mu}{2}},&  {\rm if} \  d\in[0,d_1],\\
\mu,&  {\rm if} \  d\in[d_1,d_2],\\
r_2=b -\frac{\sqrt{b  (2 \mu -d) (2 \mu  (b -\mu )-b  d)}}{2 \mu -d},&  {\rm if} \  d\in[d_2,d_{\rm max}],
 \end{array}\right.
\end{align}
where
$$
d_1=\frac{2 b \mu  (b-\mu )-4 \sqrt{\mu ^3 (b-\mu )^3}}{(b-2 \mu )^2}, \quad d_2=
\frac{2 \left(b \mu -\mu ^2\right)}{2 b-\mu }, \quad d_{\rm max}= \frac{2(b-\mu)\mu}{b}.
$$
For $b>5\mu$ and $d\in[0,d_2]$ the optimal price is $r^L= r_1$.  For $b>5\mu$ and $d\in(d_2,d_{\rm max}]$ the optimal price $r^L$ is either $r_1$ or $r_2$. 
\end{theorem}

The proof of Theorem \ref{mainmadthm} is due to \cite{jacco} and also presented in 
Appendix~\ref{EC:proofprice}. 
Theorem \ref{mainmadthm} shows that the pricing function $r^L$ is not monotone in the dispersion (measured in MAD), and that dispersion and support both have a major influence on the pricing strategy. The theorem 
identifies the dispersion thresholds $d_1$ and $d_2$. As a function of dispersion, the price is smaller than $\mu$ and decreasing until $d_1$, then is $\mu$ until $d_2$, and then increases towards the maximal value $b$. Hence,  only when dispersion is high, the seller is willing to set a price close to the maximum price $b$. This is illustrated in Figure \ref{fig:optprice1}.  Theorem \ref{mainmadthm}  degenerates when the support exceeds 5 times $\mu$. In particular, when $b\to\infty$, the optimal price is $\mu-\sqrt{{d\mu}/{2}}$ for all $d\leq 2\mu$. This price is always lower than $\mu$. 


\begin{figure}[h]
\begin{center}
\begin{tikzpicture}
\begin{axis}[
	xlabel={$d$},
	xmin=0, xmax=0.5,
	ymin=0.2, ymax=1,
    ylabel={$r^L$},
    ylabel style={yshift=0.2cm},
    xticklabel style={
    	/pgf/number format/precision=4,
    	/pgf/number format/fixed,
	},
    scaled x ticks=false,
	yticklabel style={/pgf/number format/precision=4, /pgf/number format/fixed},
	scaled y ticks=false,
    label style={font=\small},
    ticklabel style={font=\footnotesize},
    y label style={at={(current axis.above origin)},rotate=-90,anchor=south,yshift=0.2cm}
]
\addplot[blue!70!black, line width=0.75pt] table {Figure_Data/pricing/rL1.dat};
\addplot[blue!70!black, line width=0.75pt] table {Figure_Data/pricing/rL2.dat};
\end{axis}
\end{tikzpicture}
\end{center}
\caption{The optimal price $r^L$ for $\mu=0.5$ and $b=1$, for which $d_1=1/4$ and $d_2=1/3$. 
}\label{fig:optprice1}

\end{figure}
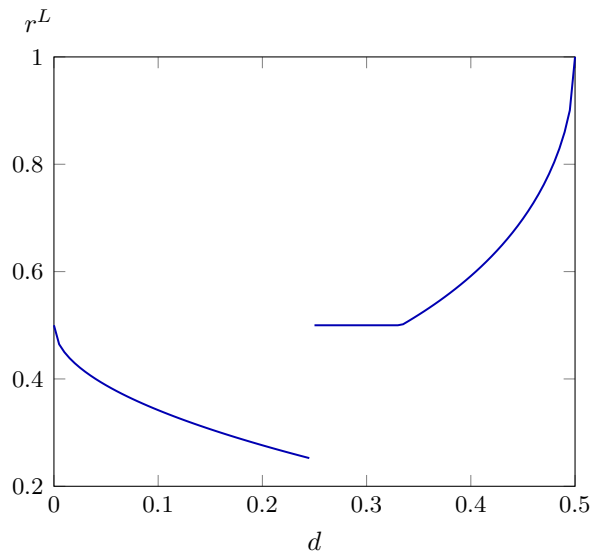

Theorem \ref{mainmadthm}  contributes to the active research field of robust pricing; see  \cite{carroll2019robustness} for a recent overview. We mention a few related works. 
 \cite{kos2015selling} show that when the seller only knows the mean valuation, the maxmin profit is always zero. This solidified intuition that in absence of an upper bound  arbitrarily high valuations cause overly pessimistic scenarios. 
Indeed, \cite{kos2015selling} show that when there is an upper bound, the seller can find a nontrivial maxmin price that is smaller than $\mu$ and generates positive expected profit.  The same holds true when the seller knows the variance (instead of upper bound); see \cite{carrasco2018optimal}. In both cases, the maxmin price  monotonically decreases with the allowed dispersion (either measured as upper bound or variance).  \cite{suzdaltsev2018distributionally} considers the case when the mean, variance and upper bound are all three known, and finds a maxmin price that is smaller than $\mu$ for low variance and greater than $\mu$ for high variance. 

The monopolistic pricing problem is connected to virtual valuations $v(r):=r-{\prob{B>r}}/{f_B(r)}$, with $f_B(r)$ the probability distribution function of $B$.
 These virtual valuations measure the surplus that can be extracted from agents, and can be used for optimal design of auctions with multiple buyers or object-types \citep{myerson1981optimal}. With a single buyer,  the optimal price is the solution to $v(r)=0$, which indeed is  $
{\rm arg\,max}_r \,r\,\prob{B> r}$. Hence, there are possibilities for deploying the robust Chebyshev bounds for other models in pricing and mechanism design, for instance distribution-free analysis of auctions with multiple independent bids \citep{suzdaltsev2018distributionally} or correlated bids 
\citep{che2019robust}.


\subsection{Stop-loss reinsurance} \label{sec: reinsurance}

Reinsurance is a classical topic in the actuarial sciences and insurance mathematics and implies that an insurance company transfers part of its risk to a reinsurance company; see e.g., \cite{asmussen2010ruin}, \cite{kaas2008modern}. 
Say an insurance company faces a total claim $S$ that is the sum of $n$ individual claims  $X_i,\,i=1,\dots,n$. The insurance company pays the claim up to a level $z$, and the reinsurance company covers the remainder. This gives rise to the so-called retention function $\psi(z,S)=\min\{S,z\}$ that represents the payment of the insurer. We provide an upper bound for the standard stop-loss retention function in Appendix~\ref{EC:retention}.

The payment function of the reinsurance company puts forward a more challenging problem when the insurance coverage is limited. In this case, a relevant performance characteristic is to what extend the insurance company benefits from the reinsurance contract. This benefit is measured with the function
\begin{equation}
    \phi(z,S)=\begin{cases}
\ m, \quad &{\rm if }\ S\geq z+m, \\
\ S-z, \quad &{\rm if }\ z\leq S\leq z+m, \\
\ 0, \quad &{\rm if }\ S\leq z. \\
\end{cases}
\end{equation}
When the total claim $S$ stays below the retention limit $z$, the insurance company covers the entire claim, but when $S$ exceeds $z$ the reinsurer pays the excess claim up to a maximum $m$. Thus, the reinsurance company does not compensate large claims that exceed the exit point $m+z$. Above this level the risk is retained by the insurance company. 

We obtain a novel bound by using primal-dual arguments.

\begin{theorem}\label{th:stoplossrein}
The expected insurer's benefit is bounded by
\begin{equation}
    \sup\limits_{\mathbb{P}\in\mathcal{P}_{(\mu,b,d)}}\mathbb{E}_{\mathbb{P}}[\phi(z,S)]=\begin{cases}
    \ \min\{m, \frac{m}{m+z}(\mu - \frac{d(b-(m+z))}{2(b-\mu)})\}, \quad & {\rm if}\ z\leq m+z\leq\mu,\\
    \ \min\{m(1-\frac{d}{2\mu}),z(\frac{d}{2\mu}-1)+\mu\}, \quad & {\rm if}\ z\leq\mu\leq z+m\leq b,\\
    \ \min\{m(1-\frac{d}{2\mu}),\frac{d m}{2(m+z-\mu)}\}, \quad & {\rm if}\ \mu\leq z\leq z+m\leq b,\\
    \end{cases}
\end{equation}
where the function $\phi(z,S)$ degenerates to $\max\{S-z,0\}$ if $z+m> b$. In this case, 
\begin{equation}
    \sup\limits_{\mathbb{P}\in\mathcal{P}_{(\mu,b,d)}}\mathbb{E}_{\mathbb{P}}[\phi(z,S)]=\begin{cases}
    \ z(\frac{d}{2\mu}-1)+\mu, & {\rm if}\ z\leq\mu,\\
    \ \frac{d(b-z)}{2(b-\mu)},\quad & {\rm if}\ z\geq\mu.\\
    \end{cases}
\end{equation}
\end{theorem}
\begin{proof}
See Appendix~\ref{app: stoplossthms}. 
\end{proof}

An illustration of the bounds for the stop-loss payments is provided in Figure~\ref{fig:paymentceding}, where we display payments as functions of $z$ with $\mu=5$, $d=1.77$, and $m=3$, $m=5$ and $m\to\infty$. We assume that the `true' total claim $S$ follows a $\text{Poisson}(5)$ distribution. Note the resemblance of the stop-loss bounds to the mean-MAD tail probability bounds in Section~\ref{sec:novelbounds}. The former bounds, however, have an additional linear part with a negative slope for $z\leq\mu\leq z+m$. This linear segment is only present when $m$ exceeds  $d\mu/(2\mu-d)$; moreover, the bound approaches a linear function for $z\leq\mu$ when $m$ is chosen sufficiently large. Additionally, letting $m\to\infty$, our example results in a bound equal to the constant $d/2$ if $z\geq\mu$, and thus the bound for the stop-loss payment of the reinsurer degenerates to a piecewise linear function consisting of two parts (a linear part with negative slope $d/(2\mu)-1$ and a constant part equal to $d/2$).

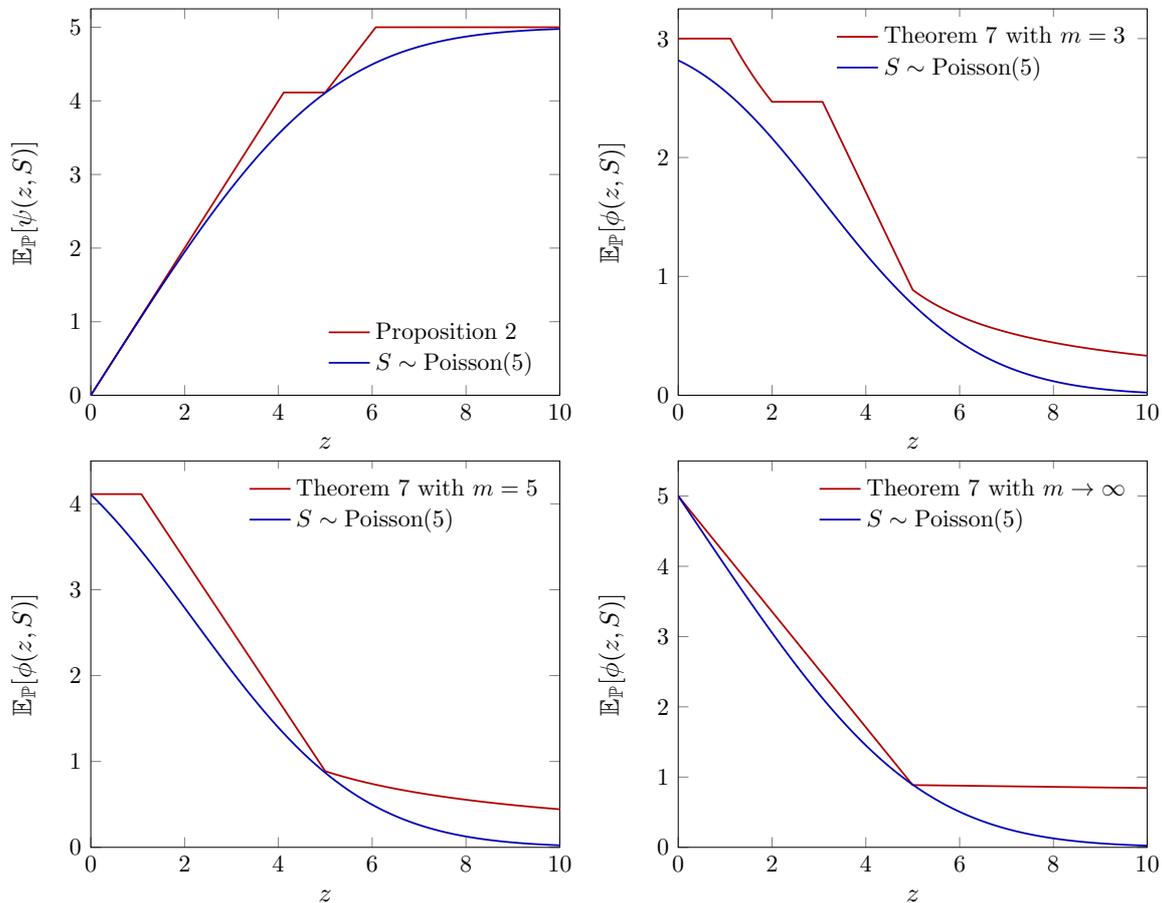
\begin{figure}[h!]
\begin{center}

\begin{tikzpicture}[yscale=0.9, xscale=0.9] 
\begin{axis}[
	xlabel={$z$},
	xmin=0, xmax=10,
	ymin=0, ymax=5.25,
    ylabel={$\mathbb{E}_{\mathbb{P}}[\psi(z,S)]$},
    ylabel style={yshift=0.2cm},
    ticklabel style={font=\small},
    xticklabel style={
    	/pgf/number format/precision=4,
    	/pgf/number format/fixed,
	},
    scaled x ticks=false,
    legend pos = south east,
    legend style={font=\small},
    legend style={draw=none},
    legend cell align={left}
]

\addplot[red!70!black, line width=0.75pt] table {Figure_Data/stoploss/stop2.dat};
\addlegendentry{Proposition~\ref{prop: stoplossceding}};
\addplot[blue!70!black, line width=0.75pt] table {Figure_Data/stoploss/stop1.dat};
\addlegendentry{$S\sim\text{Poisson}(5)$};
\end{axis}
\end{tikzpicture}%
~%
\begin{tikzpicture}[yscale=0.9, xscale=0.9] 
\begin{axis}[%
	xlabel={$z$},
	xmin=0, xmax=10,
	ymin=0, ymax=3.25,
    ylabel={$\mathbb{E}_{\mathbb{P}}[\phi(z,S)]$},
    ylabel style={yshift=0.2cm},
    ticklabel style={font=\small},
    xticklabel style={
    	/pgf/number format/precision=4,
    	/pgf/number format/fixed,
	},
    scaled x ticks=false,
    legend style={font=\small},
    legend style={draw=none},
    legend cell align={left}
]

\addplot[red!70!black, line width=0.75pt] table {Figure_Data/stoploss/stop4.dat};
\addlegendentry{Theorem~\ref{th:stoplossrein} with $m=3$};
\addplot[blue!70!black, line width=0.75pt] table {Figure_Data/stoploss/stop3.dat};
\addlegendentry{$S\sim\text{Poisson}(5)$};
\end{axis}
\end{tikzpicture}%

\begin{tikzpicture}[yscale=0.9, xscale=0.9] 
\begin{axis}[
	xlabel={$z$},
	xmin=0, xmax=10,
	ymin=0, ymax=4.5,
    ylabel={$\mathbb{E}_{\mathbb{P}}[\phi(z,S)]$},
    ylabel style={yshift=0.2cm},
    ticklabel style={font=\small},
    xticklabel style={
    	/pgf/number format/precision=4,
    	/pgf/number format/fixed,
	},
    scaled x ticks=false,
    legend style={font=\small},
    legend style={draw=none},
    legend cell align={left}
]

\addplot[red!70!black, line width=0.75pt] table {Figure_Data/stoploss/stop6.dat};
\addlegendentry{Theorem~\ref{th:stoplossrein} with $m=5$};
\addplot[blue!70!black, line width=0.75pt] table {Figure_Data/stoploss/stop5.dat};
\addlegendentry{$S\sim\text{Poisson}(5)$};
\end{axis}
\end{tikzpicture}%
~%
\begin{tikzpicture}[yscale=0.9, xscale=0.9] 
\begin{axis}[
	xlabel={$z$},
	xmin=0, xmax=10,
	ymin=0, ymax=5.5,
    ylabel={$\mathbb{E}_{\mathbb{P}}[\phi(z,S)]$},
    ylabel style={yshift=0.2cm},
    ticklabel style={font=\small},
    xticklabel style={
    	/pgf/number format/precision=4,
    	/pgf/number format/fixed,
	},
    scaled x ticks=false,
    legend style={font=\small},
    legend style={draw=none},
    legend cell align={left}
]

\addplot[red!70!black, line width=0.75pt] table {Figure_Data/stoploss/stop8.dat};
\addlegendentry{Theorem~\ref{th:stoplossrein} with $m\rightarrow\infty$};
\addplot[blue!70!black, line width=0.75pt] table {Figure_Data/stoploss/stop7.dat};
\addlegendentry{$S\sim\text{Poisson}(5)$};
\end{axis}
\end{tikzpicture}%

\end{center}
    \caption{The bounds and `true' values of the expected claim payment of the insurance company $\mathbb{E}_{\mathbb{P}}[\psi(z,S)]$ and the reinsurer $\mathbb{E}_{\mathbb{P}}[\phi(z,S)]$ as functions of the retention limit $z$ with $\mu=5$, $d=1.77$, and $m=3$, $m=5$ or $m\to\infty$. The red lines depict the upper bounds and the blue lines give the true expected claim payments when $S\sim\text{Poisson}(5)$.}
    \label{fig:paymentceding}
\end{figure}

These results complement the literature on tight bounds for expected claim payments. \cite{cox1991bounds}, considering bounded support and known first and second moment, obtains tight bounds using general results for moment problems. Other related works explore ways to sharpen the bounds using additional information. When modifying the ambiguity set by incorporating skewness information, imposing unimodality and symmetry conditions, or using higher order moments, the gap between the upper and lower bounds narrows significantly; see \cite{heijnen1990best}, \cite{devylder1982analytical}, and \cite{jansen1986upper}. Note that the mean-MAD information can easily be extended with skewness parameters, such as the probability $\beta=\prob{S\geq\mu}$ or the median. In our case it is also possible to impose unimodality and symmetry conditions by altering the dual problem. Section~4 of \cite{popescu2005semidefinite} discusses these modifications for general piecewise polynomial functions in the constraints of the dual problem. We discuss one such extension in the next section: the multivariate stop-loss reinsurance problem.

\section{More applications for sums and optimization}
\label{sec: applications2}

As alluded to in the introduction, the novel tail bounds are part of a much larger research effort within the area of distributionally robust optimization (DRO), trying to exploit the tractability that comes with mean-MAD ambiguity constraints. To show this connection with DRO, we first extend our tail probability bound to sums of random variables (that arise often in DRO applications) in Section \ref{sec: sums} and illustrate the effectiveness with an insurance example in Section~\ref{sec: insuranceexample}. Section \ref{sec: acc} then discusses the application of tail probability bounds to reformulate ambiguous chance constraints in DRO.
In Section \ref{sec: radiotherapy} we provide a realistic DRO example that arises in radiotherapy optimization.

\subsection{Sums of random variables}\label{sec: sums}

Widely used in probability theory and stochastic OR, sums of random variables find application in areas such as inventory management, service operations management, mathematical finance and credit risk. Mathematical techniques for sums or random variables are covered in many standard texts on probability theory  \citep{chung2001course,feller1971}. For sums of i.i.d.~random variables, variance then enters naturally (e.g.,~variance of the sum, central limit theorem), also for deriving tail bounds. The MAD of i.i.d. variables, on the other hand, cannot simply be summed. Leveraging the tight univariate tail bound, we establish a generic multivariate tail bound for sums of random variables. We do so without a specific application in mind, but with the goal of deriving broadly applicable distribution-free bounds.


Consider \(n\) random variables \(X_1, \ldots, X_n\) with known support, mean and MAD, and consider the worst-case tail probability
\begin{equation} \label{eq: joint tail prob}
\sup_{\bP \in \cF} \bP \Big( \sum_{i=1}^n X_i \geq t \Big) 
\end{equation}
with \(\cF\) is the multi-dimensional ambiguity set, i.e., 
\begin{equation} \label{eq: joint ambig}
\cF = \left\{ \bP : \mathcal{B}^n \rightarrow [0, 1] \relmiddle| \bP\left(X_i \in \left[0, b_i\right]\right) = 1, \: \bE_\bP \left[ X_i \right] = \mu_i, \: \bE_\bP \left[\left| X_i - \mu_i \right| \right] = d_i, \: i=1,\ldots,n \right\},
\end{equation}
and \(\mathcal{B}^n\) the \(n\)-dimensional Borel \(\sigma\)-algebra.
Note that we do not make any assumptions with regard to (in)dependence or correlation between the random variables. To analyze \eqref{eq: joint tail prob} we define a new random variable \(Y = \sum_{i=1}^n X_i\). Clearly, the support and mean of \(Y\) follow from \eqref{eq: joint ambig}:
\[\bP\Big(Y \in \Big[0, \; \sum_{i=1}^n b_i \Big] \Big) = 1, \qquad \qquad \bE_\bP \left[Y \right] = \sum_{i=1}^n \mu_i.\]
Unfortunately, the mean absolute deviation is unknown and applying Theorem~\ref{theorem:main} is thus not straightforward. To ease notation, we shall denote the sums of \(\mu_i\), \(b_i\) and \(d_i\), by \(\bar{b}\), \(\bar{\mu}\) and \(\bar{d}\), respectively. Theorem~\ref{thm: aggregate multi dim} presents a bound on \eqref{eq: joint tail prob} for any upper bound on the mean absolute deviation of \(Y\). 
\begin{theorem} \label{thm: aggregate multi dim}
Assume
\[\bE\Big[\Big| \sum_{i=1}^n X_i - \bar{\mu} \Big|\Big] \leq \hat{d}.\]
Then, 
\begin{equation} \label{eq: final expr joint tail}
 \sup_{\bP \in \cF} \bP \Big( \sum_{i=1}^n X_i \geq t \Big) \leq \begin{cases}   1, & {\rm if }\ t \leq \bar{\mu}, \\ \min \Big\{ \frac{\hat{d}}{2\left(t - \bar{\mu}\right)}, \; \frac{\bar{\mu}}{t} \Big\}, & {\rm if }\ t \in (\bar{\mu}, \bar{b}]. \end{cases}
 \end{equation}
\end{theorem}
\begin{proof}
We consider the following ambiguity set for the distribution of \(Y\):
\begin{equation} \label{eq: aggregated ambig}
\cG = \{ \bP \mid \bP\left(Y \in \left[ 0, \; \bar{b} \right] \right) = 1, \enskip \bE_\bP \left[Y \right] = \bar{\mu}, \enskip \bE\left[ \left| Y - \bar{\mu} \right| \right] \leq \hat{d} \}.
\end{equation}
It should be noted that \eqref{eq: aggregated ambig} is not an ambiguity set of the form \eqref{eq: mean mad ambiguity}, because of the inequality for the mean absolute deviation. This also means that the tightest upper bound this approach can obtain for \(t \in \left[0, \bar{\mu}\right]\) is 1, as the distribution with probability mass 1 on \(\bar{\mu}\) is an element of \(\cG\). For \(t \in \left(\bar{\mu}, \bar{b}\right]\), we can however infer from Theorem \ref{theorem:main} that
\[\sup_{\bP \in \cF} \bP \Big(\sum_{i=1}^n X_i \geq t \Big) \leq \sup_{\bP \in \cG} \bP \left( Y \geq t\right) \hspace{6.28cm} \]
\[\hspace{2.7cm} = \begin{cases} 1, & \text{ if } t \leq \bar{\mu}, \\ \max_{d \in \left[0, \hat{d}\right]} \Big\{ \min \Big\{\frac{d}{2\left(t - \bar{\mu}\right)}, \: 1 - \frac{d}{2\bar{\mu} } \Big\} \Big\}, & \text{ if } t \in (\bar{\mu}, \bar{b}].  \end{cases}\]
We will now simplify this expression by solving the maximization problem over \(d\) explicitly. To that end, we first note that the minimum is taken over two linear functions of \(d\), an increasing and a decreasing one. Therefore, the global maximum is at the intersection of these functions, and the optimal \(d\) is thus either \(\hat{d}\) or \(d^*\), where \(d^*\) is such that
\[\frac{d^*}{2\left(t - \bar{\mu}\right)} = 1 - \frac{d^*}{2 \bar{\mu} }.\]
Solving this equation yields
\[d^* = \frac{2 \bar{\mu} \left(t - \bar{\mu}\right)}{ t }.\]
We remark that \(\hat{d}\) is the optimal solution when
\[\frac{\hat{d}}{2\left(t - \bar{\mu}\right)} < 1 - \frac{\hat{d}}{2\bar{\mu} },\]
as this means that \(\hat{d} < d^*\). Therefore, we find that
\[\max_{d \in \left[0, \hat{d}\right]} \Big\{ \min \Big\{\frac{d}{2\left(t - \bar{\mu}\right)}, \: 1 - \frac{d}{2 \bar{\mu} } \Big\} \Big\} = \min \Big\{ \frac{\hat{d}}{2\left(t - \bar{\mu}\right)}, \; \frac{d^*}{2\left(t - \bar{\mu}\right)} \Big\} = \min \Big\{ \frac{\hat{d}}{2\left(t - \bar{\mu}\right)}, \; \frac{\bar{\mu} }{\tau } \Big\}. \]
\end{proof}

The most obvious candidate for \(\hat{d}\) is given by the sum of all mean absolute deviations \(\bar{d}\), which is clearly an upper bound as \citep{postek2018robust}:
\[\bE\Big[\Big| \sum_{i=1}^n X_i - \bar{\mu} \Big|\Big] = \bE \Big[ \Big| \sum_{i=1}^n X_i - \mu_i \Big| \Big] \leq \sum_{i=1}^n \bE \left[\left|X_i - \mu_i \right| \right] = \bar{d}.\]
This bound, however, is generally not tight. It is tight, for example, when \(\mu_i\), \(b_i\) and \(d_i\) are equal for all \(i = 1, \ldots, n\). We will use this bound in the remainder of this section. Several other possible bounds that can be used are described by~\cite{postek2018robust}.


The allowed correlation structure is convenient in many situations. Take for instance the portfolio loss in credit risk, traditionally modeled as
			${L}_{n} =\sum_{i=1}^n X_i$, 				
with $X_1,\ldots, X_n$ the losses (due to default) of the individual obligors \citep{glasserman2005importance}. 
The standard scenario in credit risk is that losses are positively correlated, allowing $ L_n$ to assume relatively large values, which can be measured in terms of
 {Value-at-Risk}, the $\alpha$ quantile of the loss distribution, i.e.~$
{\rm VaR}_\alpha := \inf\{ t \ge 0 : \mathbb{P}[L_n \le t] \ge \alpha \}
$. The multivariate tail bound can be translated directly into bounds for {Value-at-Risk}. 


\subsection{Insurance portfolio example}\label{sec: insuranceexample}

Consider an insurer that holds a portfolio that can incur random losses $X_1,...,X_n$, which correspond to different types of insurance claims. The insurer considers the cumulative value of the claims $X_1,...,X_n$ and the probability that this value exceeds the available capital $t$; that is, the insurer is interested in the ruin probability, which is given by
\begin{equation}
    \mathbb{P}(X_1+\dots+X_n \geq t),
\end{equation}
where the distribution $\bP$ lies in $\cF$. Similar to the portfolio loss in credit risk, the allowed dependence structure proves useful when considering insurance problems with catastrophic risks.

Figure~\ref{fig: multi-dim example} demonstrates the multivariate bound with an experiment inspired by \citet{vanParys2016frechet}, Section~6. Suppose that an insurance company sells three insurance policies. The claims are modeled by random variables $X_i,\ i=1,2,3$. The insurance company only has information about their support, mean, and MAD. Assume that the losses are lognormally distributed with location parameter $\bar{m}_i$ and scale parameter $v_i$. The parameters have the following values: $\bar{m}_1=-0.3,\ v_1=0.8, \ \bar{m}_2=0.4,\ v_2=0.5, \ \bar{m}_3=0.8,$ and $v_3=0.5$. The insurer has reliably estimated the mean and MAD of these distributions. Using bounds for sums of random variables with mean-MAD information, we can provide a conservative bound for the probability of the event that the total claim exceeds $t$. In Figure~\ref{fig: multi-dim example} we show the mean-MAD bound and the actual ruin probabilities for two different dependence structures, i.e., the independence and comonotonic copula which couple the $X_i,\ i=1,2,3$. Sums of random variables are known to be the `riskiest' when they are comonotonic; that is, the terms grow simultaneously and hence a component can in no way hedge another one. The joint cumulative distribution of $X_1,\dots,X_n$ attains the so-called Fr\'echet-Hoeffding upper bound \citep{kaas2008modern}:
$$
\prob{X_1\leq x_1,\dots,X_n\leq x_n}=\min\limits_{i=1,\dots,n}\prob{Y_i\leq x_i},
$$
where $Y_i\sim X_i,\,i=1,\dots,n,$ which entails that these are, in a certain sense, the most related variables. Notice that our novel conservative tail bound holds for all dependence structures, and it also covers heavy-tailed distributions with finite MAD. This is shown in Figure~\ref{fig: multi-dim claim example}. Estimating dependence between insurance claims is often practically infeasible with only a limited amount of data; see \citet{mcneil2015quantitative} for a comprehensive discussion on this topic. It is doable, however, to estimate the mean and MAD accurately even when data is scarce.

\begin{figure}[h!]
\begin{center}
\begin{tikzpicture}
\begin{axis}[
	xlabel={\(t\)},
	xmin=5.5, xmax=15,
	ymin=0, ymax=1,
    ylabel={\(\mathbb{P}\left(X_1 + X_2+X_3 \geq t\right)\)},
    ylabel style={yshift=0.2cm},
    label style={font=\small},
    ticklabel style={font=\small},
    xticklabel style={
    	/pgf/number format/precision=4,
    	/pgf/number format/fixed,
	},
    scaled x ticks=false,
    legend style={draw=none, font=\small},
    legend cell align={left}
]
\addplot[green!70!black, line width=0.75pt] table {Figure_Data/multiDim/dim1.dat};
\addlegendentry{independent};
\addplot[red!70!black, line width=0.75pt] table {Figure_Data/multiDim/dim2.dat};
\addlegendentry{comonotonic};
\addplot[blue!70!black, line width=0.75pt] table {Figure_Data/multiDim/dim3.dat};
\addlegendentry{Theorem~\ref{thm: aggregate multi dim}};
\end{axis}
\end{tikzpicture}
\end{center}
\caption{An example of the multi-dimensional bound for a risk aggregation application. The red and green line are obtained through simulation. \label{fig: multi-dim example}}
\end{figure}
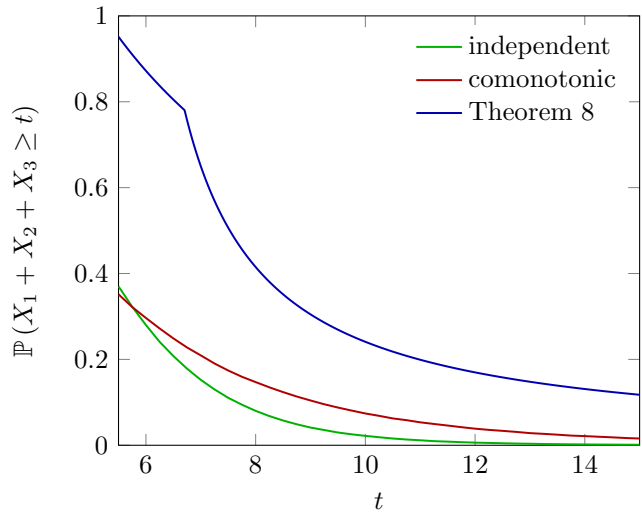


In the previous section we assumed that the distributional parameters of $S$ are known. In practice, however, the insurer relies on claim data of individual contracts. The total loss often consists of several different types of insurance contracts, e.g., the three claim types mentioned in the example above. Assume next that we have exact knowledge of the mean $\mu_i$, MAD $d_i$, and support of each of these losses $X_1,\ldots,X_n$. We are now interested in evaluating
\begin{equation}\label{eq: joint claim}
    \sup\limits_{\mathbb{P}\in\mathcal{F}}\bE_{\bP}\Big[\phi\Big(z,\sum_{i=1}^n X_i\Big)\Big].
\end{equation}


Note that we again do not make any assumptions with respect to the dependence structure of the random variables. To evaluate \eqref{eq: joint claim} we sum them all together: \(S_n \coloneqq \sum_{i=1}^n X_i\). The support and mean of the aggregate risk \(S_n\) are given by
\[\bP\Big(S_n \in \Big[0, \; \sum_{i=1}^n b_i \Big] \Big) = 1, \qquad \qquad \bE_\bP [S_n ] = \sum_{i=1}^n \mu_i.\]
The mean absolute deviation is again unknown, and hence we follow an approach similar to the derivation of the multi-dimensional tail probability bound. We also adopt the same notation. Theorem~\ref{thm: aggregateclaim} presents an upper bound of \eqref{eq: joint claim} for any upper bound on the mean absolute deviation of \(S_n\).

\begin{theorem} \label{thm: aggregateclaim}
For any \(\hat{d}\) such that
\[\bE\Big[\Big| \sum_{i=1}^n X_i - \bar{\mu} \Big|\Big] \leq \hat{d},\]
it holds that
\begin{equation} \label{eq: final expr joint claim}
 \sup\limits_{\mathbb{P}\in\mathcal{F}}\bE_{\bP}[\phi(z,S_n)] \leq \begin{cases}
    \ m, \quad & {\rm if}\ z\leq m+z\leq\bar{\mu},\\
    \ \min\{\frac{m\bar{\mu}}{m+z},z(\frac{\hat{d}}{2\bar{\mu}}-1)+\bar{\mu}\}, \quad & {\rm if}\ z\leq\bar{\mu}\leq z+m\leq \bar{b},\\
    \ \min\{\frac{m\bar{\mu}}{m+z},\frac{\hat{d} m}{2(m+z-\bar{\mu})}\}, \quad & {\rm if}\ \bar{\mu}\leq z\leq z+m\leq \bar{b}.\\
    \end{cases}
 \end{equation}
\end{theorem}
\begin{proof}
See Appendix~\ref{app: stoplossthms}. 
\end{proof}


Figure~\ref{fig: multi-dim claim example} displays our bound applied to a stop-loss reinsurance contract. As before, we model the three claims as random variables $X_i,\ i=1,2,3,$ for which the insurance company only has information regarding the support, mean, and MAD. For the sake of comparison, assume that the `true' losses are lognormally distributed with the aforementioned parameter values and that the insurer has reliably estimated the mean and MAD. Using the bounds for sums of random variables with mean-MAD information, we can provide a conservative bound for the expected stop-loss payment given the retention limit $z$. In Figure~\ref{fig: multi-dim claim example} we show the mean-MAD bound and the `true' expected claim payments for two different dependency structures. We again consider the independence and comonotonic copulas as the underlying structures coupling the losses $X_i,\ i=1,2,3$.

\begin{figure}[h!]
\begin{center}
\begin{tikzpicture}
\begin{axis}[
	xlabel={\(z\)},
	xmin=0, xmax=15,
	ymin=0, ymax=5.5,
    ylabel={\(\bE_{\bP}\Big[\phi\Big(z,\sum_{i=1}^3 X_i\Big)\Big]\)},
    ylabel style={yshift=0.2cm},
    xticklabel style={
    	/pgf/number format/precision=4,
    	/pgf/number format/fixed,
	},
    scaled x ticks=false,
	yticklabel style={/pgf/number format/precision=4, /pgf/number format/fixed},
	scaled y ticks=false,
    label style={font=\small},
    ticklabel style={font=\footnotesize},
    legend style={draw=none, font=\footnotesize},
    legend cell align={left}
]
\addplot[green!70!black, line width=0.75pt] table {Figure_Data/stoploss/stop9.dat};
\addlegendentry{independent};
\addplot[red!70!black, line width=0.75pt] table {Figure_Data/stoploss/stop10.dat};
\addlegendentry{comonotonic};
\addplot[blue!70!black, line width=0.75pt] table {Figure_Data/stoploss/stop11.dat};
\addlegendentry{Theorem~\ref{thm: aggregateclaim}};
\end{axis}
\end{tikzpicture}
\end{center}
\caption{An example of the multi-dimensional bound for the stop-loss risk aggregation application with $m=10$. The red and green line are obtained through simulation. \label{fig: multi-dim claim example}}
\end{figure}
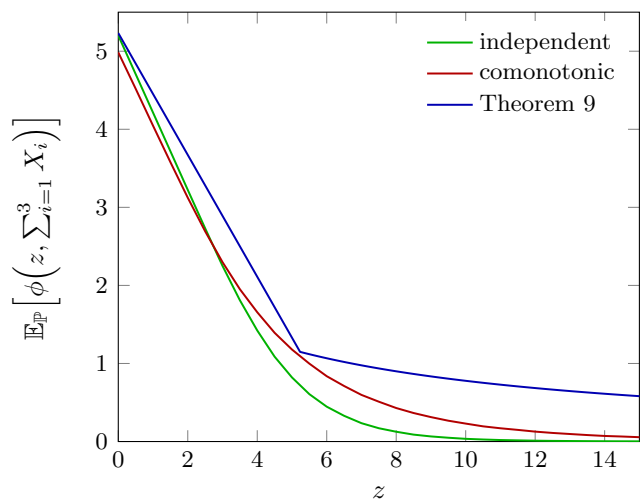

\subsection{Ambiguous chance constraints} \label{sec: acc} 
A large class of decision problems in OR can be formulated as optimization problems of the form 
\begin{align*}
    \min_{\mathbf{x}} \enskip & f(\mathbf{x}) \\
    \text{s.t.} \enskip       & g_i(\mathbf{x}, \bm{Z}) \leq 0 && i=1,\ldots,m,
\end{align*}
for some convex functions \(f\) and \(g_1, \ldots, g_m\). Here, \(\mathbf{x}\) denotes the decision variable, while \(\bm{Z }\) is some given parameter. In many applications, \(\bm{Z}\) is uncertain and the constraints are often replaced by chance constraints, i.e., for some accepted risk level \(\epsilon \in (0, 1)\), it is instead required that
\[\bP\left(g_i (\mathbf{x}, \bm{Z}) \leq 0 \right) \geq 1 - \epsilon,\]
for each \(i=1,\ldots,m\). This type of chance constraint is referred to as a single chance constraint, while a single probabilistic constraint on all constraints being satisfied simultaneously is known as a joint chance constraint. 
Examples of such applications include, but are not limited to finance \citep{dert2000optimal}, network design \citep{wang2007beta} and call-center staffing \citep{gurvich2010staffing}. Chance constraints suffer from tractability issues, however, and additionally require an exact specification of the distribution. Recently, therefore, there has been an emerging interest in ambiguous chance constraints, in which the distribution of the uncertain parameters is not fully specified. A single ambiguous chance constraint takes the form 
\begin{equation} \label{eq: general acc}
    \inf_{\bP \in \cP} \bP\left( g(\mathbf{x}, \bm{Z}) \leq 0 \right) \geq 1 - \epsilon,
\end{equation}
for some ambiguity set \(\cP\). The primary goals in analyzing such constraints are (i) determining under which conditions \eqref{eq: general acc} defines a convex set in \(\mathbf{x}\) and (ii) finding a representation and/or approximation of this set in terms of simple convex inequalities. Often, such conditions and representations are prohibitively hard to find when considering joint ambiguous chance constraints. We therefore contain our discussion to single ambiguous chance constraints.

One way to find convex reformulations of single ambiguous chance constraints is the use of classical probability inequalities. Hoeffding's inequality, for example, has been used by \citet{ben2000robust} and \citet{bertsimas2005optimal} to derive approximations and reformulations of \eqref{eq: general acc}, respectively, when \(g\) is linear and the components of \(\bm{Z}\) are independent, symmetric, and bounded. Building on that, \citet{bertsimas2019probabilistic} use sub-Gaussian theory to derive safe approximations under the same assumptions on \(\bm{Z}\), and the relaxed assumption that \(g\) is concave in \(\bm{Z}\).
Similarly, a generalized Chebyshev inequality is used by \citet{xu2012distributional} to find convex reformulations, while \citet{nemirovski2007convex} and \citet{postek2018robust} use Bernstein bounds to derive convex approximations. 
The tail probability bounds we derive also allow for such methods to be applied. We discuss convex reformulations of single chance constraints in which the uncertainty is present as a single random variable. 
Specifically, we consider the case where \(m = 1\), i.e., there is only a single uncertain parameter \(Z\), and use our main result from Theorem \ref{theorem:main} to reformulate the semi-infinite constraint \eqref{eq: general acc} into a convex constraint for certain forms of \(g\).



We first present a convex reformulation of an ambiguous chance constraint when \(g\) is convex in \(\mathbf{x}\) and affine in \(Z\). This case is often referred to as right-hand side uncertainty.

\begin{theorem} \label{thm: convex rhs cc}
Let \(\tilde{g} : \bR^n \rightarrow \bR\) and let \(Z\) be a \(1\)-dimensional random variable whose distribution lies in the ambiguity set
\[\cP = \left\{\bP : \bP \left[Z \in \left[-1, 1\right]\right] = 1, \enskip \bE \left[Z \right] = 0, \enskip \bE \left[ \left| Z \right| \right] = d \right\},\]
for some \(d \in \left[0, 1\right]\).
For any \(\epsilon \in \left(0, \frac{1}{2}\right)\) and \( \mathbf{x} \in\bR^n\) it holds that
\begin{equation} \label{eq: convex rhs cc}
\inf_{\bP \in \cP} \bP \left[\tilde{g}(\mathbf{x}) + Z \leq 0 \right] \geq 1 - \epsilon,
\end{equation}
if and only if
\begin{equation} \label{eq: convex rhs cc final}
\tilde{g}(\mathbf{x}) + \min\left\{1, \; \frac{d}{2\epsilon} \right\} \leq 0.
\end{equation}
\end{theorem}
\begin{proof}
We first rewrite \eqref{eq: convex rhs cc} to
\[\sup_{\bP \in \cP} \bP \left[ Z > - \tilde{g}(\mathbf{x}) \right] \leq \epsilon.\]
From Theorem \ref{theorem:main} and the fact that \(\epsilon < \frac{1}{2}\) we know that it must hold that \(- \tilde{g}(\mathbf{x}) > \bE \left[ Z \right] = 0\). Given that requirement, we know by Theorem \ref{theorem:main} that 
\[\sup_{\bP \in \cP} \bP \left[Z > - \tilde{g}(\mathbf{x}) \right] = \begin{cases} \min \left\{\frac{d}{- 2 \tilde{g}(\mathbf{x}) }, \: 1 - \frac{d}{2}\right\} & \text{ if } - \tilde{g}(\mathbf{x}) < 1 \\
0 & \text{ if } - \tilde{g}(\mathbf{x}) \geq 1. \end{cases}\]
From \(d \in \left[0, 1\right]\) and \(\epsilon \in \left(0, \frac{1}{2}\right)\), it follows that \(1 - \frac{d}{2} > \epsilon\), and thus any feasible solution \(\mathbf{x}\) must satisfy \(- \tilde{g}(\mathbf{x}) \geq 1\) and/or \(\frac{d}{- 2 \tilde{g}(\mathbf{x})} \leq \epsilon\). The latter can be equivalently stated as
\[- \tilde{g}(\mathbf{x}) \geq \frac{d}{2 \epsilon},\]
which can easily be combined with the former as
\[- \tilde{g}(\mathbf{x}) \geq \min\left\{1, \; \frac{d}{2 \epsilon} \right\} \iff \tilde{g}(\mathbf{x}) + \min \left\{ 1, \; \frac{d}{2 \epsilon} \right\} \leq 0.\]
Because \(\frac{d}{2 \epsilon} > 0\), we find that the requirement \( - \tilde{g}(\mathbf{x}) > 0\) is redundant, and thus \eqref{eq: convex rhs cc} is equivalent to \eqref{eq: convex rhs cc final}. 
\end{proof}

We note that assuming \(d \in \left[0, 1\right]\) is equivalent to assuming that \(\cP_{(\mu, b, d)}\) is nonempty. Moreover, we note that we cannot assume a support of \(\left[-1, 1\right]\) and a mean of 0 without loss of generality, as this implies that the support of \(Z\) is symmetric around the mean. It is straightforward, however, to extend our results to an ambiguity set with support \(\left[-1, u \right]\) for some \(u > 0\), which can be assumed without loss of generality. 

A similar reasoning to that in Theorem \ref{thm: convex rhs cc} can be applied to joint chance constraints with independent right-hand side uncertainty. Here, because of the use of the support information of the random variable, we do not provide an exact reformulation, but a safe approximation instead. 
Providing a tractable reformulation when uncertainty is present in the left-hand side, on the other hand, is significantly more complicated. Using our results, we provide a reformulation when g is bilinear in \(\mathbf{x}\) and \(Z\). For the sake of conciseness, the first two results as well as the proof of the result below are included in Appendix \ref{app: acc results}.





\begin{theorem} \label{thm: single lhs cc}
Let \(\bar{\mathbf{a}}, \hat{\mathbf{a}} \in \bR^n\), \(h \in \bR\) and \(Z \in \bR\) be a random variable whose distribution lies in the ambiguity set
\[\cP = \left\{\bP : \bP \left[ Z \in \left[-1, 1\right]\right] = 1, \enskip \bE \left[ Z \right] = 0, \enskip \bE \left[ \left| Z \right| \right] = d \right\},\]
for some \(d \in \left[0, 1 \right]\). For any \(\epsilon \in \left(0, \frac{1}{2}\right)\) and \(\mathbf{x} \in \bR^n\) it holds that
\begin{equation} \label{eq: lhs affine cc}
\inf_{\bP \in \cP} \bP \left[\left(\bar{\mathbf{a}} + Z \hat{\mathbf{a}} \right)^\top \mathbf{x} \leq h \right] \geq 1 - \epsilon,
\end{equation}
if and only if
\begin{equation} \label{eq: lhs affine cc final}
\bar{\mathbf{a}}^\top \mathbf{x} + \min \left\{1, \; \frac{d}{2\epsilon}\right\} \cdot |\hat{\mathbf{a}}^\top \mathbf{x} | \leq h.
\end{equation}
\end{theorem}

Observe that \eqref{eq: lhs affine cc final} has a linear representation, and optimization problems containing ambiguous chance constraints of the form \eqref{eq: lhs affine cc} can thus be solved very efficiently. Also observe that our results extend to any constraint that consists of a bilinear term in \(\bm{x}\) and \(Z\) and any other convex term independent of \(Z\).

For the sake of conciseness, we only presented convex reformulations for two types of ambiguous chance constraints in this section. The tail probability bound derived in Theorem \ref{theorem:main} can be applied to derive convex reformulations and safe approximations to other ambiguous chance constraints as well. 

We mention two related works to our discussion on ambiguous chance constraints. \citet{hanasusanto2017ambiguous} present a tractable framework for joint ambiguous chance constraints under a few simplifying conditions. In particular, they assume a conic, hence unbounded, support, which is a key difference to our approach. Their approach is very powerful in settings for which an unbounded support makes sense, however, as they are able to elegantly deal with \emph{joint} ambiguous chance constraints as well.
\citet{xie2018deterministic}, on the other hand, consider ambiguous chance constraints given a bounded support and moment information. Their assumptions on the ambiguity set do, however, exclude exact distributional information on nonlinear functions of the uncertain parameter, which we do assume in exact knowledge of the mean absolute deviation. 

\subsection{Optimization problem from radiotherapy} \label{sec: radiotherapy}

We now illustrate these implications by applying our result to an optimization problem that arises in radiotherapy. 
Here, the biological effective radiation dose delivered to a tumor is to be maximized subject to a constraint on the biological effective dose delivered to the surrounding healthy tissue. Mathematically, the biological effective dose (BED) for a dose \(\mathbf{x} \in \bR^n\) delivered over \(n\) fractions is given by
\[B(\mathbf{x}) = \sum_{t=1}^n x_t + \frac{1}{\rho} x_t^2,\]
where \(\rho\) is the radiosensitivity parameter of the irradiated tissue. 
More specifically, 
it can be interpreted as the tissue's sensitivity to fractionation, where a low value indicates a high sensitivity to fractionation, i.e., the distribution of treatment over multiple fractions. 

While there is an extensive body of research on the value of \(\rho\) for different tumor sites, it remains subject to significant uncertainty \citep{joiner2016basic}. Moreover, since this value can differ from patient to patient, there is a very limited amount of data available and there is little evidence to suggest it follows some well known distribution. Throughout the rest of the example, we denote the sensitivity to fractionation by \(\rho_1\) and \(\rho_2\) for the tumor and the surrounding healthy tissue, respectively.

For illustrative purposes, we consider a setting in which it has been decided to deliver the treatment over two fractions, i.e., the optimization variables are limited to the dose in the first and second fraction. Moreover, we focus on the uncertainty of \(\rho_2\), and thus model the restriction of sparing the healthy tissue through an ambiguous chance constraint. Mathematically, we wish to solve the following optimization problem \citep{eikelder2019adjustable}:
\begin{subequations} \label{eq: radiotherapy acc}
    \begin{align}
        \max_{\mathbf{x} \in \bR^2} \enskip & x_1 + x_2 + \frac{1}{\rho_1} \left( x_1^2 + x_2^2 \right) \\
        \text{s.t.} \enskip & \bP\left( \sigma (x_1 + x_2) + \frac{1}{\rho_2} \sigma^2 \left(x_1^2 + x_2^2 \right) \leq t(\rho_2) \right) \geq 1 - \epsilon && \forall \bP \in \cP_{(\mu, b, d)} \label{eq: radiotherapy acc b}\\
        & x_1, x_2 \geq x_{min},
    \end{align}
\end{subequations}
where \(\sigma\) is the generalized dose-sparing factor that denotes the fraction of the mean tumor dose that the healthy tissue receives on average, \(x_{min}\) is the minimum dose that must be delivered in each fraction, and \(t(\rho_2)\) denotes the tolerance level of the healthy tissue and is given by
\[t(\rho_2) = \phi D \left(1 + \frac{\phi D}{T} \frac{1}{\rho_2} \right).\]
In other words, the healthy tissue is known to tolerate a total dose of \(D\) gray if it is delivered in \(T\) fractions under dose shape factor \(\phi\). This dose shape factor is a parameter that characterizes the spatial heterogeneity of a dose distribution \citep{perko2018derivation}. 

The ambiguity of \(\rho\) is modeled through the mean-MAD ambiguity set, where the lower bound of the support is given by \(a\) instead of 0. The ambiguous chance constraint \eqref{eq: radiotherapy acc b} is not naturally stated in a form that Theorem \ref{thm: convex rhs cc} or \ref{thm: single lhs cc} can be applied to. It can be rewritten, however, as
\begin{equation} \label{eq: reformulation radiotherapy acc}
\bP\left(\rho \cdot \left(\sigma(x_1 + x_2) - \phi D \right) > \frac{\phi^2 D^2}{T} -  \sigma^2 (x_1^2 + x_2^2 )\right) \leq \epsilon \qquad \qquad \forall \bP \in \cP_{(\mu, b, d)},
\end{equation}
where we note multiplication by \(\rho\) is allowed as its support is nonnegative. Leveraging the tail probability bound, we find for \(\epsilon \in (0, \frac{\mu - a}{b - a} )\) that \eqref{eq: reformulation radiotherapy acc} is equivalent to
\[\mu \sigma (x_1 + x_2) + \sigma^2 (x_1^2 + x_2^2) + \frac{d}{2 \epsilon} \left| \sigma (x_1 + x_2) - \phi D \right| \leq \mu \phi D + \frac{\phi^2 D^2}{T}.\]
We solve \eqref{eq: radiotherapy acc} for a specific, realistic set of parameters taken from \citet{eikelder2019adjustable}, which are reported in Table \ref{tab: radiotherapy parameters}.
\begin{table}[h]
    \centering
    \begin{tabular}{lr}
        \toprule
        Parameter   & Value \\ \midrule
        \(\tau\)    & 4 \\
        \(\sigma\)  & 0.9 \\
        \(\phi\)    & 2 \\
        \(D\)       & 27 \\
        \(T\)       & 5 \\
        \(x_{min}\) & 1.5 \\
        \(a\)       & 3 \\
        \(b\)       & 6 \\
        \(\mu\)     & 4 \\
        \(d\)       & 0.25 \\
        \bottomrule
    \end{tabular}
    \caption{Parameter values used for solving \eqref{eq: radiotherapy acc}.}
    \label{tab: radiotherapy parameters}
\end{table}
Figure \ref{fig: radiotherapy feas region} shows the feasible region and optimal solution of \eqref{eq: radiotherapy acc} for different values of \(\epsilon\) as well as the feasible region when we assume having the exact knowledge that \(\rho = \mu\). Remarkable in this example is the similarity between the feasible region of the problem without uncertainty and that of the ambiguous problem for \(\epsilon = 0.1\) and \(\epsilon = 0.05\). From the feasible region for \(\epsilon = 0.01\), however, it is clear that requiring that a low risk of violation results in a solution that is much worse in terms of tumor BED. It does, on the other hand, illustrate how the shape of the feasible region changes with \(\epsilon\): the feasibility of unbalanced solutions, i.e., solutions that administer a different dose in the two fractions, is impacted much more severely than that of balanced solutions. 
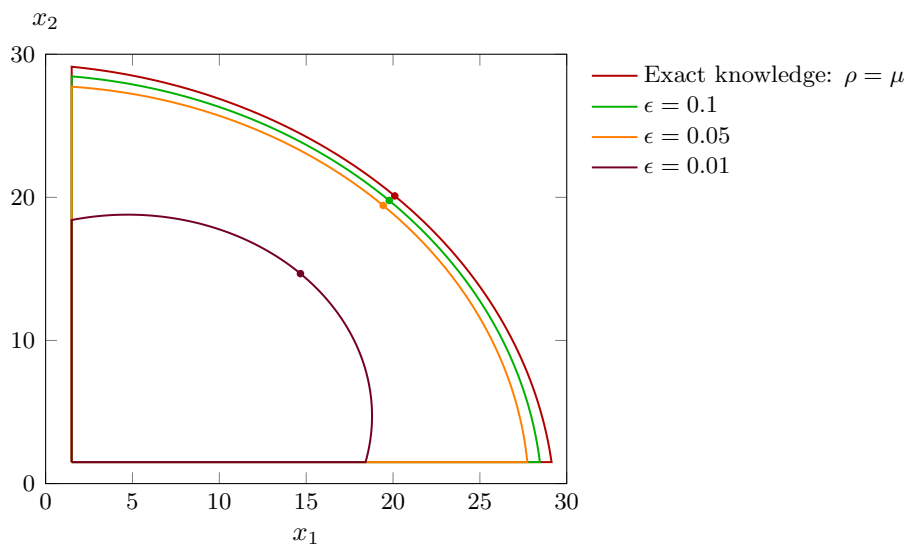
\begin{figure}[h!]
    \centering
    \begin{tikzpicture}
    \begin{axis}[
    	xlabel={$x_1$},
    	xmin=0, xmax=30,
    	ymin=0, ymax=30,
        ylabel={$x_2$},
        y label style={at={(current axis.above origin)},rotate=-90,anchor=south,yshift=0.2cm},
        xticklabel style={
        	/pgf/number format/precision=4,
        	/pgf/number format/fixed,
    	},
    	yticklabel style={/pgf/number format/precision=4, /pgf/number format/fixed},
        label style={font=\small},
        ticklabel style={font=\footnotesize},
        legend style={draw=none, font=\footnotesize},
        legend cell align={left},
        legend pos = outer north east
    ]
    \addplot[red!70!black, line width=0.75pt] table {Figure_Data/ACC/Radio_Feasible_Nominal.dat};
    \addlegendentry{Exact knowledge: \(\rho = \mu\)};
    \node at (20.0998, 20.0998)[red!70!black,circle,fill,inner sep=1pt]{};
    
    \addplot[green!70!black, line width=0.75pt] table {Figure_Data/ACC/Radio_Feasible_High_Epsilon.dat};
    \addlegendentry{\(\epsilon = 0.1\)};
    \node at (19.7795, 19.7795)[green!70!black,circle,fill,inner sep=1pt]{};
    
    \addplot[orange, line width=0.75pt] table {Figure_Data/ACC/Radio_Feasible_Medium_Epsilon.dat};
    \addlegendentry{\(\epsilon = 0.05\)};
    \node at (19.4323, 19.4323)[orange,circle,fill,inner sep=1pt]{};
    
    \addplot[purple!60!black, line width=0.75pt] table {Figure_Data/ACC/Radio_Feasible_Low_Epsilon.dat};
    \addlegendentry{\(\epsilon = 0.01\)};
    \node at (14.6704, 14.6704)[purple!60!black,circle,fill,inner sep=1pt]{};

    \end{axis}
    \end{tikzpicture}
    \caption{The feasible region and optimal solution of \eqref{eq: radiotherapy acc} for different values of \(\epsilon\) as well as exact knowledge that \(\rho = \mu\). Dots indicate the optimal solution.}
    \label{fig: radiotherapy feas region}
\end{figure}


\section{Conclusion and outlook}

Tail probabilities are ubiquitous in probabilistic studies in many areas of science and application domains. As the original Chebyshev's inequality for mean-variance ambiguity, we expect our novel tail bounds for mean-MAD ambiguity to find many applications. 

In our search for tight bounds under limited information, we had to solve for the worst-case distribution and worst-case value of the expectation of the indicator function $\1\{ X\geq t \}$. In this paper the limited information was captured ambiguity sets $\mathcal{P}_{(\mu, b, d)}$ and $\mathcal{P}_{(\mu, b, d,\beta)}$, and it turned out that the combination of the non-convex indicator function with these ambiguity set gave rise to semi-infinite linear programs with easy, closed-form solutions. 

In future work, we expect to find more such solvable classes, i.e.~specific combinations of objective function (other than the indicator function) and ambiguity sets that together give rise to solvable liner programs and hence easy extremal distributions. In this way, one could try to sharpen the tail bounds by including more information (e.g.~higher moments or percentiles), or to consider objective functions other than the tail probability. Our proof method based on solving the dual problem with piecewise-linear majorants is not tailor-made for  the indicator function, and could potentially work for a much larger class of (measurable) objective functions. Another direction we shall pursue is the application of the bounds to more complex, and possible high-dimensional robust optimization problems. To do so, we shall leverage the connection with the quickly evolving field of DRO, as illustrated by examples in Section \ref{sec: applications2}. Indeed, minmax and maxmin decision problems arise naturally, and the bounds and proof techniques can help in advancing that field.

\ACKNOWLEDGMENT{We thank Vera van den Dool (ORTEC) for her ideas and efforts in the preliminary stages of this research and for sparking our interest in the topic, and Jacco van Eijk for communicating the proof of Theorem \ref{mainmadthm}. The research of the first author was funded by the Netherlands Organisation for Scientific Research (NWO) Research Talent [Grant 406.17.511].}

\bibliographystyle{apalike} 
\bibliography{References,bibbook}

\appendix

\section{Proofs of tail bounds}\label{EC:betaproofs}
\begin{proof}[Proof Theorem~\ref{theorem:betaUB}]
We will show that additional information on a particular instance of the tail distribution (e.g., $\beta=\prob{X\geq\mu}$) results in tighter bounds. We again consider the Borel measurable function $\1_{\{x\geq t\}}$. Under $\mathcal{P}_{(\mu,b,d,\beta)}$ ambiguity of the random variable $X$ we now need to solve
\begin{equation}\label{eq:primal2}
\begin{aligned}
&\! \sup_{\mathbb{P}\in \mathcal{M}^+} &  &\int_x \1_{\{x\geq t\}}{\rm d} \mathbb{P}(x)\\
&\text{s.t.} &      & \int_x {\rm d}\mathbb{P}(x)=1, \ \int_x x{\rm d}\mathbb{P}(x)=\mu, \ \int_x |x-\mu|{\rm d}\mathbb{P}(x)=d,\ \int_x \1_{\{x\geq\mu\}}{\rm d}\mathbb{P}(x)=\beta,   
\end{aligned}
\end{equation}
 which is a semi-infinite linear program with four equality constraints.


Consider the dual of \eqref{eq:primal2},
\begin{equation}\label{eq:dual2}
\begin{aligned}
&\inf_{\lambda_0,\lambda_1,\lambda_2, \lambda_3} &  &\lambda_0 + \lambda_1 \mu+\lambda_2 d+\lambda_3\beta\\
&\text{s.t.} &      & \1_{\{x\geq t\}}\leq \lambda_0  +\lambda_1 x+\lambda_2 |x-\mu|+\lambda_3 \1_{\{x\geq\mu\}}, \ \forall x\in[0,b].
\end{aligned}
\end{equation}
Define $F(x)=\lambda_0  +\lambda_1 x+\lambda_2 |x-\mu|+\lambda_3\1_{\{x\geq\mu\}}$. Then the inequality in \eqref{eq:dual2} can be written as $\1_{\{x\geq t\}}\leq F(x)$, $\forall x$, i.e.~$F(x)$ majorizes $\1_{\{x\geq t\}}$. Note that $F(x)$ has both a 'kink' and a jump discontinuity at $x=\mu$. The dual problem has four variables, and therefore the tightest majorant touches $\1\{x\geq t\}$ at four or fewer points. Since $F(x)$ is piecewise linear with a jump discontinuity there are four candidate scenarios, which are described in Figure~\ref{fig:majors2}. When $t\in[0,\mu)$, $F(x)$ touches $\1_{\{x\geq t\}}$ in $\{0,t,\mu,b\}$ (scenario 1a), or $F(x)=1$ and touches in $\{t,\mu,b\}$ (scenario 1b). When $t\in[\mu,b]$, $F(x)$ touches in $\{0,\mu,t\}$ (scenario 2a), or in $\{0,t,b\}$ (scenario 2b).

\begin{figure}[h!]
\begin{center}
\begin{minipage}[b]{0.45\linewidth}
\begin{tikzpicture}[scale=1.3]
\draw [<->,thick] (0,2.5) node (yaxis) [above] {}
        |- (5,0) node (xaxis) [right] {$x$};
        \draw[blue,thick] (0.5,2.2) -- (1.2,2.2) node[right] {$\1\{x\geq t\}$};
        \draw[densely dashed] (0.5,1.8) -- (1.2,1.8) node[right] {$F_{1a}(x)$};
        \draw[loosely dashed] (0.5,1.4) -- (1.2,1.4) node[right] {$F_{1b}(x)$};
        \coordinate (mu+) at (4,1);
        \coordinate (mu-) at (4,0);
        \coordinate (mu--) at (4,0.01);
        \coordinate (z) at (0, 0);
        \coordinate (z--) at (0.01,0.01);
        \coordinate (b) at (5,1);
        \coordinate (b+) at (5,2.5);
        \coordinate (t+) at (2.5,1);
        \coordinate (t-) at (2.5,0);
        \draw[blue, thick] (z) -- (t-);
        \draw[blue, dotted] (t-) -- (t+);
        \draw[blue, thick] (t+) -- (b);
        \draw[densely dashed] (z) -- (t+);
        \draw[densely dashed] (t+) -- (4,1.6);
        \draw[dotted] (4,1.6) -- (4,1.03);
        \draw[densely dashed] (4,1.03) -- (5,1.03);
        \draw[loosely dashed] (0,1.01) -- (5,1.01);
        \draw[dotted] (mu-) -- (mu-) node[below] {$\mu$};
        \draw[dotted] (t-) -- (t-) node[below] {$t$};
        \draw[dotted] (0,0) -- (0,0) node[below] {$0$};
        \draw[dotted] (4.85,0) --(4.85,0) node[below] {$b$};

\end{tikzpicture}
\end{minipage}
\hfill
\begin{minipage}[b]{0.45\linewidth}
\begin{tikzpicture}[scale=1.3]
\draw [<->,thick] (0,2.5) node (yaxis) [above] {}
        |- (5,0) node (xaxis) [right] {$x$};
        \draw[blue,thick] (0.5,2.2) -- (1.2,2.2) node[right] {$\1\{x\geq t\}$};
        \draw[densely dashed] (0.5,1.8) -- (1.2,1.8) node[right] {$F_{2a}(x)$};
        \draw[loosely dashed] (0.5,1.4) -- (1.2,1.4) node[right] {$F_{2b}(x)$};
        \coordinate (mu+) at (2,1);
        \coordinate (mu-) at (2,0);
        \coordinate (mu--) at (2,0.02);
        \coordinate (z) at (0, 0);
        \coordinate (z--) at (0,0.02);
        \coordinate (b) at (5,1);
        \coordinate (b+) at (5,2.5);
        \coordinate (t+) at (3.25,1);
        \coordinate (t-) at (3.25,0);
        \draw[blue, thick] (z) -- (t-);
        \draw[blue, dotted] (t-) -- (t+);
        \draw[blue, thick] (t+) -- (b);
        \draw[densely dashed] (z--) -- (mu--);
        \draw[densely dashed] (mu--) -- (t+);
        \draw[densely dashed] (t+) -- (b+);
        \draw[loosely dashed] (0,0.02) -- (2,0.02);
        \draw[dotted] (2,0.02) -- (2,1.02);
        \draw[loosely dashed] (2,1.02) -- (5,1.02);
        \draw[dotted] (mu-) -- (mu-) node[below] {$\mu$};
        \draw[dotted] (t-) -- (t-) node[below] {$t$};
        \draw[dotted] (0,0) -- (0,0) node[below] {$0$};
        \draw[dotted] (4.85,0) --(4.85,0) node[below] {$b$};

\end{tikzpicture}
\end{minipage}

\caption{Scenario 1 and the majorizing functions $F_{1a}(x)$ and $F_{1b}(x)$ under scenarios 1a and 1b, respectively. Scenario 2 and the majorizing functions $F_{2a}(x)$ and $F_{2b}(x)$ under scenarios 2a and 2b, respectively.}\label{fig:majors2}
\end{center}
\end{figure}
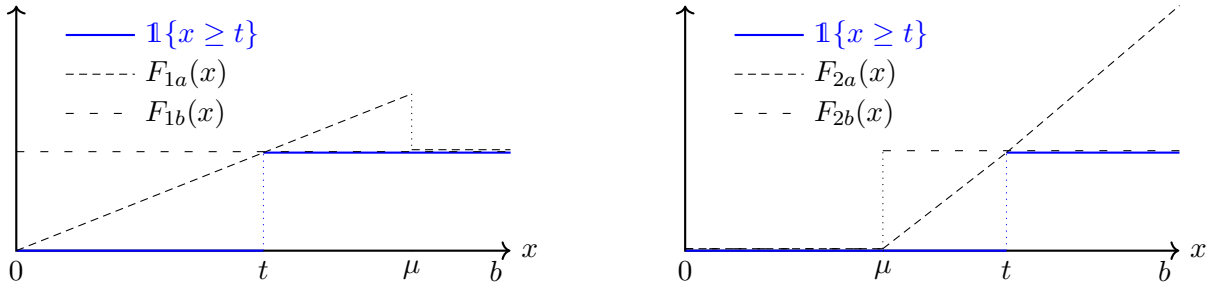

Scenario 1a implies $F(0)=0, \, F(t)=F(\mu)=F(b)=1,$ which gives
\begin{equation}
    \lambda_0=\frac{\mu}{2t}, \quad \lambda_1=\frac{1}{2t}, \quad \lambda_2=-\frac{1}{2t}, \quad \lambda_3=\frac{\mu-t}{t},
\end{equation}
and objective value
\begin{equation}
    \lambda_0+\lambda_1 \mu + \lambda_2 d + \lambda_3 \beta = \frac{(1-\beta)\mu + \beta t}{t}-\frac{d}{2t}.
\end{equation}
Solving the primal problem \eqref{eq:primal2} with probability masses on the points $\{0,t,\mu,b\}$ gives
\begin{equation}
    \int_{x} \1\{x\geq t\}\mathrm{d}\mathbb{P}{(x)}=p_{t}+p_\mu + p_b=\frac{(1-\beta)\mu + \beta t}{t}-\frac{d}{2t}.
\end{equation}
Since primal and dual feasible solutions have the same objective value we have strong duality and hence found the optimal solutions.

Scenario 1b implies that $F(0)=F(t)=F(\mu)=F(b)=1$, and hence $\lambda_0=1,\, \lambda_1=\lambda_2=\lambda_3=0$ with objective value 1. It is clear that the optimal primal objective value is also equal to 1.

Scenario 2a implies $F(0)=F(\mu)=0,\,F(t)=1$ which gives
\begin{equation}
    \lambda_0=-\frac{\mu}{2(t-\mu)}, \quad \lambda_1 =\lambda_2= \frac{1}{2(t-\mu)}, \quad \lambda_3 = 0,
\end{equation}
with objective value
\begin{equation}
    \lambda_0 + \lambda_1 \mu + \lambda_2 d + \lambda_3 \beta = \frac{d}{2(t-\mu)}.
\end{equation}
Solving the optimal probabilities for the primal problem \eqref{eq:primal2} indeed shows that $p_t=\frac{d}{2(t-\mu)}$.

Scenario 2b implies that $F(0)=0, \, F(\mu)=F(t)=F(b)=1$, which gives as the dual feasible solution
\begin{equation}
    \lambda_0=\lambda_1=\lambda_2=0, \quad \lambda_3 = 1,
\end{equation}
and objective value
\begin{equation}
    \lambda_0 + \lambda_1 \mu + \lambda_2 d + \lambda_3 \beta = \beta.
\end{equation}
Solving the optimal probabilities of \eqref{eq:primal2} confirms that $p_0 = (1-\beta)$.

The proof is then completed by looking which scenario prevails on a specific interval, and these intervals can be determined by simply equating the minimum objective values and thereafter solving for $t$ to find $\tau_1$ and $\tau_2$ for, respectively, scenario 1 and scenario 2. 
\end{proof}

\begin{proof}[Proof Theorem~\ref{theorem:betaLB}]
We will use similar arguments as in the proof of Theorem~\ref{theorem:betaUB}, but now we consider a dual problem where we are maximizing a minorizing function. Under $\mathcal{P}_{(\mu,b,d,\beta)}$ ambiguity of the random variable $X$ we now need to solve
\begin{equation}\label{eq:primal3}
\begin{aligned}
&\! \inf_{\mathbb{P}\in \mathcal{M}^+} &  &\int_x \1_{\{x > t\}}{\rm d} \mathbb{P}(x)\\
&\text{s.t.} &      & \int_x {\rm d}\mathbb{P}(x)=1, \ \int_x x{\rm d}\mathbb{P}(x)=\mu, \ \int_x |x-\mu|{\rm d}\mathbb{P}(x)=d,\ \int_x \1_{\{x\geq\mu\}}{\rm d}\mathbb{P}(x)=\beta,   
\end{aligned}
\end{equation}
 which is a semi-infinite linear program with four equality constraints. The dual problem is given by
\begin{equation}\label{eq:dualmaximize}
\begin{aligned}
&\sup_{\lambda_0,\lambda_1,\lambda_2,\lambda_3} &  &\lambda_0 + \lambda_1 \mu+\lambda_2 d+\lambda_3 \beta\\
&\text{s.t.} &      & \1_{\{x > t\}}\geq \lambda_0  +\lambda_1 x+\lambda_2 |x-\mu| + \lambda_3 \1_{\{x\geq \mu\}} \eqqcolon F(x), \ \forall x\in[0,b].
\end{aligned}
\end{equation} 
Note that $F(x)$ has both a 'kink' and a jump discontinuity at $x=\mu$. Additionally, we are to construct the tightest minorant this time. The dual problem has four variables, and hence the tightest minorant touches $\1_{\{x> t\}}$ at four or fewer points. Since $F(x)$ is piecewise linear with a jump discontinuity there are four candidate solutions, which are depicted in Figure~\ref{fig:minors1}. Note that we have a strict inequality inside of the indicator function. When $t\in[0,\mu)$, $F(x)$ touches $\1_{\{x> t\}}$ in $\{t,\mu,b\}$ (scenario 1a) or in $\{0,t,\mu,b\}$ (scenario 1b). When $t\in[\mu,b]$, $F(x)$ touches in $\{0,t,b\}$ (scenario 2a), or $F(x)=0$ and touches in $\{0,\mu,t\}$ (scenario 2b).

Scenario 1a implies $F(t)=0$, $F(\mu)=F(b)=1$, which gives the dual solution
\begin{equation}
    \lambda_0 = \frac{2t - \mu}{2(t-\mu)},\quad \lambda_1 = -\frac{1}{2(t-\mu)}, \quad \lambda_2=\frac{1}{2(t-\mu)}, \quad \lambda_3 = 0,
\end{equation}
and objective value
\begin{equation}
    \lambda_0+\lambda_1\mu+\lambda_2 d+\lambda_3\beta = 1-\frac{d}{2(\mu-t)}.
\end{equation}
Solving the primal problem \eqref{eq:primal3} with probability masses on the points $\{t,\mu,b\}$ gives
\begin{equation}
    \int_{x}\1\{x>t\}\,{\rm d}\mathbb{P}(x)=1-p_t=1-\frac{d}{2(\mu-t)}.
\end{equation}

\begin{figure}[h!]
\begin{center}
\begin{minipage}[b]{0.45\linewidth}
\begin{tikzpicture}[scale=1.3]
\draw [<->,thick] (0,2.5) node (yaxis) [above] {}
        |- (5,0) node (xaxis) [right] {$x$};
        \draw[blue,thick] (0.5,2.2) -- (1.2,2.2) node[right] {$\1\{x> t\}$};
        \draw[densely dashed] (0.5,1.8) -- (1.2,1.8) node[right] {$F_{1a}(x)$};
        \draw[loosely dashed] (0.5,1.4) -- (1.2,1.4) node[right] {$F_{1b}(x)$};
        \coordinate (mu) at (4,1);
        \coordinate (mu-) at (4,0);
        \coordinate (mu--) at (4,0.99);
        \coordinate (z) at (0, 0);
        \coordinate (b) at (5,1);
        \coordinate (t-) at (2.5,0);
        \coordinate (t+) at (2.5,1);
        \draw[blue, thick] (z) -- (t-);
        \draw[blue, dotted] (t-) -- (t+);
        \draw[blue, thick] (t+) -- (b);
        \draw[densely dashed] (1.3,-.7) -- (mu--);
        \draw[densely dashed] (mu--) -- (5,0.99);
        \draw[loosely dashed] (0,-.02) -- (4,-.02);
        \draw[dotted] (4,-.01) -- (mu--);
        \draw[loosely dashed] (4,0.98) -- (5,0.98);
        \draw[dotted] (mu-) -- (mu-)        node[below] {$\mu$};
        \draw[dotted] (t-) -- (t-) node[below] {$t$};
        \draw[dotted] (0,0) -- (0,0) node[below] {$0$};
        \draw[dotted] (4.85,0) --(4.85,0) node[below] {$b$};

\end{tikzpicture}
\end{minipage}
\hfill
\begin{minipage}[b]{0.45\linewidth}
\begin{tikzpicture}[scale=1.3]
\draw [<->,thick] (0,2.5) node (yaxis) [above] {}
        |- (5,0) node (xaxis) [right] {$x$};
        \draw[blue,thick] (0.5,2.2) -- (1.2,2.2) node[right] {$\1\{x> t\}$};
        \draw[densely dashed] (0.5,1.8) -- (1.2,1.8) node[right] {$F_{2a}(x)$};
        \draw[loosely dashed] (0.5,1.4) -- (1.2,1.4) node[right] {$F_{2b}(x)$};
        \coordinate (mu+) at (2,1);
        \coordinate (mu-) at (2,0);
        \coordinate (mu--) at (2,-0.01);
        \coordinate (z) at (0, 0);
        \coordinate (z--) at (0,-0.01);
        \coordinate (b) at (5,1);
        \coordinate (b+) at (5,2.5);
        \coordinate (t+) at (3.25,1);
        \coordinate (t-) at (3.25,0);
        \draw[blue, thick] (z) -- (t-);
        \draw[blue, dotted] (t-) -- (t+);
        \draw[blue, thick] (t+) -- (b);
        \draw[densely dashed] (z--) -- (mu--);
        \draw[dotted] (mu--) -- (2,-0.71);
        \draw[densely dashed] (2,-0.71) -- (t-);
        \draw[densely dashed] (t-) -- (b);
        \draw[loosely dashed] (0,-0.02) -- (5,-0.02);
        \draw[dotted] (mu-) -- (mu-) node[below] {$\mu$};
        \draw[dotted] (t-) -- (t-) node[below] {$t$};
        \draw[dotted] (0,0) -- (0,0) node[below] {$0$};
        \draw[dotted] (4.85,0) --(4.85,0) node[below] {$b$};

\end{tikzpicture}
\end{minipage}

\caption{Scenario 1 and the minorizing functions $F_{1a}(x)$ and $F_{1b}(x)$ under scenarios 1a and 1b, respectively. Scenario 2 and the minorizing functions $F_{2a}(x)$ and $F_{2b}(x)$ under scenarios 2a and 2b, respectively.}\label{fig:minors1}
\end{center}
\end{figure}
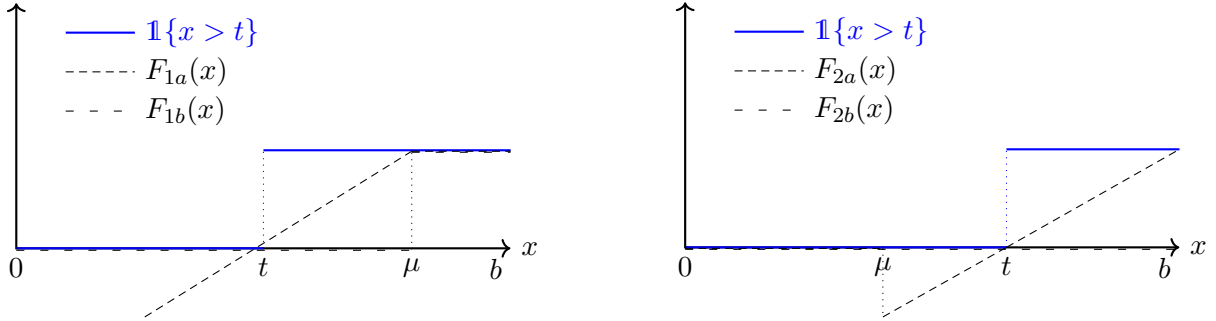

Scenario 1b implies that $F(0)=F(t)=0$, $F(\mu)=F(b)=1$, and hence $\lambda_0=\lambda_1=\lambda_2=0$, $\lambda_3=1$ with objective value $\beta$. Now solving the primal problem with probability masses on $\{0,t,\mu,b\}$ gives us
\begin{equation}
        \int_{x}\1\{x>t\}\,{\rm d}\mathbb{P}(x)=p_\mu + p_b=\beta.
\end{equation}
Scenario 2a implies that $F(0)=F(t)=0$ and $F(b)=1$, which results in 
\begin{equation}
    \lambda_0=\frac{\mu}{2(t-b)},\quad \lambda_1=\lambda_2=\frac{1}{2(b-t)},\quad \lambda_3=-\frac{(\mu-t)}{(b-t)},
\end{equation}
with objective value
\begin{equation}
    \lambda_0+\lambda_1\mu + \lambda_2 d+\lambda_3\beta = \frac{\beta(\mu-t)}{(b-t)}+\frac{d}{2(b-t)}.
\end{equation}
Indeed, solving the primal problem with probability masses on $\{0,t,b\}$ gives $p_b = \frac{\beta(\mu-t)}{(b-t)}+\frac{d}{2(b-t)}$.

Scenario 2b implies that $F(0)=F(\mu)=F(t)=F(b)=0$, which gives the dual feasible solution
\begin{equation}
    \lambda_0=\lambda_1=\lambda_2=\lambda_3=0, 
\end{equation}
with objective value 0. All probability mass is placed on points that are less than or equal to $t$. Hence, the optimal primal objective value is also equal to 0.

Finally, the proof is completed by inspecting which scenario prevails on a specific interval. These intervals are determined by equating the maximum objective values and solving these equations with respect to $t$ to find $\tau_1$ and $\tau_2$ for scenario 1 and 2, respectively.

\end{proof}

\section{Comparison with tight bounds for ${(\mu,b,\sigma)}$ ambiguity}\label{app: deschepper}

Next to the comparison with Cantelli's inequality in Section~\ref{sec:compare}, we also look at the tight bounds for the $(\mu,b,\sigma)$-ambiguity set. \cite{schepper1995general} provide expressions for these bounds. The upper bound is given by
\begin{equation}\label{eq: supschepper}
\sup\limits_{\bP\in\cP_{(\mu,b,\sigma)}}\bP(X\geq t)=\begin{cases}
     1, \quad &  t\in\Big[0,\mu-\frac{\sigma^2}{b-\mu}\Big], \\
     1-\frac{\sigma^2+(b-\mu)(t-\mu)}{bt}, \quad &  t\in\Big[\mu-\frac{\sigma^2}{b-\mu},\mu+\frac{\sigma^2}{\mu}\Big], \\
     \frac{\sigma^2}{\sigma^2+(t-\mu)^2}, &  t\in\Big[\mu+\frac{\sigma^2}{\mu},b\Big].
\end{cases}
\end{equation}
The lower bound equals
\begin{equation}\label{eq: infschepper}
\inf\limits_{\bP\in\cP_{(\mu,b,\sigma)}}\bP(X\geq t)=\begin{cases}
     \frac{(\mu-t)^2}{(\mu-t)^2+\sigma^2}, \quad &  t\in\Big[0,\mu-\frac{\sigma^2}{b-\mu}\Big], \\
     \frac{\sigma^2+\mu(\mu-t)}{b(b-t)}, \quad &  t\in\Big[\mu-\frac{\sigma^2}{b-\mu},\mu+\frac{\sigma^2}{\mu}\Big], \\
     0, &  t\in\Big[\mu+\frac{\sigma^2}{\mu},b\Big].
\end{cases}
\end{equation}

Comparing these bounds with their MAD equivalents is again not straightforward, and we will be using a similar numerical example to compare our results with those of \cite{schepper1995general}. We use the following parameter setting: $a=0$, $\mu=1$, $b=2$, $d=1/4$. Furthermore, we consider three values for $\sigma$: $\sigma=d=1/4$, $\sigma=1/3$, and $\sigma=\sqrt{db/2}=1/2$.

\begin{figure}[h!]
\begin{center}
\begin{tikzpicture}
\begin{axis}[
	xlabel={\(t\)},
	xmin=0.5, xmax=2,
	ymin=0, ymax=1.1,
    ylabel={\(\sup\bP\left(X \geq t\right)\)},
    ylabel style={yshift=0.2cm},
    xticklabel style={
    	/pgf/number format/precision=4,
    	/pgf/number format/fixed,
	},
	yticklabel style={/pgf/number format/precision=4, /pgf/number format/fixed},
    label style={font=\small},
    ticklabel style={font=\footnotesize},
    legend style={draw=none, font=\footnotesize},
    legend cell align={left},
    legend pos = outer north east
]
\addplot[blue!70!black, line width=0.75pt] table {Figure_Data/Schepper/meanMAD.dat};
\addlegendentry{Theorem \ref{theorem:main}};

\addplot[green!70!black, line width=0.75pt] table {Figure_Data/Schepper/meanVariance1.dat};
\addlegendentry{$(\mu,b,\sigma)$ upper bound with \(\sigma = \frac{1}{4}\)};

\addplot[orange!70!black, line width=0.75pt] table {Figure_Data/Schepper/meanVariance2.dat};
\addlegendentry{$(\mu,b,\sigma)$ upper bound with \(\sigma = \frac{1}{3}\)};

\addplot[red!70!black, line width=0.75pt] table {Figure_Data/Schepper/meanVariance3.dat};
\addlegendentry{$(\mu,b,\sigma)$ upper bound with \(\sigma = \frac{1}{2}\)};

\end{axis}
\end{tikzpicture}
\end{center}
\caption{A comparison of the mean-MAD upper bound with the upper bound of \cite{schepper1995general} for three different values of \(\sigma\) with the parameter chosen as follows: $a=0$, \(\mu = 1\), \(b = 2\) and \(d = \frac{1}{4}\). \label{fig: ub comparison schepper}}
\end{figure}
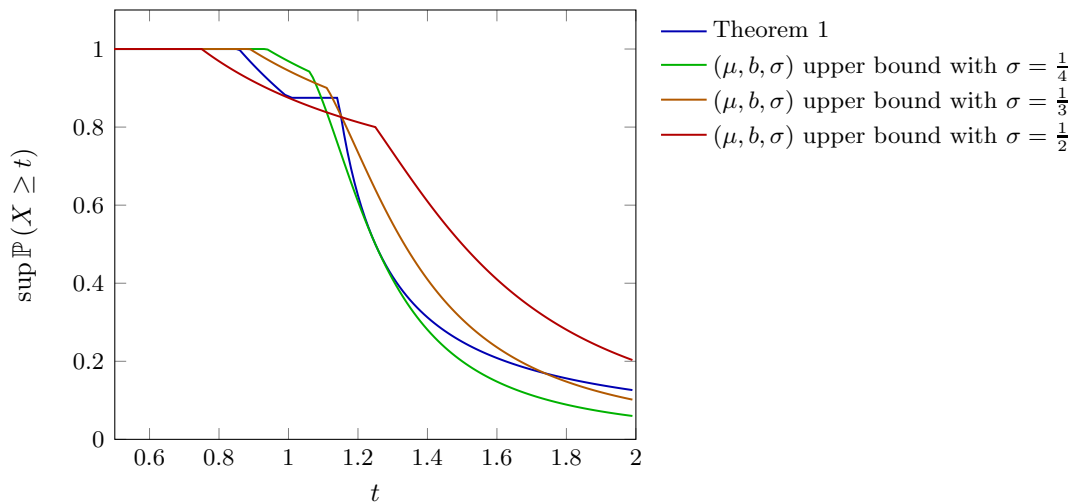

Figure~\ref{fig: ub comparison schepper} shows the upper bounds for mean-MAD ambiguity and the ${(\mu,b,\sigma)}$-ambiguity sets. The difference with Cantelli's inequality essentially lies in the behavior of the bound near the mean $\mu$. Expression~\eqref{eq: supschepper} sharpens the bounds of the tail probability for $t\in[\mu-\sigma^2/(b-\mu),\mu+\sigma^2/\mu]$ by using information about the upper bound of the support. Note that Cantelli's inequality and the tight upper bound are equivalent in the tail. Another interesting observation is the equivalence of the mean-variance and mean-MAD bound for $t=\mu$ and $\sigma=\sqrt{db/2}$. Figure~\ref{fig: lb comparison schepper} paints a similar picture for the lower bounds.

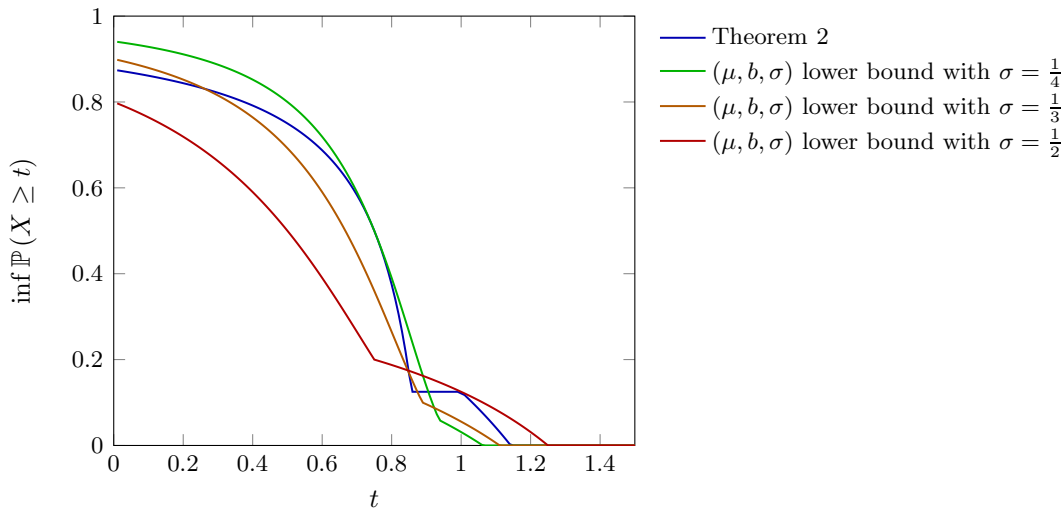
\begin{figure}[h!]
\begin{center}
\begin{tikzpicture}
\begin{axis}[
	xlabel={\(t\)},
	xmin=0, xmax=1.5,
	ymin=0, ymax=1,
    ylabel={\(\inf\bP\left(X \geq t\right)\)},
    ylabel style={yshift=0.2cm},
    xticklabel style={
    	/pgf/number format/precision=4,
    	/pgf/number format/fixed,
	},
	yticklabel style={/pgf/number format/precision=4, /pgf/number format/fixed},
    label style={font=\small},
    ticklabel style={font=\footnotesize},
    legend style={draw=none, font=\footnotesize},
    legend cell align={left},
    legend pos = outer north east
]

\addplot[blue!70!black, line width=0.75pt] table {Figure_Data/Schepper/meanMADL.dat};
\addlegendentry{Theorem \ref{theorem:inf}};

\addplot[green!70!black, line width=0.75pt] table {Figure_Data/Schepper/meanVarianceL1.dat};
\addlegendentry{$(\mu,b,\sigma)$ lower bound with \(\sigma = \frac{1}{4}\)};

\addplot[orange!70!black, line width=0.75pt] table {Figure_Data/Schepper/meanVarianceL2.dat};
\addlegendentry{$(\mu,b,\sigma)$ lower bound with \(\sigma = \frac{1}{3}\)};

\addplot[red!70!black, line width=0.75pt] table {Figure_Data/Schepper/meanVarianceL3.dat};
\addlegendentry{$(\mu,b,\sigma)$ lower bound with \(\sigma = \frac{1}{2}\)};

\end{axis}
\end{tikzpicture}
\end{center}
\caption{A comparison of the mean-MAD lower bound with the lower bound of \cite{schepper1995general} for three different values of \(\sigma\) with the parameters chosen as follows: $a=0$, \(\mu = 1\), \(b = 2\) and \(d = \frac{1}{4}\). \label{fig: lb comparison schepper}}
\end{figure}

\section{Newsvendor model with skewness information (EC)}\label{newsskew}
We now consider the newsvendor problem under restricted ambiguity where we consider demand distributions with given mean $\mu$, MAD $d$, skewness information $\bP(D\geq\mu)$, and bounded support $[0,b]$. The following problem is the mean-MAD counterpart of the mean-variance-semivariance model discussed in the work of \cite{natarajan2007mean}:
\begin{equation}
    \max_{q}\inf\limits_{\bP\in\cP_{(\mu,b,d,\beta)}}\mathbb{E}_{\bP}[\pi(q,D)],
\end{equation}
where $\beta$ adopts the role of the semivariance to model skewness information. Instead of solving this problem directly, we apply the robust Chebyshev bounds with skewness information to the first-order condition of the newsvendor problem. Thus, we will use the results from Theorem~\ref{theorem:betaUB} and Theorem~\ref{theorem:betaLB} to bound the tail distribution of the demand $D$. Tight lower and upper bounds for the optimal order quantity $q^*$ follow from $\inf_{\bP\in\cP_{(\mu,b,d,\beta)}}\bP(D>q)$ and $\sup_{\bP\in\cP_{(\mu,b,d,\beta)}}\bP(D>q)$, respectively. The following result provides an interval that contains the optimal order quantity $q^*$.

\begin{theorem}[Order quantity bounds under mean-MAD-$\beta$ ambiguity]\label{newsbetathm}  Suppose the newsvendor knows the mean $\mu$, the mean absolute deviation $d$, the probability $\bP(D\geq \mu)=\beta$ and the upper 
bound $b$ of the demand distribution $\mathbb{P}(D\leq q)$. The optimal order quantity $q^*$ that solves $\max_q\mathbb{E}_\mathbb{P}[\pi(q, D)]$ is then contained in the interval $[q^L,q^U]$ with
\begin{align}\label{optimalqcasesspecial}
    [q^L,q^U]&=
    \begin{cases}
    \ \Big[0,\, \frac{(1-\beta)\mu-d/2}{(1-\eta-\beta)}\Big],  \quad & {\rm if }\ \eta < \frac{d}{2\mu}, \\
    \ \Big[\mu-\frac{d}{2\eta},\, \mu \Big],  \quad & {\rm if }\ \frac{d}{2\mu}\leq \eta< 1-\beta, \\
    \ \mu, \quad & {\rm if }\ \eta=1-\beta,\\
    \ \Big[\mu,\, \mu+\frac{d}{2(1-\eta)}\Big],  \quad & {\rm if }\ (1-\beta)< \eta <  1-\frac{d}{2(b-\mu)}, \\
    \ \Big[\frac{b (1-\eta)-\beta\mu-d/2}{(1-\eta-\beta)},\, b\Big], & {\rm if }\ \eta\geq 1-\frac{d}{2(b-\mu)}.
    \end{cases}
\end{align}
\end{theorem}

This result provides a robust policy that models the uncertainty captured by the ambiguity set $\cP_{(\mu,b,d,\beta)}$. Obviously, ordering the mean is optimal if $\eta=1-\beta$. The lower bound $q^L$ relates to the worst-case demand distribution. Similar to the mean-MAD-range case, $q^L$ is larger than $\mu$ when the profit margin $\eta$ exceeds $1-d/(2(b-\mu))$. Hence, an interesting observation is that the skewness information does not influence the point at which we order more than the mean. This can be contrasted with the results of \cite{natarajan2007mean}. These authors show for $\cP_{(\mu,\sigma,s)}$ ambiguity that the order quantity is greater than $\mu$ if $\eta>\frac12(1+s)$, where $s$ is the normalized semivariance. 

Table~\ref{res:newstable2} shows that the bounded support $[0,b]$ again influences the intervals for low and high profit margins. The intervals in this section, however, are sharper than the ones found in Table~\ref{res:newstable}, which is a consequence of the additional information regarding the skewness of the demand distribution.


\begin{table}[h!]
\begin{center}
\begin{tabular}{rrrrr}\hline
{ $\eta$}     & { $b=10$}           & {$b=15$}           & { $b=20$}            & { $b=\infty$}           \\ \hline
0.01 & $[0.00, 3.57]$  & $[0.00, 3.57]$  & $[0.00, 3.57]$  & $[0, 3.57]$  \\
0.1  & $[0.00, 4.38]$  & $[0.00, 4.38]$  & $[0.00, 4.38]$  & $[0.00, 4.38]$  \\
0.2  & $[ 1.25,5.00]$  & $[ 1.25,5.00]$  & $[ 1.25,5.00]$  & $[ 1.25,5.00]$  \\
0.4  & $[ 3.13,5.00]$  & $[ 3.13,5.00]$  & $[ 3.13,5.00]$  & $[ 3.13,5.00]$  \\
0.5  & 5.00            & 5.00             & 5.00              & 5.00              \\
0.7  & $[5.00, 7.50]$  & $[5.00, 7.50]$  & $[5.00, 7.50]$  & $[5.00, 7.50]$  \\
0.9  & $[ 5.63,10.00]$ & $[5.00, 12.50]$ & $[5.00, 12.50]$ & $[5.00, 12.50]$ \\
0.95 & $[ 6.11,10.00]$ & $[ 5.56,15.00]$ & $[ 5.00,20.00]$ & $[5.00, 20.00]$ \\
0.99 & $[ 6.43,10.00]$ & $[ 6.33,15.00]$ & $[ 6.22,20.00]$ & $[5.00, 80.00]$   \\ \hline   
\end{tabular}
\end{center}
\caption{The intervals $[q^L,q^U]$ for mean-MAD-$\beta$ ambiguity with $\mu=5$, $d=1.5$, $\beta=0.5$ and various profit margins $\eta$.}\label{res:newstable2}
\end{table}


\section{Upper bound for retention function} \label{EC:retention}

\begin{proposition}\label{prop: stoplossceding}
The worst-case expected claim payment of the direct insurer as a function of the retention limit $z$ is given by
\begin{equation}\label{eq: cedentclaim}
    \sup\limits_{\mathbb{P}\in\mathcal{P}_{(\mu,b,d)}}\mathbb{E}_{\mathbb{P}}[\psi(z,S)]=\begin{cases}
    \ z, \quad & {\rm if }\ z\in[0,\tau_1],\\
    \ \mu - \frac{d(b-z)}{2(b-\mu)}, \quad & {\rm if }\ z\in[\tau_1,\mu],\\
    \ z(1-\frac{d}{2\mu} ), \quad & {\rm if }\ z\in[\mu,\tau_2],\\
    \ \mu, &{\rm if }\ z\in[\tau_2,b],
    \end{cases}
\end{equation}
where the values of $\tau_1$ and $\tau_2$ are given by
\begin{equation*}
    \tau_1=\mu-\frac{d(b-\mu)}{2(b-\mu)-d},\quad \tau_2=\mu+\frac{d\mu}{2\mu-d}.
\end{equation*}

\end{proposition}
\begin{proof}
First, note that
\begin{equation*}
    \psi(z,S)=\min\{S,z\}=S-\max\{S-z,0\},
\end{equation*}
and hence our problem boils down to solving
\begin{equation}
    \sup\limits_{\mathbb{P}\in\mathcal{P}_{(\mu,b,d)}}\mathbb{E}_{\mathbb{P}}[\psi(z,S)] = \mu - \inf\limits_{\mathbb{P}\in\mathcal{P}_{(\mu,b,d)}}\mathbb{E}_{\mathbb{P}}[\max\{S-z,0\}].
\end{equation}
The second term is convex in the uncertain parameter and therefore we can apply the lower bounds discussed by \cite{postek2018robust}, that is, we solve the optimization problem
\begin{equation}\label{eq: posteksolve}
    \inf\limits_{\mathbb{P}\in\mathcal{P}_{(\mu,b,d)}}\mathbb{E}_{\mathbb{P}}[\max\{S-z,0\}]=\min\limits_{\frac{d}{2(b-\mu)}\leq\theta\leq 1-\frac{d}{2\mu}}\Big\{\theta \max\{\mu+\frac{d}{2\theta}-z,0\} +(1-\theta) \max\{\mu-\frac{d}{2(1-\theta)}-z,0\} \Big\},
\end{equation}
which is convex and piecewise linear in the optimization variable $\theta$. Hence, one can find the optimal solution value that depends on the retention limit $z$. Solving problem \eqref{eq: posteksolve} and subtracting the optimal value from $\mu$ results in the four cases mentioned in \eqref{eq: cedentclaim}.
\end{proof}

\section{Proof of Theorem~\ref{mainmadthm} (EC)}\label{EC:proofprice}

We shall now solve \eqref{mpp}. 

Let us make some observations about the functions $G(t):=\inf_{\mathbb{P} \in \mathcal{P}_{(\mu,d)}}\prob{B\geq t}$  and 
$F(r):=r G(r)$. The function $G(t)$ starts in $G(0)=(2\mu-d)/(2\mu)$, decays until $\tau_1$, remains flat for $t\in[\tau_1,\mu]$, decays until reaching zero at $t=\tau_2$, and then remains zero. The function $F(r)$ starts in $0$, is concave until $\tau_1$,  increases linearly for $t\in[\tau_1,\mu]$, and then remains concave until reaching zero. This implies that the maximum of $F(r)$ is the maximum of the first concave part, the point $F(\mu)$, or the maximum of the second concave part. 

The first concave part is given by the function 
$$
F_1(r):=r-\frac{dr}{2(\mu-r)}  \ {\rm for} \  r\in[0,\tau_1],
$$
for which $F_1'(r)=0$ gives
$$
r_1=\mu-\sqrt{\frac{d\mu}{2}}
$$
and $F_1(r_1)=\mu+{d}/2-\sqrt{2d\mu}$. This value should be compared with $F(\mu)=\frac{d\mu}{2(b-\mu)}$, and in fact solving for $d$ for which $F_1(r_1)=F(\mu)$ gives 
$$
d_1=\frac{2 b \mu  (b-\mu )-4 \sqrt{\mu ^3 (b-\mu )^3}}{(b-2 \mu )^2}.
$$
The second  concave part is given by the function 
$$
F_2(r):=r \left(\frac{\mu -r}{b-r}+\frac{d r}{2 \mu  (b-r)}\right)
  \ {\rm for} \  r\in[\mu,\tau_2],
$$
for which $F_2'(r)=0$ gives
$$
r_2
=b -\frac{\sqrt{b  (2 \mu -d) (2 \mu  (b -\mu )-b  d)}}{2 \mu -d}
$$
and $$F_2(r_2)=\frac{2 b\mu -b  d-\mu ^2 -\sqrt{b  (2 \mu -d) (2 \mu  (b -\mu )-b  d)}}{\mu }.$$
Solving for $d$ for which $F_2(r_2)=F(\mu)$ gives 
$$
d_2=
\frac{2 \left(b \mu -\mu ^2\right)}{2 b-\mu }.
$$
Upon reflection, $d_2$ must be the point where the right-derivative of $F_2(r)$  turns positive, which indeed is the case. 
It can be shown that $d_2\leq d_1$ for $\mu\in[0,b/5]$ and $d_2\geq d_1$ for $\mu\in[b/5,b]$; see Figure \ref{fig:ddd}. 

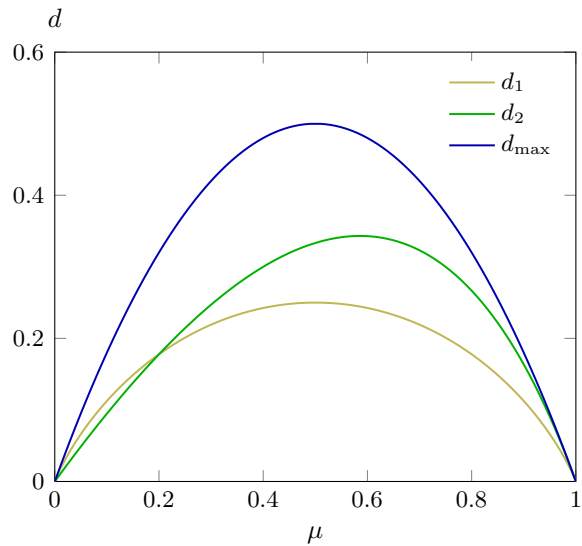
\begin{figure}[h!]
\begin{center}
\begin{tikzpicture}
\begin{axis}[
	xlabel={$\mu$},
	xmin=0, xmax=1,
	ymin=0, ymax=0.6,
    ylabel={$d$},
    label style={font=\small},
    ticklabel style={font=\footnotesize},
    xticklabel style={
    	/pgf/number format/precision=4,
    	/pgf/number format/fixed,
	},
    scaled x ticks=false,
    y label style={at={(current axis.above origin)},rotate=-90,anchor=south,yshift=0.2cm},
    legend style={draw=none, font=\footnotesize},
    legend cell align={left}
]
\addplot[yellow!70!black, line width=0.75pt] table {Figure_Data/pricing/d1.dat};
\addlegendentry{$d_1$};
\addplot[green!70!black, line width=0.75pt] table {Figure_Data/pricing/d2.dat};
\addlegendentry{$d_2$};
\addplot[blue!70!black, line width=0.75pt] table {Figure_Data/pricing/dmax.dat};
\addlegendentry{$d_{\text{max}}$};
\end{axis}
\end{tikzpicture}
\end{center}
\caption{The functions $d_1=d_1(\mu,b), \, d_2=d_2(\mu,b),$ and $d_{\text{max}}=d_{\text{max}}(\mu,b)=\frac{2(b-\mu)\mu}{b}$ for $b=1$.   
}\label{fig:ddd}
\end{figure}

First consider the case $d_1\leq d_2\leq d_{\rm max}$, hence assuming $b\leq 5 \mu$. For $d\in[0,d_1]$, the maximum of $F(r)$ is located at $r_1$.   For $d\in[d_1,d_2]$, the maximum of $F(r)$ is at $r=\mu$, because $F(\mu)\geq F(r_1)$ and the function $F(r)$ will not increase for $r\geq \mu$. For $d\in[d_2,d_{\rm max}]$, the maximum of $F(r)$ is located at $r_2$, because $F(\mu)\leq F(r_1)$ and the function $F(r)$ still increases after $r=\mu$ until $r=r_2$. Figure \ref{fig:plotsF} illustrates these three scenarios by plotting $F(r)$ for various values of $d$.   

Then consider the case  $d_2\leq d_1\leq d_{\rm max}$, hence assuming $b\geq 5 \mu$. 
Now $r=\mu$ is no longer a candidate optimizer, because $F(r)$ viewed as a function of $d$, will become increasing at $r=\mu$ before $F(\mu)$ beats $F(r_1)$. Therefore, the maximum will be in either $r_1$ or $r_2$. It will be $r_1$ when $F(r_1)> F(r_2)$ and vice versa. Solving for $d$ for which $F(r_1)= F(r_2)$ is computationally tractable, but does not lead to a closed-form solution. See Figure \ref{fig:plotsFcase2} for an example of this case.

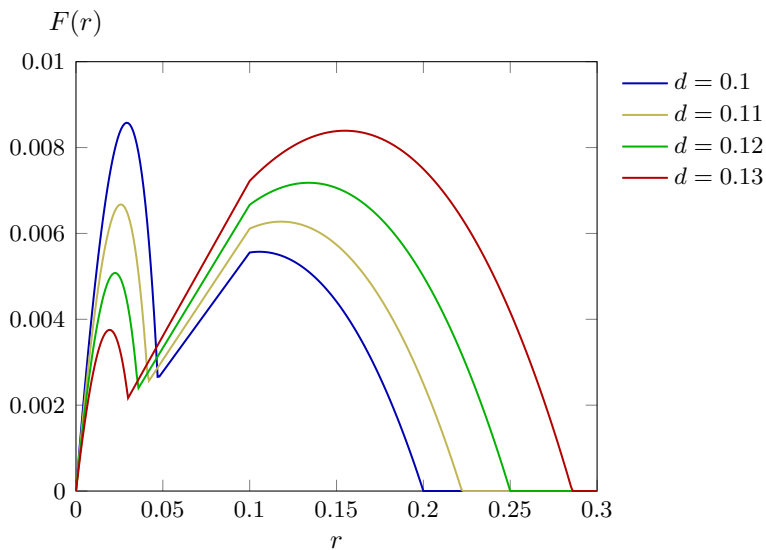
\begin{figure}[h!]
\begin{center}
\begin{tikzpicture}
\begin{axis}[
	xlabel={$r$},
	xmin=0, xmax=0.3,
	ymin=0, ymax=0.01,
    ylabel={$F(r)$},
    xticklabel style={
    	/pgf/number format/precision=4,
    	/pgf/number format/fixed
	},
	yticklabel style={/pgf/number format/precision=4, /pgf/number format/fixed},
	scaled y ticks=false,
    label style={font=\small},
    ticklabel style={font=\footnotesize},
    y label style={at={(current axis.above origin)},rotate=-90,anchor=south,yshift=0.2cm},
    legend style={draw=none, font=\footnotesize},
    legend cell align={left},
    legend pos = outer north east
]
\addplot[blue!70!black, line width=0.75pt] table {Figure_Data/pricing/Fr1.dat};
\addlegendentry{$d=0.1$};
\addplot[yellow!70!black, line width=0.75pt] table {Figure_Data/pricing/Fr2.dat};
\addlegendentry{$d=0.11$};
\addplot[green!70!black, line width=0.75pt] table {Figure_Data/pricing/Fr3.dat};
\addlegendentry{$d=0.12$};
\addplot[red!70!black, line width=0.75pt] table {Figure_Data/pricing/Fr4.dat};
\addlegendentry{$d=0.13$};
\end{axis}
\end{tikzpicture}
\end{center}
\caption{The functions $F(r)$ for $\mu=0.1$ and $b=1$ for which $d_1=0.1125, \, d_2=0.0947$, and $d_{\text{max}}=1/2$.   
}\label{fig:plotsFcase2}
\end{figure}

\begin{figure}[h!]
\begin{center}
\begin{tikzpicture}
\begin{axis}[
	xlabel={$d$},
	xmin=0, xmax=0.5,
	ymin=0, ymax=0.5,
    ylabel={$F(r^*)$},
    xticklabel style={
    	/pgf/number format/precision=4,
    	/pgf/number format/fixed,
	},
    scaled x ticks=false,
    label style={font=\small},
    ticklabel style={font=\footnotesize},
    y label style={at={(current axis.above origin)},rotate=-90,anchor=south,yshift=0.2cm},
]
\addplot[blue!70!black, line width=0.75pt] table {Figure_Data/pricing/optrevenue.dat};
\end{axis}
\end{tikzpicture}
\end{center}
\caption{The optimal revenue $F(r^*)$ for $\mu=0.5$ and $b=1$.   
}\label{fig:optrev1}
\end{figure}
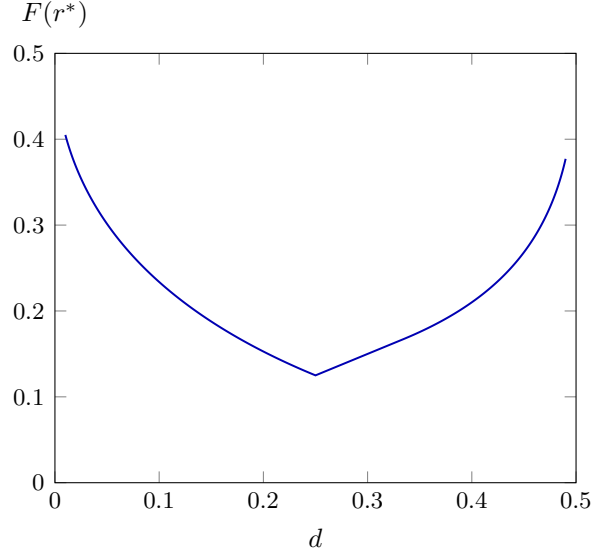

\section{Proofs of distribution-free stop-loss bounds}\label{app: stoplossthms}

\begin{proof}[Proof of Theorem~\ref{th:stoplossrein}]
We will show via primal-dual reasoning that the stated stop-loss formulas are tight upper bounds. We now consider the measurable function $\phi(z,s)$. Under $\mathcal{P}_{(\mu,b,d)}$ ambiguity of the random variable $S$ we now need to solve
\begin{equation}\label{eq:primalstop}
\begin{aligned}
&\! \sup_{\mathbb{P}\in \mathcal{M}^+} &  &\int_s \phi(z,s){\rm d} \mathbb{P}(s)\\
&\text{s.t.} &      & \int_s {\rm d}\mathbb{P}(s)=1, \ \int_s s{\rm d}\mathbb{P}(s)=\mu, \ \int_s |s-\mu|{\rm d}\mathbb{P}(s)=d, 
\end{aligned}
\end{equation}
 which is a semi-infinite linear program with three equality constraints.


Consider the dual of \eqref{eq:primalstop},
\begin{equation}\label{eq:dualstop}
\begin{aligned}
&\inf_{\lambda_0,\lambda_1,\lambda_2} &  &\lambda_0 + \lambda_1 \mu+\lambda_2 d\\
&\text{s.t.} &      & \phi(z,s)\leq \lambda_0  +\lambda_1 s+\lambda_2 |s-\mu|, \ \forall s\in[0,b].
\end{aligned}
\end{equation}
Define $F(s)\coloneqq\lambda_0  +\lambda_1 s+\lambda_2 |s-\mu|$. Then the inequality in \eqref{eq:dualstop} can be written as $\phi(z,s)\leq F(s)$, $\forall s$, i.e.~$F(s)$ majorizes the 'staircase' function $\phi(z,s)$. Note that $F(s)$ has a 'kink' at $x=\mu$. The dual problem has three variables, and therefore there exists a majorant that touches $\phi(z,s)$ at three or fewer points. Since $F(s)$ is piecewise linear with a 'wedge' shape there are six candidate scenarios, which are displayed in Figure~\ref{fig:majorsstop}. When $m+z\leq\mu$, $F(s)=1$ and touches $\phi(z,s)$ in $[m+z,b]$ (scenario 1a), or $F(s)$ touches $\phi(z,s)$ in $\{0,m+z,b\}$ (scenario 1b). When $z\leq\mu\leq m+z$, $F(s)$ touches in $\{0\}\cup[\mu,m+z]$ (scenario 2a) or in $\{0\}\cup[m+z,b]$ (scenario 2b). Finally, if $\mu\leq z\leq m+z$, $F(s)$ coincides with $\phi(z,s)$ in $[0,\mu]\cup\{m+z\}$ (scenario 3a) or in $\{0\}\cup[m+z,b]$ (scenario 3b).

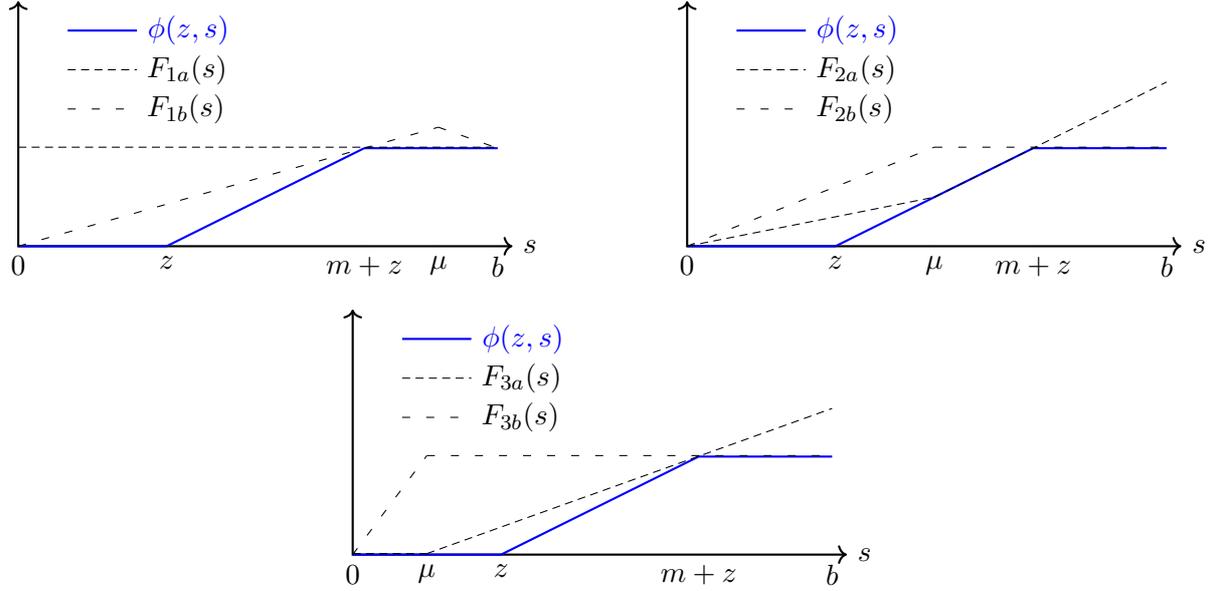
\begin{figure}[h!]
\begin{center}
\begin{minipage}[b]{0.45\linewidth}
\begin{tikzpicture}[scale=1.3]
\draw [<->,thick] (0,2.5) node (yaxis) [above] {}
        |- (5,0) node (xaxis) [right] {$s$};
        \draw[blue,thick] (0.5,2.2) -- (1.2,2.2) node[right] {$\phi(z,s)$};
        \draw[densely dashed] (0.5,1.8) -- (1.2,1.8) node[right] {$F_{1a}(s)$};
        \draw[loosely dashed] (0.5,1.4) -- (1.2,1.4) node[right] {$F_{1b}(s)$};
        \coordinate (z) at (1.5,0);
        \coordinate (mz) at (3.5,0);
        \coordinate (mu) at (4.25,0);
        \draw[blue, thick] (0,0) -- (z);
        \draw[blue, thick] (z) -- (3.5,1);
        \draw[blue, thick] (3.5,1) -- (4.85,1);
        \draw[densely dashed] (0,1.01) -- (4.85,1.01);
        \draw[loosely dashed] (0,0) -- (3.5,1);
        \draw[loosely dashed] (3.5,1) -- (4.25,1.214);
        \draw[loosely dashed] (4.25,1.214) -- (4.85,1);
        \draw[dotted] (z) -- (z) node[below] {$z$};
        \draw[dotted] (mz) -- (mz) node[below] {$m+z$};
        \draw[dotted] (0,0) -- (0,0) node[below] {$0$};
        \draw[dotted] (4.85,0) --(4.85,0) node[below] {$b$};
        \draw[dotted] (mu) --(mu) node[below] {$\mu$};

\end{tikzpicture}
\end{minipage}
\hfill
\begin{minipage}[b]{0.45\linewidth}
\begin{tikzpicture}[scale=1.3]
\draw [<->,thick] (0,2.5) node (yaxis) [above] {}
        |- (5,0) node (xaxis) [right] {$s$};
        \draw[blue,thick] (0.5,2.2) -- (1.2,2.2) node[right] {$\phi(z,s)$};
        \draw[densely dashed] (0.5,1.8) -- (1.2,1.8) node[right] {$F_{2a}(s)$};
        \draw[loosely dashed] (0.5,1.4) -- (1.2,1.4) node[right] {$F_{2b}(s)$};
        \coordinate (z) at (1.5,0);
        \coordinate (mz) at (3.5,0);
        \coordinate (mu) at (2.5,0);
        \draw[blue, thick] (0,0) -- (z);
        \draw[blue, thick] (z) -- (3.5,1);
        \draw[blue, thick] (3.5,1) -- (4.85,1);
        
        \draw[densely dashed] (0,0)--(2.5,0.5);
        \draw[densely dashed] (2.5,0.5)--(4.85,1.675);
        
        \draw[loosely dashed] (0,0)--(2.5,1.01);
        \draw[loosely dashed] (2.5,1.01)--(4.85,1.01);

        \draw[dotted] (z) -- (z) node[below] {$z$};
        \draw[dotted] (mz) -- (mz) node[below] {$m+z$};
        \draw[dotted] (0,0) -- (0,0) node[below] {$0$};
        \draw[dotted] (4.85,0) --(4.85,0) node[below] {$b$};
        \draw[dotted] (mu) --(mu) node[below] {$\mu$};

\end{tikzpicture}
\end{minipage}

\begin{minipage}[b]{0.45\linewidth}
\begin{tikzpicture}[scale=1.3]
\draw [<->,thick] (0,2.5) node (yaxis) [above] {}
        |- (5,0) node (xaxis) [right] {$s$};
        \draw[blue,thick] (0.5,2.2) -- (1.2,2.2) node[right] {$\phi(z,s)$};
        \draw[densely dashed] (0.5,1.8) -- (1.2,1.8) node[right] {$F_{3a}(s)$};
        \draw[loosely dashed] (0.5,1.4) -- (1.2,1.4) node[right] {$F_{3b}(s)$};
        \coordinate (z) at (1.5,0);
        \coordinate (mz) at (3.5,0);
        \coordinate (mu) at (0.75,0);
        \draw[blue, thick] (0,0) -- (z);
        \draw[blue, thick] (z) -- (3.5,1);
        \draw[blue, thick] (3.5,1) -- (4.85,1);
        
        \draw[densely dashed] (0,0.01)--(0.75,0.01);
        \draw[densely dashed] (0.75,0.01)--(4.85,1.491);
        
        \draw[loosely dashed] (0,0)--(0.75,1.01);
        \draw[loosely dashed] (0.75,1.01)--(4.85,1.01);

        \draw[dotted] (z) -- (z) node[below] {$z$};
        \draw[dotted] (mz) -- (mz) node[below] {$m+z$};
        \draw[dotted] (0,0) -- (0,0) node[below] {$0$};
        \draw[dotted] (4.85,0) --(4.85,0) node[below] {$b$};
        \draw[dotted] (mu) --(mu) node[below] {$\mu$};

\end{tikzpicture}
\end{minipage}

\caption{Scenario 1 and the majorizing functions $F_{1a}(x)$ and $F_{1b}(x)$ under scenarios 1a and 1b, respectively. Scenario 2 and the majorizing functions $F_{2a}(x)$ and $F_{2b}(x)$ under scenarios 2a and 2b, respectively. Scenario 3 and the majorizing functions $F_{3a}(x)$ and $F_{3b}(x)$ under scenarios 3a and 3b, respectively.}\label{fig:majorsstop}
\end{center}
\end{figure}

Scenario 1a implies that $F(0)=F(m+z)=F(\mu)=F(b)=m$, and hence $\lambda_0=m,\, \lambda_1=\lambda_2=0$ with objective value $m$. It is clear that the optimal primal objective value is also equal to $m$ as the primal solution can only assign probability to values greater than or equal to $m+z$ (which is a consequence of complementary slackness).

Scenario 1b implies $F(0)=0, \, F(m+z)=F(b)=m,$ which gives
\begin{equation*}
    \lambda_0=\frac{m\mu(b-(m+z))}{2(m+z)(b-\mu)}, \quad \lambda_1=\frac{m(b+m+z-2\mu)}{2(m+z)(b-\mu)}, \quad \lambda_2=\frac{m(m+z-b)}{2(m+z)(b-\mu)},
\end{equation*}
and objective value
\begin{equation*}
    \lambda_0+\lambda_1 \mu + \lambda_2 d = \frac{m}{m+z}\left(\mu-\frac{d(b-(m+z))}{2(b-\mu)}\right).
\end{equation*}
Solving the primal problem \eqref{eq:primalstop} with probability masses on the points $\{0,m+z,b\}$ gives
\begin{equation*}
    \int_{s}\phi(z,s)\,{\rm d}\mathbb{P}(s)=m\left(\frac{2\mu+bd/(\mu-b)}{2(m+z)}+\frac{d}{2(b-\mu)}\right)=\frac{m}{m+z}\left(\mu-\frac{d(b-(m+z))}{2(b-\mu)}\right).
\end{equation*}
Since primal and dual feasible solutions have the same objective value we have strong duality and hence found the optimal solutions.

Scenario 2a implies $F(0)=0,\,F(\mu)=\mu-z,\,F(m+z)=m$ which gives
\begin{equation*}
    \lambda_0=-\frac{z}{2}, \quad \lambda_1 = 1-\frac{z}{2\mu}, \quad \lambda_2 = \frac{z}{2\mu},
\end{equation*}
with objective value
\begin{equation*}
    \lambda_0 + \lambda_1 \mu + \lambda_2 d  = z\left(\frac{d}{2\mu}-1\right)+\mu.
\end{equation*}
Solving the optimal probabilities for the primal problem \eqref{eq:primalstop} with masses on $\{0,\mu,m+z\}$ indeed gives
\begin{equation*}
    \int_{s}\phi(z,s)\,{\rm d}\mathbb{P}(s)=(\mu-z)\left(1-\frac{d(m+z)}{2\mu(m+z-\mu)}\right)+m\left(\frac{d}{2(m+z\mu)}\right)= z \left( \frac{d}{2\mu}-1 \right) +\mu.
\end{equation*}

Scenario 2b implies that $F(0)=0, \, F(\mu)=F(m+z)=F(b)=m$, which gives as the dual feasible solution
\begin{equation*}
    \lambda_0=\frac{m}{2},\quad\lambda_1=\frac{m}{2\mu}, \quad \lambda_2 = -\frac{m}{2\mu},
\end{equation*}
and objective value
\begin{equation*}
    \lambda_0 + \lambda_1 \mu + \lambda_2 d  = m \left( 1-\frac{d}{2\mu} \right).
\end{equation*}
Solving for the corresponding optimal probabilities of \eqref{eq:primalstop} confirms that  
\begin{equation*}
    \int_{s}\phi(z,s)\,{\rm d}\mathbb{P}(s)=p_{(m+z)}m + p_b m=m \left(1-\frac{d}{2\mu} \right).
\end{equation*}

Scenario 3a implies $F(0)=0,\,F(\mu)=0,\,F(m+z)=m$ which gives
\begin{equation*}
    \lambda_0=-\frac{m\mu}{2(m+z-\mu)}, \quad \lambda_1 = \frac{m}{2(m+z-\mu)}, \quad \lambda_2 = \frac{m}{2(m+z-\mu)},
\end{equation*}
with objective value
\begin{equation*}
    \lambda_0 + \lambda_1 \mu + \lambda_2 d  = \frac{dm}{2(m+z-\mu)}.
\end{equation*}
Solving the optimal probabilities for the primal problem \eqref{eq:primalstop} with masses on $\{0,\mu,m+z\}$ indeed gives
\begin{equation*}
    \int_{s}\phi(z,s)\,{\rm d}\mathbb{P}(s)=p_{(m+z)}m=\frac{dm}{2(m+z-\mu)}.
\end{equation*}

Scenario 3b implies that $F(0)=0, \, F(\mu)=F(m+z)=F(b)=m$, which again gives as the dual feasible solution
\begin{equation*}
    \lambda_0=\frac{m}{2},\quad\lambda_1=\frac{m}{2\mu}, \quad \lambda_2 = -\frac{m}{2\mu},
\end{equation*}
and objective value
\begin{equation*}
    \lambda_0 + \lambda_1 \mu + \lambda_2 d  = m \left(1-\frac{d}{2\mu} \right).
\end{equation*}
Solving for the corresponding optimal probabilities of \eqref{eq:primalstop} on $\{0,m+z,b\}$ confirms that  
\begin{equation*}
    \int_{s}\phi(z,s)\,{\rm d}\mathbb{P}(s)=p_{(m+z)}m + p_b m=m(1-\frac{d}{2\mu}).
\end{equation*}

The proof of the first part of the theorem is then completed by taking the minimum for each scenario. The second part is an immediate consequence of upper bound~(8) in \cite{postek2018robust}, which is a result that was already shown by \cite{ben1972more}.

\end{proof}

\begin{proof}[Proof of Theorem~\ref{thm: aggregateclaim}]
We consider the following ambiguity set for the distribution of \(S\):
\begin{equation} \label{eq: aggregated ambig2}
\cG = \{ \bP \mid \bP\left(S \in \left[ 0, \; \bar{b} \right] \right) = 1, \enskip \bE_\bP \left[S \right] = \bar{\mu}, \enskip \bE\left[ \left| S - \bar{\mu} \right| \right] \leq \hat{d} \}.
\end{equation}

We can deduce from $\cF\subseteq\cG$ and Theorem \ref{th:stoplossrein} that
\[\sup_{\bP \in \cF} \bE_{\bP}\Big[ \phi\Big(z,\sum_{i=1}^n X_i \Big)\Big] \leq \sup_{\bP \in \cG} \bE_{\bP} [\phi\left(z, S \right)] \hspace{6.28cm} \]
\[\hspace{2.7cm} = \begin{cases} \max_{d \in \left[0, \hat{d}\right]} \Big\{ \min \{m, \: \frac{m}{m+z}(\mu-\frac{d(\bar{b}-(m+z))}{2(\bar{b}-\bar{\mu})} \} \Big\}, & \text{ if } z\leq m+z \leq\bar{\mu},\\
\max_{d \in \left[0, \hat{d}\right]} \Big\{ \min \{m(1-\frac{d}{2\bar{\mu}}), \: z(\frac{d}{2\bar{\mu}}-1)+\bar{\mu}\} \Big\}, & \text{ if } z\leq \bar{\mu} \leq m+z\leq\bar{b},\\
\max_{d \in \left[0, \hat{d}\right]} \Big\{ \min \{m(1-\frac{d}{2\bar{\mu}}), \:\frac{dm}{2(m+z-\bar{\mu})}\} \Big\}, & \text{ if } \bar{\mu}\leq z\leq m+z\leq\bar{b}.\\
\end{cases}\]
We will now solve the maximization problems over \(d\) explicitly. For the instance with $m+z\leq\bar{\mu}$ the problem is easily solved by recognizing that the distribution with probability mass 1 on $\bar{\mu}$ is a member of $\cG$, and therefore the maximum $m$ will be paid almost surely.

For the second case, note that we take the minimum over two linear functions of \(d\), an increasing and a decreasing one. Therefore, the global maximum is at the intersection of these functions, and the optimal \(d\) is thus either \(\hat{d}\) or \(d^*\), where \(d^*\) is such that
\[m(1-\frac{d}{2\bar{\mu}}) = z(\frac{d}{2\bar{\mu}}-1)+\bar{\mu}.\]
Solving this equation yields
\[d^* = \frac{2\bar{\mu}(m+z-\bar{\mu})}{m+z}.\]
Note that \(\hat{d}\) is the optimal solution when
\[z(\frac{d}{2\bar{\mu}}-1)+\bar{\mu}< m(1-\frac{d}{2\bar{\mu}}),\]
as this means that \(\hat{d} < d^*\). Therefore, we find that
\begin{align*}
\max_{d \in \left[0, \hat{d}\right]} \Big\{ \min \{m(1-\frac{d}{2\bar{\mu}}), \: z(\frac{d}{2\bar{\mu}}-1)+\bar{\mu}\} \Big\} & = \min \{ m(1-\frac{d^*}{2\bar{\mu}}), \: z(\frac{\hat{d}}{2\bar{\mu}}-1)+\bar{\mu} \} \nonumber \\ 
& = \min \{ \frac{m\bar{\mu}}{m+z}, \; z(\frac{\hat{d}}{2\bar{\mu}}-1)+\bar{\mu} \}. 
\end{align*} 
Finally, the third case can be solved in the exact same way, resulting in the established theorem.
\end{proof}

\section{Additional results on ambiguous chance constraints} \label{app: acc results}

\begin{proof}[Proof of Theorem \ref{thm: single lhs cc}]
Rewriting \eqref{eq: lhs affine cc} we find
\begin{align*}
\inf_{\bP \in \cP} \bP \left[\left(\bar{\mathbf{a}} + Z \hat{\mathbf{a}} \right)^\top \mathbf{x} \leq h  \right] = \enskip & 1 - \sup_{\bP \in \cP} \bP \left[ \left(\bar{\bm{a}} + Z \hat{\bm{a}} \right)^\top \mathbf{x} \leq h \right] \\
= \enskip & 1 - \sup_{\bP \in \cP} \bP \left[ Z \cdot \hat{\bm{a}}^\top \mathbf{x}  > h - \bar{\bm{a}}^\top \mathbf{x} \right].
\end{align*}
Thus, clearly, \eqref{eq: lhs affine cc} is equivalent to
\begin{equation} \label{eq: lhs affine proof 2} 
\sup_{\bP \in \cP} \bP \left[ Z \cdot \hat{\bm{a}}^\top \mathbf{x} > h - \bar{\bm{a}}^\top \mathbf{x} \right] \leq \epsilon,
\end{equation}
and since \(\epsilon \in \left(0, \frac{1}{2}\right)\), it must hold that \(h - \bar{\bm{a}}^\top \mathbf{x} > 0\), because for any \(h - \bar{\bm{a}}^\top \mathbf{x} \leq 0\), the supremum is at least \(\frac{1}{2}\) by Theorem \ref{theorem:main}. Given that \(h - \bar{\bm{a}}^\top \mathbf{x} > 0\) we thus find
\[\sup_{\bP \in \cP} \bP \left[ Z \cdot \hat{\bm{a}}^\top \mathbf{x}  > h - \bar{\bm{a}}^\top \mathbf{x} \right] = \begin{cases} \min\left\{\frac{d \cdot | \hat{\bm{a}}^\top \mathbf{x} |}{2(h - \bar{\bm{a}}^\top \mathbf{x})}, \; 1 - \frac{d}{2} \right\} & \text{ if } h - \bar{\bm{a}}^\top \mathbf{x} < | \hat{\bm{a}}^\top \mathbf{x}| \\ 0 & \text{ if } h - \bar{\bm{a}}^\top \mathbf{x} \geq | \hat{\bm{a}}^\top \mathbf{x} |. \end{cases}\] 
Since \(1 - \frac{d}{2} \geq \frac{1}{2}\), it must hold that \(\min\left\{\frac{d \cdot | \hat{\bm{a}}^\top \mathbf{x} |}{2(h - \bar{\bm{a}}^\top \mathbf{x})}, \; 1 - \frac{d}{2} \right\} = \frac{d \cdot |\hat{\bm{a}}^\top \mathbf{x} |}{2(h - \bar{\bm{a}}^\top \mathbf{x})}\) for any \(\mathbf{x}\) that satisfies \eqref{eq: lhs affine proof 2}. Rewriting the case where \(h - \bar{\bm{a}}^\top \mathbf{x} < | \hat{\bm{a}}^\top \mathbf{x}|\) we find
\begin{align*}
\frac{d \cdot | \hat{\bm{a}}^\top \mathbf{x} |}{2(h - \bar{\bm{a}}^\top \mathbf{x})} \leq \epsilon \iff & \bar{\bm{a}}^\top \mathbf{x} + \frac{d}{2 \epsilon} \cdot | \hat{\bm{a}}^\top \mathbf{x} | \leq h.
\end{align*}
We can thus combine both cases, such that \eqref{eq: lhs affine proof 2} is equivalent to
\[\bar{\bm{a}}^\top \mathbf{x} + \min \left\{1, \; \frac{d}{2\epsilon}\right\} \cdot |\hat{\bm{a}}^\top \mathbf{x} | \leq h,\]
where it should be noted that this implies \(h - \bar{\bm{a}}^\top \mathbf{x} > 0\), which is therefore redundant. 
\end{proof}

\begin{theorem} \label{thm: convex rhs cc u}
Let \(g : \bR^n \mapsto \bR\), \(h \in \bR\) and let \(Z\) be an \(1\)-dimensional random variable whose distribution lies in the ambiguity set
\[\cP = \left\{\bP : \bP \left[Z \in \left[-1, b\right]\right] = 1, \enskip \bE \left[Z\right] = 0, \enskip \bE \left[ \left| Z \right| \right] = d \right\},\]
for some \(d \in \left[0, \frac{2u}{1+b} \right]\).
For any \(\epsilon \in \left(0, \frac{1}{1 + b}\right)\) and \(\mathbf{x} \in \bR^n\) it holds that
\begin{equation*} 
\inf_{\bP \in \cP} \bP \left[g(\mathbf{x}) + Z \leq 0 \right] \geq 1 - \epsilon,
\end{equation*}
if and only if
\begin{equation*}
g(\mathbf{x}) + \min \left\{b, \frac{d}{2 \epsilon} \right\} \leq 0.
\end{equation*}
\end{theorem}

\begin{theorem} \label{thm: joint indep cc}
Let \(g_i : \bR^n \mapsto \bR\) for \(i = 1, \ldots, m\) and let \(\bm{Z}\) be an \(m\)-dimensional random variable whose distribution lies in the ambiguity set
\[\cP = \left\{\bP : \bP \left[Z_i \in \left[-1, 1\right]\right] = 1, \enskip \bE \left[Z_i\right] = 0, \enskip \bE \left[ \left| Z_i \right| \right] = d_i \enskip \forall i, \quad Z_i \bot Z_j \enskip \forall i \neq j \right\},\]
for some \(\mathbf{d} \in \bR^m\) such that \(d_i \in \left[0, 1\right]\) for all \(i\). Let \(\epsilon \in \left(0, \frac{1}{2} \right)\) and \(\mathcal{I}\) be the set of indices \(i\) such that \(\frac{d_i}{2 \epsilon} \leq 1\).
For any \(\mathbf{y} \in \bR^n\) it holds that
\begin{equation} \label{eq: joint indep cc}
\inf_{\bP \in \cP} \bP \left[g_i(\mathbf{y}) + Z_i \leq 0 \quad \forall i\right] \geq 1 - \epsilon,
\end{equation}
if
\begin{subequations} \label{eq: joint indep cc final}
\begin{empheq}[left={\empheqlbrace\,}]{align}
\sum_{i \in \mathcal{I}} \log \left[1 - \frac{d_i}{- 2 g_i(\mathbf{x})}\right] & \geq \log \left[ 1 - \epsilon \right]  \\
g_i(\mathbf{x}) + 1 & \leq 0 & i \not\in \mathcal{I} \\
g_i(\mathbf{x}) & < 0 & i \in \mathcal{I}&,
\end{empheq}
\end{subequations}
which is a convex set of constraints if all \(g_i\) are convex functions.
\end{theorem}
\begin{proof}
Using the pairwise independence of \(\bm{Z}\) we find
\begin{align}
\inf_{\bP \in \cP} \bP \left[g_i(\mathbf{x}) + Z_i \leq 0 \quad \forall i\right] = \enskip & \inf_{\bP \in \cP} \prod_{i=1}^m \bP \left[g_i(\mathbf{x}) + Z_i \leq 0 \right] \nonumber \\
= \enskip & \prod_{i=1}^m \inf_{\bP \in \cP} \bP \left[g_i(\mathbf{x}) + Z_i \leq 0 \right] \nonumber \\
= \enskip & \prod_{i=1}^m \left[1 - \sup_{\bP \in \cP} \bP \left[g_i(\mathbf{x}) + Z_i > 0 \right]\right] \nonumber \\
= \enskip & \prod_{i=1}^m \left[1 - \sup_{\bP \in \cP} \bP \left[Z_i > - g_i(\mathbf{x})  \right]\right]. \label{eq: joint indep cc proof 1}
\end{align}
From this, it readily follows that it must at least hold that \(\sup_{\bP \in \cP} \bP \left[Z_i >  - g_i(\mathbf{x})  \right] \leq \epsilon\) for all \(i\). From Theorem \ref{theorem:main} we consequently know that it must hold that \( - g_i(\mathbf{x}) > \bE \left[ Z_i \right] = 0 \) for all \(i\), as we know that \(\epsilon < \frac{1}{2}\).
Given that \(- g_i(\mathbf{x}) > 0\), we know that
\begin{equation} \label{eq: joint indep cc proof 2}
\sup_{\bP \in \cP} \bP \left[ Z_i > - g_i(\mathbf{x}) \right] = \begin{cases} \min \left\{\frac{d_i}{ - 2 g_i(\mathbf{x}) }, \: 1 - \frac{d_i}{2}\right\} & \text{ if } - g_i(\mathbf{x}) < 1 \\
0 & \text{ if } - g_i(\mathbf{x}) \geq 1. \end{cases}
\end{equation}
Since \( 1 - \frac{d_i}{2} \geq \frac{1}{2}\) it follows from \(\epsilon < \frac{1}{2}\) that it must hold for any feasible solution \(\mathbf{x}\) that
\[ \min \left\{\frac{d_i}{ - 2 g_i(\mathbf{x})}, \: 1 - \frac{d_i}{2} \right\} =  \frac{d_i}{ - 2 g_i(\mathbf{x})} \qquad \qquad i=1,\ldots,m. \]
Moreover, we note that \(\frac{d_i}{- 2 g_i(\mathbf{x})} \leq \epsilon\) is equivalent to \( - g_i(\mathbf{x}) \geq \frac{d_i}{2 \epsilon}\) and thus if \(\frac{d_i}{2 \epsilon} \geq 1\), imposing this is overly restrictive, as we know from \eqref{eq: joint indep cc proof 2} that the worst-case probability of violation is 0, not \(\epsilon\) in that situation. For all such \(i\), we thus simply require that \(g_i(\mathbf{x}) + 1 \leq 0\), such that
\[\sup_{\bP \in \cP} \bP \left[ Z_i > - g_i(\mathbf{x}) \right] = 0 \qquad \qquad \forall i \not\in \mathcal{I}.\]

Given this analysis, we find that \eqref{eq: joint indep cc proof 1} is equal to
\[\prod_{i \in \mathcal{I}} \left[ 1 - \sup_{\bP \in \cP} \bP \left[ Z_i > - g_i(\mathbf{x}) \right] \right] \leq \prod_{i \in \mathcal{I}} \left[ 1 - \frac{d_i}{ - 2 g_i(\mathbf{x})} \right],\]
and thus \eqref{eq: joint indep cc} holds if
\[ \begin{cases}
\sum_{i \in \mathcal{I}} \log \left[ 1 - \frac{d_i}{ - 2 g_i(\mathbf{x})} \right] \geq \log\left[1 - \epsilon\right] \\
g_i(\mathbf{x}) + 1 \leq 0 & i \not\in \mathcal{I} \\
g_i(\mathbf{x}) < 0 & i \in \mathcal{I}.
\end{cases} \]

\end{proof}

\begin{theorem} \label{thm: joint equal cc}
Let \(g_i : \bR^n \to \bR\) be convex for \(i=1,\ldots,m\) and let \(Z\) be a 1-dimensional random variable whose distribution lies in the ambiguity set
\[\cP = \left\{\bP : \bP \left[Z \in \left[-1, 1\right]\right] = 1, \enskip \bE \left[Z\right] = 0, \enskip \bE \left[ \left| Z \right| \right] = d \right\},\]
for some \(d \in \left[0, 1\right]\).
For any \(\epsilon \in \left(0, \frac{1}{2}\right)\) and \(\mathbf{x} \in \bR^n\) it holds that
\begin{equation} \label{eq: joint equal cc}
\inf_{\bP \in \cP} \bP \left[g_i(\mathbf{x}) + Z \leq 0 \quad \forall i\right] \geq 1 - \epsilon,
\end{equation}
if and only if
\begin{equation} \label{eq: joint equal cc final}
\max_i \left\{ g_i(\mathbf{x}) \right\} + \min \left\{1, \frac{d}{2 \epsilon} \right\} \leq 0.
\end{equation}
\end{theorem}
\begin{proof}
Using the fact that every constraint features the same uncertain parameter \(Z\) we find
\begin{align*}
\inf_{\bP \in \cP} \bP \left[ g_i(\mathbf{x}) + Z \leq 0 \quad \forall i \right] = \enskip & \inf_{\bP \in \cP} \bP \left[Z \leq - g_i(\mathbf{x}) \quad \forall i \right] \\
= \enskip & \inf_{\bP \in \cP} \bP \left[Z \leq \min_i \left\{ - g_i(\mathbf{x}) \right\} \right] \\
= \enskip & 1 - \sup_{\bP \in \cP} \bP \left[ Z > \min_i \left\{ - g_i(\mathbf{x}) \right\} \right].
\end{align*}
We thus know that \eqref{eq: joint equal cc} is equivalent to
\begin{equation} \label{eq: joint equal cc intermed}
\sup_{\bP \in \cP} \bP \left[ Z > \min_i \left\{- g_i(\mathbf{x}) \right\} \right] \leq \epsilon,
\end{equation}
to which we can apply Theorem~\ref{theorem:main} with \(t = \min_i \left\{ - g_i(\mathbf{x}) \right\}\). Since we know \(\epsilon < \frac{1}{2}\), we once again find that it must hold that \(\min_i \left\{ - g_i(\mathbf{x}) \right\} > 0\) and thus by the same reasoning as in the proof of Theorem~\ref{thm: joint indep cc} we find that it must hold that
\[\sup_{\bP \in \cP} \bP \left[ Z > \min_i \left\{- g_i(\mathbf{x}) \right\} \right] = \begin{cases} \min \left\{\frac{d}{2 \min_i \left\{- g_i(\mathbf{x}) \right\}}, \; 1 - \frac{d}{2} \right\} & \text{ if } \min_i \{- g_i(\mathbf{x}) \} < 1 \\ 0 & \text{ if } \min_i \{- g_i(\mathbf{x}) \} \geq 1. \end{cases} \]
Once again, it must hold that \( \min \left\{\frac{d}{2 \min_i \left\{- g_i(\mathbf{x}) \right\}}, \; 1 - \frac{d}{2} \right\} = \frac{d}{2 \min_i \left\{- g_i(\mathbf{x}) \right\}}\) and subsequently both cases can be combined. We thus find that \eqref{eq: joint equal cc} is equivalent to
\[\min_i \left\{- g_i(\mathbf{x}) \right\} \geq \min \left\{1, \; \frac{d}{2 \epsilon} \right\},\]
which makes the condition \(\min_i \left\{- g_i(\mathbf{x}) \right\} > 0\) redundant. This constraint is equivalent to
\[\max_i \left\{ g_i(\mathbf{x}) \right\} + \min \left\{1, \frac{d}{2 \epsilon} \right\} \leq 0,\]
which concludes the proof.
\end{proof}

\end{document}

%% file: Basics.tex
\usepackage[english]{babel}

\usepackage{amssymb}
\usepackage{array,amsmath}
\usepackage{amsthm}

\usepackage{mathtools}

\usepackage{dsfont}

\usepackage{empheq}

\newcommand{\bP}{\mathbb{P}}
\newcommand{\cP}{\mathcal{P}}
\newcommand{\bE}{\mathbb{E}}

\newcommand{\cF}{\mathcal{F}}
\newcommand{\cG}{\mathcal{G}}

\newcommand{\bR}{\mathbb{R}}

\newcommand{\1}{\mathds{1}}

\usepackage{bm}

\newtheorem{theorem}{Theorem}

\newtheorem{proposition}{Proposition}

\theoremstyle{remark}

\usepackage{graphicx}
\usepackage{tikz}
\usetikzlibrary{arrows.meta}
\usetikzlibrary{bending}
\usepackage{pgfplots}
\pgfplotsset{compat=1.15}

\usepackage{subcaption}

\usepackage{natbib}

\usepackage{multirow}
\usepackage{booktabs}
\usepackage{longtable}
\makeatletter
\newbox\LT@firstfoot
\def\endfirstfoot{\LT@end@hd@ft\LT@firstfoot}
\newdimen\LT@footdiff
\def\LT@start{%
  \let\LT@start\endgraf
  \endgraf\penalty\z@
  \vskip\LTpre
  \LT@footdiff-\ht\LT@foot
  \advance\LT@footdiff\ht\LT@firstfoot
  \dimen@\pagetotal
  \advance\dimen@ \ht\ifvoid\LT@firsthead\LT@head\else\LT@firsthead\fi
  \advance\dimen@ \dp\ifvoid\LT@firsthead\LT@head\else\LT@firsthead\fi
  \advance\dimen@ \ht\ifvoid\LT@firstfoot\LT@foot\else\LT@firstfoot\fi
  \dimen@ii\vfuzz
  \vfuzz\maxdimen
  \setbox\tw@\copy\z@
  \setbox\tw@\vsplit\tw@ to \ht\@arstrutbox
  \setbox\tw@\vbox{\unvbox\tw@}%
  \vfuzz\dimen@ii
  \advance\dimen@ \ht
      \ifdim\ht\@arstrutbox>\ht\tw@\@arstrutbox\else\tw@\fi
  \advance\dimen@\dp
      \ifdim\dp\@arstrutbox>\dp\tw@\@arstrutbox\else\tw@\fi
  \advance\dimen@ -\pagegoal
  \ifdim \dimen@>\z@\vfil\break\fi
  \global\@colroom\@colht
  \ifvoid\LT@firstfoot
    \ifvoid\LT@foot
    \else
      \advance\vsize-\ht\LT@foot
      \global\advance\@colroom-\ht\LT@foot
      \dimen@\pagegoal\advance\dimen@-\ht\LT@foot\pagegoal\dimen@
      \maxdepth\z@
    \fi
  \else
    \advance\vsize-\ht\LT@firstfoot
    \global\advance\@colroom-\ht\LT@firstfoot
    \dimen@\pagegoal\advance\dimen@-\ht\LT@firstfoot\pagegoal\dimen@
    \maxdepth\z@
  \fi
  \ifvoid\LT@firsthead\copy\LT@head\else\box\LT@firsthead\fi\nobreak
  \output{\LT@output}%
}
\def\LT@output{%
  \ifnum\outputpenalty <-\@Mi
    \ifnum\outputpenalty > -\LT@end@pen
      \LT@err{floats and marginpars not allowed in a longtable}\@ehc
    \else
      \setbox\z@\vbox{\unvbox\@cclv}%
      \ifdim \ht\LT@lastfoot>\ht\LT@foot
        \dimen@\pagegoal
        \advance\dimen@-\ht\LT@lastfoot
        \ifdim\dimen@<\ht\z@
          \setbox\@cclv\vbox{\unvbox\z@\copy\LT@foot\vss}%
          \@makecol
          \@outputpage
          \setbox\z@\vbox{\box\LT@head}%
        \fi
      \fi  
      \global\@colroom\@colht
      \global\vsize\@colht   
      \vbox
        {\unvbox\z@\box\ifvoid\LT@lastfoot\LT@foot\else\LT@lastfoot\fi}%
    \fi
  \else
    \ifvoid\LT@firstfoot
      \setbox\@cclv\vbox{\unvbox\@cclv\copy\LT@foot\vss}%
      \@makecol
      \@outputpage
      \global\vsize\@colroom
    \else
      \setbox\@cclv\vbox{\unvbox\@cclv\box\LT@firstfoot\vss}%
      \@makecol
      \@outputpage
      \global\advance\@colroom\LT@footdiff
      \global\vsize\@colroom
    \fi
    \copy\LT@head\nobreak
  \fi
}
\makeatother

\usepackage{lipsum}

\usepackage[shortlabels]{enumitem}

\usepackage{todonotes}

\usepackage{algorithm}
\usepackage{algpseudocode}
\usepackage{float}
\newfloat{algorithm}{t}{lop}

\usepackage[top=25mm, bottom=25mm, left=25mm, right=25mm]{geometry}

\newcommand{\relmiddle}[1]{\mathrel{}\middle#1\mathrel{}}

\newcommand \Template[2]{
\usepackage{fancyhdr}
\pagestyle{fancy}
\fancyhf{}
\fancyhead[RO]{#1}
\fancyhead[LO]{#2}
\fancyfoot[CO]{\thepage}

\renewcommand{\headrulewidth}{1pt}
\renewcommand{\footrulewidth}{1pt}}

\usepackage{titling}

\newcommand \linedabstractkw[2]{
  \renewcommand\maketitlehookd{%
    \mbox{}\medskip\par
    \centering
    \hrule\medskip
    \begin{minipage}{0.9\textwidth}
    \footnotesize \textbf{Abstract}\\ \footnotesize #1\\
    
	\small \textit{Keywords: } \small #2
    \end{minipage}\medskip\hrule\medskip
    }      
}

\newcommand \ACKNOWLEDGMENT[1]{
	\section*{Acknowledgments}
	\small #1
}